\documentclass[a4]{amsart}

\usepackage{amsthm,amssymb,amsmath}
\usepackage{color}
\RequirePackage[numbers]{natbib}
\RequirePackage[colorlinks,citecolor=blue,urlcolor=blue,linkcolor=blue]{hyperref}
\RequirePackage[OT1]{fontenc}

\usepackage[toc,page]{appendix}
\usepackage[utf8]{inputenc}
\usepackage{float}
\usepackage[shortlabels]{enumitem}

\usepackage[disable,textwidth=4cm, textsize=footnotesize]{todonotes}
\numberwithin{equation}{section}
\usepackage{xargs,hyperref,aliascnt,bbm}
\usepackage{ifthen}

\newtheorem{theorem}{Theorem}
\newaliascnt{proposition}{theorem}
\newtheorem{proposition}[proposition]{Proposition}
\aliascntresetthe{proposition}

\newaliascnt{lemma}{theorem}
\newtheorem{lemma}[lemma]{Lemma}
\aliascntresetthe{lemma}

\newaliascnt{corollary}{theorem}

\aliascntresetthe{corollary}

\theoremstyle{definition}
\newtheorem{example}{Example}
\newtheorem{remark}{Remark}
\newaliascnt{definition}{theorem}
\newtheorem{definition}[definition]{Definition}
\aliascntresetthe{definition}

\newcounter{hypA'}

\newenvironment{hyp}[1]{
\begin{enumerate}[label=(\textbf{\sf #1}-\arabic*),resume=hyp#1]\begin{sf}}
{\end{sf}\end{enumerate}}

\def\rme{\mathrm{e}}
\def\Cset{\mathsf{C}}

\def\Xset{\mathsf{X}}
\def\Yset{\mathsf{Y}}

\def\Zset{\mathsf{Z}}

\def\Xsigma{\mathcal{X}}
\def\Ysigma{\mathcal{Y}}
\def\Zsigma{\mathcal{Z}}

\def\met{\Delta}
\def\rset{\ensuremath{\mathbb{R}}}
\def\zset{\ensuremath{\mathbb{Z}}}
\def\cset{\ensuremath{\mathbb{C}}}

\def\mcf{\mathcal{F}}

\def\eqsp{}	%\def\eqsp{\;}
\def\PE{\mathbb{E}}

\def\PP{\mathbb{P}}

\def\rmd{\mathrm{d}}

\newcommandx\LogInt[5][1=\theta,4=,5=Y]{\upsilon_{#4}^{#1}\langle {#5}_{#2:#3} \rangle}

\newcommand{\as}{\mbox{-a.s.}}

\newcommandx{\aslim}[1]{\ensuremath{\stackrel{#1-\text{a.s.}}{\longrightarrow}}}

\newcommandx\sequence[3][2=n,3=\zset]{\ensuremath{\{ #1_{#2}\,:\, #2 \in #3\}}}
\newcommand\nsequence[2]{\ensuremath{\{ #1\,:\, #2\}}}
\newcommandx\dsequence[4][3=n,4=\zset]{\ensuremath{\{ (#1_{#3},#2_{#3})\,:\,#3 \in #4\}}}

\newcommand{\CPE}[3][]
{\ifthenelse{\equal{#1}{}}%
{\mathbb{E}\left[\left. #2 \, \right| #3 \right]}
{\mathbb{E}_{#1}\left[\left. #2 \, \right| #3 \right]}
}
\newcommand{\CPEu}[4][]
{\ifthenelse{\equal{#1}{}}%
{\mathbb{E}\left[\left. #2 \, \right| #3 \right]}
{\mathbb{E}^{#1}_{#2}\left[\left. #3 \, \right| #4 \right]}
}

\newcommand{\CPEv}[3][]
{\ifthenelse{\equal{#1}{}}%
{\mathbb{E}_\star\left[\left. #2 \, \right| #3 \right]}
{\mathbb{E}_\star_{#1}\left[\left. #2 \, \right| #3 \right]}
}

\newcommand{\CPP}[3][]
{\ifthenelse{\equal{#1}{}}%
{\mathbb{P}\left[\left. #2 \, \right| #3 \right]}
{\mathbb{P}_{#1}\left[\left. #2 \, \right| #3 \right]}
}

\newcommand{\CPPu}[3][]
{\ifthenelse{\equal{#1}{}}%
{\mathbb{P}\left[\left. #2 \, \right| #3 \right]}
{\mathbb{P}^{#1}\left[\left. #2 \, \right| #3 \right]}
}

\newcommand{\tCPPu}[3][]
{\ifthenelse{\equal{#1}{}}%
{\tilde{\mathbb{P}}\left[\left. #2 \, \right| #3 \right]}
{\tilde{\mathbb{P}}^{#1}\left[\left. #2 \, \right| #3 \right]}
}

\newcommandx{\chunk}[3]%
{\ensuremath{#1}_{#2:#3}}

\def\1{\mathbbm{1}}

\newcommandx\proj[2][1=,2=]{
\ifthenelse{\equal{#1}{}}
{\operatorname{X}}
{\operatorname{X}_{#1:#2}}i
}

\newcommand{\eqdef}{:=}
\newcommand{\set}[2]{\left\{#1:#2\right\}}
\newcommand{\setvect}[2]{\left(#1\right)_{#2}}

\newcommandx\lkdM[3][1=,3=]{
\ifthenelse{\equal{#2}{}}
{ \mathsf{L}_{#1}^{#3}}
{ \mathsf{L}_{#1}^{#3}\langle #2\rangle}
}
\newcommandx\lkdMStat[3][1=,3=]{
\ifthenelse{\equal{#2}{}}
{ \bar{\mathsf{L}}_{#1}^{#3}}
{ \bar{\mathsf{L}}_{#1}^{#3}\langle #2 \rangle}
}

\newcommandx\lkd[3][1=,3=]{
\ifthenelse{\equal{#2}{}}
{ \ell_{#1}^{#3}}
{ \ell_{#1}^{#3}\langle #2\rangle}
}
\newcommandx\lkdStat[3][1=,3=]{
\ifthenelse{\equal{#2}{}}
{ \bar \ell_{#1}^{#3}}
{ \bar \ell_{#1}^{#3}\langle #2 \rangle}
}

\def\initmle{z^{(\text{\tiny{i}})}}

\def\initmlex{x^{(\text{\tiny{i}})}}
\def\initmlexi{\xi^{(\text{\tiny{i}})}}
\def\initmley{y^{(\text{\tiny{i}})}}
\def\initmleu{u^{(\text{\tiny{i}})}}
\newcommandx\phimomentinit[1][1=\theta]{\phi^{#1}}

\def\piX{\pi_{\Xset}}
\def\piY{\pi_{\Yset}}

\newcommand{\mlY}[1]{\hat\theta_{#1}}
\newcommand{\argmax}{\mathop{\mathrm{argmax}}}

\newcommandx{\normLip}[2][1=]{\mathrm{Lip}(#2;#1)}
\newcommandx{\wass}[2][1]{\lVert #2\rVert_{#1}}

\newcommandx{\wasser}[3][1=]{\mathcal{W}_{#1}\left(#2,#3\right)}
\newcommandx{\proho}[3][1=]{\mathcal{P}_{#1}\left(#2,#3\right)}

\newcommandx{\dobru}[2][1=]{\dobrush_{#1}\left( #2\right)}
\newcommand{\dobrush}{\Delta}

\newcommand{\Uset}{\mathrm U}
\newcommand{\Usigma}{\mathcal U}

\newcommandx{\Pcan}[2][1=,2=]{\mathbb{P}_{#1}^{#2}}
\newcommandx{\Ecan}[2][1=,2=]{\mathbb{E}_{#1}^{#2}}
\newcommand{\Xinit}{\xi}

\newcommandx\cesp[4][1=,2=]{\ensuremath{{\mathbb E}_{#1}^{#2}\left[ \left. #3 \right| #4 \right]}}

\newcommandx{\f}[3][1=\theta, 3=]{\psi^{#1}_{#3}\langle #2 \rangle}
\newcommandx{\tf}[2][1=\theta]{\tilde{\psi}^{#1}\langle #2 \rangle}
\newcommandx{\F}[2][1=\theta]{\Psi^{#1}\langle #2 \rangle}
\newcommandx{\tF}[2][1=\theta]{\tilde{\Psi}^{#1}\langle #2 \rangle}
\newcommandx{\flp}[2][1=\theta]{\widehat{\psi}^{#1}\langle #2 \rangle}
\newcommandx{\FLP}[2][1=\theta]{\widehat{\Psi}^{#1}\langle #2 \rangle}
\newcommand{\thv}{{\theta_\star}}
\newcommandx{\kap}[3][1=\theta]{
\ifthenelse{\equal{#3}{}}
{\kappa^{#1}\langle #2\rangle }
{\kappa^{#1}\langle #2\rangle (#3)}
}

\def\lnp{\ln^{+}}

\def\zsetp{\zset_{\geq 0}}
\def\zsetn{\zset_{\leq 0}}
\def\zsetpnz{\zset_{>0}}

\def\rsetp{\rset_{\geq0}}

\def\rsetpnz{\rset_{>0}}

\def\VX{V_{\Xset}}
\def\VY{V_{\Yset}}
\def\VU{V_{\Uset}}
\def\WX{W_{\Xset}}
\def\WY{W_{\Yset}}
\def\WU{W_{\Uset}}
\newcommandx{\probdoeblin}[3][1=]{\mu^{#1}_{#2}\langle #3 \rangle}
\newcommand{\Pblock}[2][]
{\ifthenelse{\equal{#1}{}}{\boldsymbol{\operatorname{L}}\langle#2\rangle}{\boldsymbol{\operatorname{L}}^{#1}\langle#2\rangle}
}

\newcommand{\ConPblock}[3][]
{\ifthenelse{\equal{#1}{}}{\boldsymbol{\operatorname{L}}\langle#2|#3\rangle}{\boldsymbol{\operatorname{L}}^{#1}\langle#2|#3\rangle}
}

\newcommand{\pblock}[2][]
{\ifthenelse{\equal{#1}{}}{\mathbf{\ell}\langle#2\rangle}{\mathbf{\ell}^{#1}\langle #2\rangle}
}

\newcommand{\Sset}{\mathsf{S}}
\newcommand{\Ssigma}{\mathcal{S}}

\newcommand{\Vset}{\mathsf{V}}

\newcommandx{\limlike}[4][1=\theta, 2=\theta_\star]{p^{#1,#2}\left( #3 | #4 \right)}

\def\Xmet{\boldsymbol{\delta}_\Xset}

\def\Zmet{\boldsymbol{\delta}_\Zset}

\def\Umet{\boldsymbol{\delta}_\Uset}
\newcommand{\Projarg}[2]{\Pi_{#1}\left(#2\right)}
\newcommand{\Proj}[1]{\Pi_{#1}}

\newcommandx{\vnorm}[2][1=V]{\left|#2\right|_{#1}}

\newcommand{\emptyseq}{\emptyset}

\newcommand{\underg}{\underline{g}}
\newcommand{\underG}{\underline{G}}
\newcommand{\underh}{\underline{\ell}}

\newcommand{\X}{{\mathrm{L}}}
\newcommand{\func}{{\mathbf{f}}}
\newcommand{\lr}[1]{\left(#1\right)}
\newcommand{\lrb}[1]{\left[#1\right]}
\newcommand{\tvdist}[2]{\mathrm{d}_{\mathrm{TV}}(#1,#2)}

\begin{document}

\title[General-order observation-driven models]{General-order observation-driven models: ergodicity and consistency of
  the maximum likelihood estimator}

\author{Tepmony Sim},
\author{Randal Douc}
\author{Fran\c{c}ois Roueff}

\address{MIT Research Unit \\
Institute of Technology of Cambodia \\
12156 Phnom Penh \\
 Cambodia
}
\email{tepmony.sim@itc.edu.kh}

 \address{D\'epartement  CITI \\
 CNRS UMR 5157 \\
 T\'el\'ecom SudParis \\
 Institut Polytechnique de Paris \\
 91000 \'Evry }
\email{randal.douc@telecom-sudparis.eu}

\address{LTCI,
 T\'el\'ecom Paris\\
 Institut Polytechnique de Paris \\
 19 place Marguerite Perey,\\
 91120 Palaiseau \\
 France }
\email{roueff@telecom-paristech.fr}

\date{\today}

%\runauthor{R. Douc, F. Roueff and T. Sim}

\begin{abstract}
  \sloppy The class of observation-driven models (ODMs) includes
  many models of non-linear time series which, in a fashion similar to, yet
  different from, hidden Markov models (HMMs), involve hidden
  variables. Interestingly, in contrast to most HMMs,
  ODMs enjoy likelihoods that can be computed exactly with
  computational complexity of the same order as the number of observations,
  making maximum likelihood estimation the privileged approach for statistical
  inference for these models. A celebrated example of general order
  ODMs is the GARCH$(p,q)$ model, for which ergodicity and
  inference has been studied extensively.  However little is known on more
  general models, in particular integer-valued ones, such as the log-linear
  Poisson GARCH or the NBIN-GARCH of order $(p,q)$ about which most of the
  existing results seem restricted to the case $p=q=1$. Here we fill this gap
  and derive ergodicity conditions for general ODMs. The
  consistency and the asymptotic normality of the maximum likelihood estimator
  (MLE) can then be derived using the method already developed for
  first order ODMs.
\end{abstract}


\maketitle

\sloppy

\section{Introduction}
\sloppy
Since they were introduced in \cite{cox:1981}, observation-driven models have been receiving
renewed interest in recent years. These models are widely applied in various
fields ranging from economics (see \cite{pindyck1998econometric}),
environmental study (see \cite{bhaskaran2013time}), epidemiology and public
health study (see \cite{zeger1988regression, davis1999modeling,
  ferland:latour:oraichi:2006}), finance (see \cite{liesenfeld2003univariate,
  rydberg2003dynamics,fokianos:tjostheim:2011,francq2011garch}) and population
dynamics (see \cite{ives2003estimating}). The celebrated
GARCH$(1,1)$ model introduced in \cite{bollerslev:1986} as well as
most of the models derived from this one are typical examples of
ODMs; see \cite{bollerslev08-glossary} for a list of some of
them. A list of contributions on this class of models specifically dealing with
discrete data includes \cite{streett:2000, davis:dunsmuir:streett:2003,
  heinen2003modelling, ferland:latour:oraichi:2006,
  fokianos:rahbek:tjostheim:2009, franke2010weak, fokianos:tjostheim:2011,
  henderson:matteson:woodard:2011, neuman:2011, davis:liu:2012,
  doukhan:fokianos:tjostheim:2012, dou:kou:mou:2013, fokianos2013goodness,
  christou2014quasi, christou2015count,cui-zhu-2018} and \cite{douc2015handy}.

ODMs have the nice feature that the computations of the associated
(conditional) likelihood and its derivatives are easy, the parameter estimation
is hence relatively simple, and the prediction, which is a prime objective in
many time series applications, is straightforward. However, it turns out that
the asymptotic properties of the maximum likelihood estimator (MLE) for this
class can be cumbersome to establish, except when they can be derived using
computations specific to the studied model (the GARCH$(1,1)$ case
being one of the most celebrated example). The literature concerning the
asymptotic theory of the MLE when the observed variable has Poisson
distribution includes \cite{fokianos:rahbek:tjostheim:2009,
  fokianos:tjostheim:2011, fokianos2012nonlinear} and \cite{wang2014self}. For
a more general case where the model belongs to the class of one-parameter
exponential ODMs, such as the Bernoulli, the exponential, the
negative binomial (with known shape parameter) and the Poisson
autoregressive models, the consistency and the asymptotic normality of the MLE
have been derived in \cite{davis:liu:2012}. However, the one-parameter
exponential family is inadequate to deal with models such as multi-parametric,
mixture or multivariate ODMs (the negative binomial with all
unknown parameters and mixture Poisson ODMs are examples of this
case). A more general consistency result, has been obtained recently in
\cite{dou:kou:mou:2013}, where the observed process may admit various
conditional distributions. This result has later been extended and refined in
\cite{douc2015handy}. However, most of the results obtained so far have been
derived only under the framework of GARCH$(1,1)$-type or first-order
ODMs. Yet, up to our knowledge, little is known for the
GARCH$(p,q)$-type, i.e. larger order discrete ODMs, as
highlighted as a remaining unsolved problem in
\cite{tjostheim2015count}.  

Here, following others (e.g. \cite{streett:2000,
  heinen2003modelling}), we consider a general class of ODMs that is
capable to account for several lagged variables of both hidden and
observation processes. We develop a theory for the class of
general-order ODMs parallel to the GARCH$(p,q)$ family and in particular we investigate the two  problems listed here below.
\begin{enumerate}[label=\alph*)]
\item\label{item:1} Provide a complete set of conditions for general
order ODMs implying that the process is ergodic. 
\item\label{item:2} Prove the consistency of the MLE. Under the
  assumption of well-specified models, this   can be
treated in two  separate sub-problems.
\begin{enumerate}[label=\arabic*)]
\item\label{item:3} Prove that the  the MLE is equivalence-class consistency. 
\item\label{item:4} Characterize the set of identifiable parameters.
\end{enumerate}
\end{enumerate}
In principle, the general order model can be treated by embedding it
into a first-order one and by applying the results obtained e.g. in
\cite{dou:kou:mou:2013,douc2015handy} to the embedded model. Yet the
particular form of the embedded model does not fit the usual
assumptions tailored for standard first-order ODMs.
This is why, as pointed out in \cite{tjostheim2015count},
solving Problem~\ref{item:1} is not only a formal extension of
available results. To obtain \autoref{thm:ergodicity:gen} where
Problem~\ref{item:1} is addressed, we derive conditions by taking advantage of the asymptotic behavior of
iterated versions of the kernels involved. Incidentally, this also
allows us to improve known conditions for some first-order models, as
explained in
\autoref{rem:log:poi:ergo:cond}-\ref{item:rem:log:poi:ergo:peq1}. 
Once the ergodicity of the model is proved, we can solve Sub-problem~\ref{item:2}~\ref{item:3}.
This is done in \autoref{thm:convergence-main:gen:od}, which is
obtained almost for free from the embedding in the ODM(1,1) case. 
Sub-problem~\ref{item:2}~\ref{item:4} is much more involved and has been addressed
in \cite{douc-roueff-sim_ident-genod2020}.

To demonstrate the generality, applicability and efficiency of
our approach, we apply our results to three specific integer valued
ODMs, namely, the log-linear Poisson GARCH$(p,q)$ model, the negative
binomial integer-valued GARCH or the NBIN-GARCH$(p,q)$ model and the
Poisson AutoRegressive model with eXogenous variables, known as the
PARX model. To the best of our knowledge, the stationarity and
ergodicity as well as the asymptotic properties of the MLE for the
general log-linear Poisson GARCH$(p,q)$ and NBIN-GARCH$(p,q)$ models
have not been derived so far. For the PARX$(p,q)$ model, which can be
considered as a vector-valued ODM in our study, such results are
available in \cite{AGOSTO2016640} but our approach leads to
significantly different assumptions, as will be shown in
\autoref{sec:parx-model}. Numerical experiments involving the
log-linear Poisson GARCH$(p,q)$ and the NBIN-GARCH $(p, q)$ models can be found
in \cite[Section~5.5]{sim-tel-01458087} for a set of earthquake data
from Earthquake Hazards Program \cite{earthquake-data}.

The {paper} is structured as follows. Definitions used throughout the paper are
introduced in \autoref{sec:definitions-notation-od-gen}, where we also state
our results on three specific examples. In \autoref{sec:main-result:gen:od}, we
present our main results on the ergodicity and consistency of
the MLE for general order ODMs. Finally,
\autoref{sec:append:proofs:gen:od} contains the postponed proofs and
we gather some independent useful lemmas in \autoref{sec:useful-lemmas}.

\section{Definitions and Examples}
\label{sec:definitions-notation-od-gen}

\subsection{Observation-driven model of order $(p,q)$}
Throughout the paper we use the notation
$\chunk{u}{\ell}{m}\eqdef(u_\ell,\ldots,u_m)$ for $\ell \leq m$, with the
convention that $\chunk{u}{\ell}{m}$ is the empty sequence if $\ell>m$, so
that, for instance $(\chunk x0{(-1)},y)=y$.  The observation-driven time series
model can formally be defined as follows.
\begin{definition}[General order ODM and (V)LODM]\label{def:obs-driv-gen}
  Let $(\Xset,\Xsigma)$, $(\Yset,\Ysigma)$ and $(\Uset,\Usigma)$ be measurable
  spaces, respectively called the \emph{latent space}, the \emph{observation
    space} and the \emph{reduced observation space}. Let $(\Theta,\met)$ be a
  compact metric space, called the \emph{parameter space}.  Let $\Upsilon$ be a
  measurable function from $(\Yset,\Ysigma)$ to $(\Uset,\Usigma)$.  Let
  $\set{(\chunk{x}{1}{p},\chunk{u}{1}{q}) \mapsto
    \tilde\psi^\theta_{\chunk{u}{1}{q}}(\chunk{x}{1}{p})}{\theta \in \Theta}$ be a
  family of measurable functions from
  $(\Xset^p\times \Uset^q, \Xsigma^{\otimes p} \otimes \Usigma^{\otimes q})$ to
  $(\Xset, \Xsigma)$, called the \emph{reduced link functions} and let
  $\set{G^\theta}{\theta \in \Theta}$ be a family of probability kernels on
  $\Xset\times\Ysigma$, called the \emph{observation kernels}.
  \begin{enumerate}[label=(\roman*),wide=0pt]
\item
  A time series
  $\nsequence{Y_k}{k \ge -q+1}$ valued in $\Yset$ is said to be distributed
  according to an \emph{observation-driven model of order $(p,q)$} (hereafter,
  ODM$(p, q)$) with reduced link function $\tilde\psi^\theta$,
  admissible mapping $\Upsilon$ and observation kernel $G^\theta$ if there
  exists a process $\nsequence{X_k}{k \ge -p+1}$ on $(\Xset, \Xsigma)$ such
  that for all $k\in\zsetp$,
\begin{equation}\label{eq:def:gen-ob}
\begin{split}
&Y_{k}\mid \mathcal{F}_{k}\sim G^\theta(X_{k};\cdot)\eqsp,\\
&X_{k+1}=\tilde\psi^\theta_{\chunk{U}{k-q+1}{k}}(\chunk{X}{(k-p+1)}{k})\eqsp,
\end{split}
\end{equation}
where 
$\mathcal{F}_k=\sigma\left(\chunk{X}{(-p+1)}{k},\chunk{Y}{(-q+1)}{(k-1)}\right)$
and $U_j=\Upsilon(Y_j)$ for all $j>-q$.
\item We further say that this
model is a \emph{vector linearly observation-driven model of order
  $(p,q,p',q')$} (shortened as VLODM$(p,q,p',q')$) if moreover for
some $p',q' \in \zsetpnz$, $\Xset$ is a closed subset of $\rset^{p'}$,
$\Uset \subset \rset^{q'}$ and, for all
$x=\chunk{x}{0}{(p-1)}\in\Xset^p$, $u=\chunk{u}{0}{(q-1)}\in\Uset^q$,
and $\theta\in\Theta$,
\begin{equation}
  \label{eq:tilde-psi:affine:vector:case}
\tilde\psi^\theta_{u}(x)=\boldsymbol{\omega}(\theta)+\sum_{i=1}^p A_i(\theta) \, x_{p-i}
+\sum_{i=1}^q B_i(\theta) \, u_{q-i}\;,
\end{equation}
for some mappings $\boldsymbol{\omega}$, $\chunk A1p$ and $\chunk B1q$
defined on $\Theta$  and
valued in $\rset^{p'}$, $\left(\rset^{p'\times p'}\right)^p$ and
  $\left(\rset^{p'\times q'}\right)^q$.
In the case where $p'=q'=1$, the VLODM$(p,q,p',q')$ is simply called a
\emph{linearly observation-driven model of order $(p,q)$} (shortened
as LODM$(p,q)$).
  \end{enumerate}
\end{definition}
\begin{remark}\label{rem:observ-driv-model}
  Let us comment briefly on this definition.
  \begin{enumerate}[label=(\arabic*)]
\item The standard definition of an observation driven model does not
  include the admissible mapping $\Upsilon$ and indeed, we can define the
  same model without $\Upsilon$ by replacing the second equation
  in~(\ref{eq:def:gen-ob}) by
  $$
  X_{k+1}=\psi^\theta_{\chunk{Y}{k-q+1}{k}}(\chunk{X}{(k-p+1)}{k})\eqsp,
  $$
  where
  $\set{(\chunk{x}{1}{p},\chunk{y}{1}{q}) \mapsto
    \psi^\theta_{\chunk{y}{1}{q}}(\chunk{x}{1}{p})}{\theta \in \Theta}$ is a
  family of measurable functions from $(\Xset^p\times \Yset^q, \Xsigma^{\otimes p} \otimes \Ysigma^{\otimes q})$ to
  $(\Xset, \Xsigma)$, called the (non-reduced) \emph{link functions}, and
  defined by
  \begin{equation}
    \label{eq:non-reduced:link:function}
    \psi^\theta_y(x)=\tilde\psi^\theta_{\Upsilon^{\otimes q}(y)}(x)\;,\qquad x\in\Xset^p\,,\;y\in\Yset^q\;.
  \end{equation}
  However by inserting the mapping $\Upsilon$, we introduce some
  flexibility and this is useful for describing various ODMs with the
  same reduced link function $\tilde\psi^\theta$. For instance all
  LODMs or VLODMs use the form of reduced link function
  in~(\ref{eq:tilde-psi:affine:vector:case}) although they may use
  various mappings $\Upsilon$'s. This is the case for the GARCH,
  log-linear GARCH and NBIN-GARCH, see below, but also of the
  bivariate integer-valued GARCH model \cite{cui-zhu-2018}.
    \item
  When $p=q=1$, then the ODM$(p,q)$ defined by \eqref{eq:def:gen-ob}
  collapses to the (first-order) ODM considered in
  \cite{dou:kou:mou:2013} and \cite{douc2015handy}.  Note also that if
  $p\neq q$, setting $r\eqdef\max(p,q)$, the ODM($p,q$) can be
  embedded in an ODM($r,r$), but this requires augmenting the
  parameter dimension which might impact the identifiability of the
  model.
\item\label{item:rem:observ-driv-model-nonlin} Note that in our
  definition, $\Upsilon$ does not depend on $\theta$. In some cases,
  one can simplify the link function $\tilde{\psi}$, by allowing one
  to let $\Upsilon$ depend on the unknown parameter. To derive
  identifiability conditions or to prove the consistency of the MLE,
  the dependence of the distribution with respect to the unknown
  parameter $\theta$ is crucial.  In contrast, proving that the
  process is ergodic is always done \emph{for a given} $\theta$ and it is thus
  possible to use our set of conditions to prove ergodicity with
  $\Upsilon$ depending on $\theta$, hence reasoning for a given
  $\theta$ and a given $\Upsilon$. Moreover our set of conditions
  (namely Conditions~\ref{item:CCMShyp},
  \ref{assum:bound:rho:gen},~\ref{assum:weak:feller:X},~\ref{assum:V:gen:X},~\ref{assum:reaching:mild}
  and~\ref{assum:alpha-phi:gen} in \autoref{thm:ergodicity:gen})
  depends on $\Upsilon$ only through the image space $\Uset$, which
  can usually be taken to be the same even in situations where
  $\Upsilon$ may depend on $\theta$.
\end{enumerate}
\end{remark}
The GARCH model has been extensively studied, see, for example,
\cite{bougerol:picard:1992, francq2004maximum, francq2009tour,
  lindner2009stationarity, francq2011garch} and the references therein. Many
other examples have been derived in the class of LODMs. An
important feature of the GARCH is the fact that $x\mapsto G^\theta(x;\cdot)$
maps a family of distributions parameterized by the scale parameter $\sqrt{x}$,
which is often expressed by replacing the first line in~(\ref{eq:def:gen-ob})
by an equation of the form $Y_k=\sqrt{X_k}\epsilon_k$ with the assumption that
$\sequence{\epsilon}$ is i.i.d. Such a simple multiplicative formula no longer
holds when the observations $Y_k$'s are integers, which seriously complicates
the theoretical analysis of such models, as explained in
\cite{tjostheim2015count}.
Our results will apply to the GARCH case
but we do not improve the existing results in this case. The real
interest of our approach lies in being able to address general order integer-valued
ODMs for which existing results are scarce.  

\begin{definition}[Space $(\Zset,\Zsigma)$, Notation $\PP^\theta_\Xinit,\PE^\theta_\Xinit$]\label{def:pomm}
  Consider an ODM as in \autoref{def:obs-driv-gen}.  Then, for all
  $k\geq0$, the conditional distribution of $(Y_k,X_{k+1})$ given
  $\mathcal{F}_k $ only depends on
\begin{equation}
 \label{eq:Zk:def}
Z_k=\left(\chunk X{(k-p+1)}k,\chunk
  U{(k-q+1)}{(k-1)}\right)\in\Zset\;,
\end{equation}
where we defined
\begin{equation}
  \label{eq:Zset:def}
\Zset  =\Xset^p\times\Uset^{q-1}\quad\text{endowed with the $\sigma$-field $\Zsigma=\Xsigma^{\otimes p}\otimes\Usigma^{\otimes {(q-1)}}$.}
\end{equation}
For all $\theta\in\Theta$ and all probability distributions $\Xinit$ on
$(\Zset,\Zsigma)$, we denote by $\PP_\Xinit^\theta$ the distribution of
$\nsequence{X_k,Y_{k'}}{k>-p,\ k'>-q}$ satisfying~(\ref{eq:def:gen-ob}) and
$Z_0\sim\Xinit$, with the usual notation $\PP_z^\theta$ in the case where
$\Xinit$ is a Dirac mass at $z\in\Zset$. We replace $\PP$ by $\PE$ to
denote the corresponding expectations.
\end{definition}
We always assume that the observation kernel is dominated by a $\sigma$-finite
measure $\nu$ on $(\Yset,\Ysigma)$, that is, for all
$\theta\in\Theta$, there exists a measurable function
$g^\theta:\Xset\times\Yset \to \rsetp$ written as
$(x,y)\mapsto g^\theta (x;y)$ such that for all $x\in\Xset$,
$g^\theta (x;\cdot)$ is the density of $G^\theta (x;\cdot)$ with
respect to $\nu$.
Then the inference about the model parameter is classically performed by relying on
the likelihood of the observations $(Y_0,\ldots,Y_n)$ given
$Z_0$. For all
$z\in\Zset$, the corresponding conditional
density function $p^{\theta}(\chunk y0n|z)$ with respect to $\nu^{\otimes (n+1)}$ under parameter
$\theta\in\Theta$ is given by
\begin{equation} \label{eq:lkd:Y:cl:condX1:gen}
\chunk{y}{0}{n} \mapsto \prod_{k=0}^{n} g^\theta\left(x_k;y_{k}\right) \eqsp,
\end{equation}
where, setting $u_k=\Upsilon(y_k)$ for $k=0,\dots,n$,
 the sequence $\chunk x0n$ is defined through the initial conditions and
recursion equations
\begin{align}
  \label{eq:pq-order-rec-equation-with-u}
\begin{cases}
x_k=\Projarg {p+k}z \;,&
    -p< k\leq 0\;,\\
u_k=\Projarg {p+q+k}z \;,&
-q< k\leq -1\;,\\
x_k=\tilde{\psi}_{\chunk{u}{(k-q)}{(k-1)}}^\theta\left(\chunk x{(k-p)}{(k-1)}\right)\;,&
    1\leq k\leq n\;,
  \end{cases}
\end{align}
where, throughout the paper, for all $j\in\{1,\ldots,p+q-1\}$, we denote
by $\Projarg{j}{z}$ the $j$-th entry of  $z\in\Zset$.
Note that $x_{k+1}$ only depends on $z$ and $\chunk y0k$ for
all $k\geq0$. Throughout the paper, we use the notation: for all $n\geq1$,
$\chunk{y}{0}{n-1}\in\Yset^{n}$ and $z\in\Zset$,
\begin{align}
  \label{eq:lkd:psi:n:def:rec}
&  \f{\chunk{y}0{(n-1)}}(z)\eqdef x_{n} \;,\text{ with $x_n$ defined by \eqref{eq:pq-order-rec-equation-with-u}.}
\end{align}
Note that for $n=0$, $\chunk{y}0{-1}$ is the empty sequence and
\eqref{eq:lkd:psi:n:def:rec} is replaced by
$\f{\emptyseq}(z)=x_0=\Projarg pz$.  Given the initial condition
$Z_0=\initmle$, the (conditional) maximum likelihood estimator $\mlY{\initmle,n}$ of the
parameter $\theta$ is thus defined by
\begin{equation}
\label{eq:defi:mle-od:gen}
\mlY{\initmle,n} \in \argmax_{\theta \in \Theta} \mathsf{L}_{\initmle,n}^\theta\eqsp,
\end{equation}
where, for all $\initmle\in\Zset$,
\begin{equation}
\label{eq:defi:lkdM:gen:od}
\mathsf{L}_{\initmle,n}^\theta\eqdef   n^{-1}\sum_{k=1}^{n} \ln g^\theta\left(\f{\chunk{Y}{0}{(k-1)}}(\initmle);Y_{k}\right) \eqsp.
\end{equation}
Note that since $\sequence{X}$ is unobserved, to compute the conditional
likelihood, we used some arbitrary initial values for $\chunk{X}{(-p+1)}{0}$,
(the first $p$ entries of $\initmle$). We also use
 arbitrary values for the last $q-1$ entries of
$\initmle$. Note that we index our set of observations as $\chunk{Y}{0}{n}$ (hence
assuming $1+n$ observations to compute $\mathsf{L}_{\initmle,n}^\theta$).  The iterated function
$\f{\chunk{Y}{0}{k}}$ can be cumbersome but is very easy to compute in practice
using the recursion~(\ref{eq:pq-order-rec-equation-with-u}). Moreover, the same kind
of recursion holds for its derivatives with respect to $\theta$, allowing one
to apply gradient steps to locally maximize the likelihood.

In this contribution, we investigate the convergence of $\mlY{\initmle,n}$ as
$n\to\infty$ for some (well-chosen) value of $\initmle$ under the assumption
that the model is well specified and the observations are in \emph{the} steady
state. The first problem to solve in this line of results is thus to show the
following.
\begin{hyp}{A}
\item\label{assum:gen:identif:unique:pi:gen} For all $\theta\in\Theta$, there exists a unique stationary solution
satisfying~(\ref{eq:def:gen-ob}).
\end{hyp}
This ergodic property is the cornerstone for making statistical inference
theory work and we provide simple general conditions in
\autoref{sec:ergodicity:gen:od}. We now introduce the notation that will allow
us to refer to the stationary distribution of the model throughout the {paper}.
\begin{definition}\label{def:equi:theta:gen:od}
  If~\ref{assum:gen:identif:unique:pi:gen} holds, then
\begin{enumerate}[label=\alph*)]
\item $\PP^\theta$ denotes the distribution on
  $((\Xset\times\Yset)^{\zset},(\Xsigma\times\Ysigma)^{\otimes\zset})$ of the stationary solution
  of~(\ref{eq:def:gen-ob}) extended on $k\in\zset$, with
  $\mathcal{F}_k=\sigma(\chunk X{-\infty}{k},\chunk Y{-\infty}{(k-1)})$;
\item $\tilde{\mathbb{P}}^\theta$ denotes the projection of  $\PP^\theta$  on the component $\Yset^{\zset}$.
\end{enumerate}
We also use the symbols $\mathbb{E}^\theta$ and
$\tilde{\mathbb{E}}^\theta$ to denote the expectations corresponding
to $\PP^\theta$ and $\tilde{\mathbb{P}}^\theta$, respectively.  We
further denote by $\piX^\theta$ and $\piY^\theta$ the marginal
distributions of $X_0$ and $Y_0$ under $\PP^\theta$, on
$(\Xset,\Xsigma)$ and $(\Yset,\Ysigma)$, respectively.  As a byproduct
of the proof of~\ref{assum:gen:identif:unique:pi:gen}, one usually
obtains a function $\VX:\Xset\to\rsetp$ of interest, common to all
$\theta\in\Theta$, such that the following property holds on the
stationary distribution (see \autoref{sec:ergodicity:gen:od}).
\begin{hyp}{A}
\item\label{ass:21-lyapunov:gen:od} For all $\theta\in\Theta$, $\piX^\theta(\VX)<\infty$.
\end{hyp}
It is here stated as an assumption for convenience. Note also that, in the
following, for $V:\Xset\to\rsetp$ and $f:\Xset\to\rset$, we denote
the $V$-norm of $f$ by
$$
\vnorm[V]{f}=\sup\set{\frac{|f(x)|}{V(x)}}{x\in\Xset}\;,
$$
with the convention $0/0=0$ and we write $f\lesssim V$ if $|f|_V<\infty$.  With
this notation, under~\ref{ass:21-lyapunov:gen:od}, for any $f:\Xset\to\rset$
such that $f\lesssim\VX$, $\piX^\theta(|f|)<\infty$ holds, and similarly, since
$\piY^\theta=\piX^\theta G^\theta$ as a consequence of~(\ref{eq:def:gen-ob}),
for any $f:\Yset\to\rset$ such that $G^\theta(|f|)\lesssim\VX$, we have
$\piY^\theta(|f|)<\infty$.
\subsection{Equivalence-class consistency and identifiability}
For all $\theta, \theta'\in\Theta$, we write $\theta\sim\theta'$ if
and only if $\tilde{\mathbb{P}}^{\theta}=\tilde{\mathbb{P}}^{\theta'}$. This defines an
equivalence relation on the parameter set $\Theta$ and the corresponding
equivalence class of $\theta$ is denoted by $[\theta]\eqdef\{\theta' \in\Theta:\; \theta\sim\theta'\}$.
\end{definition}
The equivalence relationship $\sim$ was introduced by \cite{leroux:1992} as an
alternative to the classical identifiability condition. Namely, we say that
$\mlY{\initmle,n}$ is \emph{equivalence-class consistent} at the true parameter
$\thv$ if
\begin{equation}\label{eq:equi:class:consistency:od}
\lim_{n\to\infty}\met(\mlY{\initmle,n},[\thv])=0,\quad \tilde\PP^{\thv}\as
\end{equation}
Recall that $\met$ is the metric endowing the parameter
space $\Theta$. Therefore,
in~(\ref{eq:equi:class:consistency:od}),
$\met(\mlY{\initmle,n},[\thv])$ is the distance between the MLE and
the set of parameters having the same stationary distribution as the
true parameter, and the convergence is considered under this common
stationary distribution.

Identifiability can then be treated as a separate problem which consists
in determining all the parameters $\thv$ for which $[\thv]$ reduces
to the singleton $\{\thv\}$, so that equivalence-class consistency
becomes the usual consistency at and only at these parameters. The
identifibility problem is treated in \cite[Proposition~11 and Theorem~12]{douc-roueff-sim_ident-genod2020}
for general order ODMs, where many cases of interest are specifically
detailed. We will use in particular
\cite[Theorem~17 and Section~5.5]{douc-roueff-sim_ident-genod2020}. 

\subsection{Three examples}\label{sec:examples:gen:od}

To motivate our general results we first explicit three models of
interest, namely, the log-linear Poisson GARCH$(p,q)$, the
NBIN-GARCH$(p,q)$ and the PARX$(p,q)$ models. To the best of our knowledge, the
stationarity and ergodicity as well as the asymptotic properties of the
MLE for the general log-linear Poisson GARCH$(p,q)$ and
NBIN-GARCH$(p,q)$ models have not been derived so far. Such results
are available for PARX model in \cite{AGOSTO2016640} but our approach
leads to significantly different assumptions, as will be shown in
\autoref{sec:parx-model}.

Once the ergodicity of the model is established, we can investigate
the consistency of the MLE. This is done by first investigating the
equivalence-class consitency of the MLE and then we only need to add
an identifiability condition to get the consistency of the MLE, see
Theorems~\ref{theo:ergo-convergence:logpois:gen},~\ref{theo:ergo-convergence:garch:gen:nbin}
and~\ref{theo:ergo-convergence:parx} hereafter. Such results pave the
way for investigating the asymptotic normality. This is done in
\cite[Proposition~5.4.7(iii) and Proposition~5.4.15]{sim-tel-01458087}
for the first two examples and in \cite[Theorems~2
and~3]{AGOSTO2016640} for the third one.  
\subsubsection{Log-linear Poisson GARCH model}
The Log-linear Poisson GARCH model is defined as follows in our setting.
    \begin{example}
    \label{expl:log-lin-poi}
    The Log-linear Poisson GARCH$(p,q)$ Model parameterized by
    $\theta=(\omega,\chunk{a}{1}{p},\chunk{b}{1}{q})\in\Theta\subset\rset^{1+p+q}$
    is a LODM$(p,q)$ with affine reduced link function of the
    form~(\ref{eq:tilde-psi:affine:vector:case})
    with  coefficients given by
    \begin{align}\label{eq:lodm-natural-param}
      \begin{split}
      \boldsymbol{\omega}(\theta)=\omega\;,\quad
      A_i(\theta) = a_i\quad\text{for}\quad 1\leq i\leq p\\
      \text{and}\quad
      B_i(\theta)  = b_i\quad\text{for}\quad 1\leq i\leq q\;,
      \end{split}
    \end{align}
    with observations space $\Yset=\zsetp$, hidden variables space
    $\Xset=\rset$,
    admissible mapping $\Upsilon(y)=\ln(1+y)$, and observation kernel
    $G^\theta (x;\cdot)$ defined as the Poisson distribution
    with mean $\rme^x$.
\end{example}

Our condition for showing the ergodicity of general order log-linear Poisson
GARCH models requires the following definition.
For all
$x=\chunk x{(1-(p\vee q))}0\in\rset^{p\vee q}$, $m\in\zsetp$, and
$w=\chunk w{(-q+1)}{m}\in\{0,1\}^{q+m}$, define ${\flp{w}}(x)$ as
$x_{m+1}$ obtained by the recursion
 \begin{equation}
   \label{eq:log:pois:ergo:cond:rec:def}
   x_k=\sum_{j=1}^pa_j\ x_{k-j}+\sum_{j=1}^q b_j\ w_{k-j}\ x_{k-j}\;,\qquad
   1\leq k\leq m+1
 \end{equation}
 We can now state our result on
 general order log-linear Poisson GARCH models.
 \begin{theorem}\label{theo:ergo-convergence:logpois:gen}
   Consider the log-linear Poisson GARCH$(p,q)$ model, which satisfies
   Eq.~(\ref{eq:def:gen-ob}) under the setting
   of \autoref{expl:log-lin-poi}. Suppose that, for all $\theta\in\Theta$, we have
   \begin{equation}
     \label{eq:log:poi:ergo:cond}
     \lim_{m\to\infty}
    \max\set{\left|{\flp{w}}(x)\right|}{w\in\{0,1\}^{q+m}} = 0\quad\text{for all $x\in\rset^{p\vee q}$}\;,
  \end{equation}
  and define
  \begin{align}
  \label{eq:ident-pooly}
\mathrm{P}_p(z;\chunk{a}1p)=  z^p-\sum_{k=1}^pa_{k}z^{p-k} \quad\text{and}\quad
  \mathrm{Q}_q(z;\chunk{b}1q)=
\sum_{k=0}^{q-1} b_{k+1}\ z^{q-1-k}\;.
\end{align}
Then, the following assertions hold.
 \begin{enumerate}[label=(\roman*),wide=0pt]
 \item\label{item:thm:logpois:ergodic:gen:od}
For all $\theta\in\Theta$, there exists a unique stationary solution
$\nsequence{(X_k,Y_k)}{k\in\zsetp}$ to~(\ref{eq:def:gen-ob}), that
is,~\ref{assum:gen:identif:unique:pi:gen}
holds. Moreover, for any $\tau>0$,~\ref{ass:21-lyapunov:gen:od} holds with
 $\VX: \rset\to\rsetp$ defined by
 \begin{equation}\label{eq:def:V:logpois:gen}
\VX(x)=\rme^{\tau |x|}\eqsp,\qquad  x\in\rset\eqsp.
\end{equation}
\item\label{item:thm:logpois:strong:consistency:gen:od} For any
  $\initmlex_1\in\rset$ and $\initmley_1\in\zsetp$, setting
  $\initmle=(\initmlex_1,\dots,\initmlex_1,\Upsilon(\initmley_1),\cdots,\Upsilon(\initmley_1))\in\rset^p\times\rset^{q-1}$,
  the MLE $\mlY{\initmle,n}$ as defined by~(\ref{eq:defi:mle-od:gen})
  is equivalence-class consistent, that
  is, (\ref{eq:equi:class:consistency:od}) holds for any 
  $\thv\in\Theta$.
\item\label{item:thm:logpois:ident} If the true parameter
  $\thv=(\omega^\star,\chunk{a^\star}{1}{p},\chunk{b^\star}{1}{q})$ moreover
  satisfies that the polynomials $\mathrm{P}_p(\cdot;\chunk{a^\star}1p)$ and
    $\mathrm{Q}_q(\cdot;\chunk{b^\star}1q)$ defined by~(\ref{eq:ident-pooly})
have no common complex roots, then the MLE $\mlY{\initmle,n}$ is
  consistent.
 \end{enumerate}
\end{theorem}
The proof is postponed to \autoref{sec:proof-autor-conv:logpoi}.
\begin{remark}
  \label{rem:log:poi:ergo:cond}
Let us provide some insights about Condition~(\ref{eq:log:poi:ergo:cond}).
\begin{enumerate}[label=(\arabic*)]
\item\label{item:rem:log:poi:ergo:1} Using the two possible constant sequences
  $w$, $w_k=0$ for all $k$ or $w_k=1$ for all $k$ in
  ~(\ref{eq:log:pois:ergo:cond:rec:def}), we easily see
  that~(\ref{eq:log:poi:ergo:cond}) implies
  $$
\chunk a1p\in\mathcal{S}_p\quad\text{and}\quad  \chunk a1{(p\vee q)}+\chunk b1{(p\vee q)}\in\mathcal{S}_{p\vee q}\;,
$$
where we used the usual convention $a_k=0$ for $p<k\leq q$ and $b_k=0$ for
$q<k\leq p$ and where
\begin{align}
  \label{eq:char-poly-roots-ok}
 \mathcal{S}_p=\set{\chunk{c}{1}{p}\in\rset^p}{\forall
        z\in\cset,
  |z|\leq1\text{ implies }1-\sum_{k=1}^pc_kz^k\neq0}  \;.
\end{align}
\item A sufficient condition to have~(\ref{eq:log:poi:ergo:cond}) is
  \begin{equation}
    \label{eq:log:poi:ergo:cond1}
    \sup\set{\left|{\flp{w}}(x)\right|}{w\in\{0,1\}^{q},\,x\in[-1,1]^{p\vee q}}<1\;.
  \end{equation}
  Indeed, defining $\rho$ as the left-hand side of the previous display, we
  clearly have, for all $m>p\vee q$, $w\in\{0,1\}^{q+m}$ and  $x\in\rset^{p\vee
    q}$,
  $$
\left|{\flp{w}}(x)\right|\leq\rho\,
\max\set{\left|{\flp{w'}}(x)\right|}{w'\in\{0,1\}^{q+m-j},1\leq j  \leq p\vee q}\;.
  $$
\item\label{item:rem:log:poi:ergo:2} The
  first iteration~(\ref{eq:log:pois:ergo:cond:rec:def}) implies, for all
  $w\in\{0,1\}^q$ and $x\in[-1,1]^{p\vee q}$,
  $$
  |\flp{w}(x)|\leq\left[\sum_{k=1}^{p\vee q}\left(|a_k|\vee|a_k+b_k|\right)\right]\;.
  $$
  Hence a sufficient condition to
  have~(\ref{eq:log:poi:ergo:cond1}) (and thus~(\ref{eq:log:poi:ergo:cond})) is
  \begin{equation}
    \label{eq:log:poi:ergo:cond2}
  \sum_{k=1}^{p\vee q}\left(|a_k|\vee|a_k+b_k|\right)<1 \;.
\end{equation}
\item\label{item:rem:log:poi:ergo:peq1} When $p=q=1$, by Points~\ref{item:rem:log:poi:ergo:1}
  and~\ref{item:rem:log:poi:ergo:2} above,
  Condition~(\ref{eq:log:poi:ergo:cond}) is equivalent to have $|a_1|<1$ and
  $|a_1+b_1|<1$.
  This condition is weaker than the one derived in \cite{dou:kou:mou:2013} where
$|b_1|<1$ is also imposed.
\end{enumerate}
\end{remark}

\subsubsection{NBIN-GARCH model}
Our next example is the NBIN-GARCH$(p,q)$, which is defined as follows
in our setting.
  \begin{example}
    \label{exmpl:nbin-garch:gen}
    The NBIN-GARCH$(p,q)$ model is a LODM$(p,q)$ parameterized by
    $\theta=(\omega,\chunk{a}{1}{p},\chunk{b}{1}{q},r)\in\Theta\subset\rsetpnz\times\rsetp^{p+q}\times
    \rsetpnz$ with affine reduced link function of the
    form~(\ref{eq:tilde-psi:affine:vector:case})
    with coefficients given by~(\ref{eq:lodm-natural-param}),
    observations space $\Yset=\zsetp$, hidden variables space
    $\Xset=\rsetp$,
    admissible mapping $\Upsilon(y)=y$, and observation kernels
    $G^\theta(x;\cdot)$ defined as the negative binomial
    distribution with shape parameter $r>0$ and mean $r\,x$, that is,
    for all $y\in\zsetp$,
    $$
G^\theta(x;\{y\})=\frac{\Gamma(r+y)}{y\,!\,\Gamma(r)}\left(\frac{1}{1+x}\right)^r\left(\frac{x}{1+x}\right)^y\eqsp.
    $$
\end{example}
We now state our result for general order  NBIN-GARCH model.
\begin{theorem}\label{theo:ergo-convergence:garch:gen:nbin}
  Consider the NBIN-GARCH$(p,q)$ model, which satisfies
   Eq.~(\ref{eq:def:gen-ob}) under the setting
   of \autoref{exmpl:nbin-garch:gen}. Suppose that, for all
  $\theta=(\omega,\chunk{a}{1}{p},\chunk{b}{1}{q},r)\in\Theta$, we have
  \begin{equation}\label{eq:con:ergo:gen:nbin:cns}
\sum_{k=1}^pa_k+r\ \sum_{k=1}^q b_k<1\eqsp.
  \end{equation}
  Then the following assertions hold.
  \begin{enumerate}[label=(\roman*),wide=0pt]
  \item\label{item:thm:nbin:ergo:gen:od} For all $\theta\in\Theta$, there
    exists a unique stationary solution to~(\ref{eq:def:gen-ob}), that
    is,~\ref{assum:gen:identif:unique:pi:gen} holds. Moreover,
    ~\ref{ass:21-lyapunov:gen:od} holds with $\VX(x)=x$ for all $x\in\rset_+$.
   \item\label{item:thm:nbin:strong:consistency:gen:od}
     For any $\initmlex_1\in\rsetpnz$ and $\initmley_1\in\zsetp$,
  setting
  $\initmle=(\initmlex_1,\dots,\initmlex_1,\initmley_1,\cdots,\initmley_1)\in\rsetp^p\times\zsetp^{q-1}$, the MLE $\mlY{\initmle,n}$ as
  defined by~(\ref{eq:defi:mle-od:gen}) is equivalence-class consistent, that
  is, (\ref{eq:equi:class:consistency:od}) holds for any 
  $\thv\in\Theta$.
\item\label{item:thm:nbin:ident} If the true parameter
  $\thv=(\omega^\star,\chunk{a^\star}{1}{p},\chunk{b^\star}{1}{q},r^\star)$ moreover
  satisfies that the polynomials $\mathrm{P}_p(\cdot;\chunk{a^\star}1p)$ and
    $\mathrm{Q}_q(\cdot;\chunk{b^\star}1q)$  defined by~(\ref{eq:ident-pooly})
have no common complex roots, then the MLE $\mlY{\initmle,n}$ is consistent.
\end{enumerate}
\end{theorem}
The proof is postponed to \autoref{sec:proof-autor-conv:nbin}.
\begin{remark}
  \label{rem:nbin-not-optimal}
  Clearly, under the setting
   of \autoref{exmpl:nbin-garch:gen}, if Eq.~(\ref{eq:def:gen-ob}) has a
   stationary solution such that $\mu=\int x\, \piX(\rmd x)<\infty$, taking the
   expectation on both sides of the second equation in~(\ref{eq:def:gen-ob})
   and using that $\int y\, \piY(\rmd y)=r\int x\, \piX(\rmd x)$,
   then~(\ref{eq:con:ergo:gen:nbin:cns}) must hold, in which case we get 
   $$
   \mu=\left(1-\sum_{k=1}^pa_k-r\ \sum_{k=1}^q b_k\right)^{-1}\;.
   $$
   Hence~(\ref{eq:con:ergo:gen:nbin:cns}) is in fact necessary and sufficient
   to get a stationary solution admitting a finite first moment, as was already
   observed in \cite[Theorem~1]{zhu:2011} although the ergodicity is not proven
   in this reference.  However, we believe that, similarly to the classical
   GARCH$(p,q)$ processes, we can find stationary solutions to
   Eq.~(\ref{eq:def:gen-ob}) in the case where~(\ref{eq:con:ergo:gen:nbin:cns})
   does not hold.  This is left for future work.
\end{remark}

\subsubsection{The PARX model}\label{sec:parx-model}
The PARX model is similar to the standard INGARCH($p,q$) model but
with additional exogenous variables in the linear link function for
generating the hidden variables. Following \cite{AGOSTO2016640}, these
exogenous variables are assumed to satisfy some Markov dynamic of
order 1 independently of the observations and of the hidden variables
(see their Assumption~1). This leads to the following definition.
\begin{definition}[PARX model]
\label{def:parx-model}  Let $d,p,q$ and $r$ be positive integers. For
  $\gamma= \chunk{\gamma}{1}{d}\in \rsetp^d$, we set
  $\func_\gamma: \xi= (\xi_1, \ldots,\xi_d) \mapsto \sum_{j=1}^d
  \gamma_j f_j(\xi_j)$ where $f_i:\rset^r \to \rsetp$ are given
  functions for all $i\in\{1,\ldots,d\}$. Let ${\mathcal P}(x)$ denote
  the Poisson distribution with parameter $x$ and $L$ be a given
  Markov kernel on $\rset^r\times {\mathcal B} (\rset^r)$. We define
  the {Poisson AutoRegressive model with eXogenous variables},
  shortly written as PARX$(p,q)$, as a class of processes
  parameterized by 
  $\theta=(\omega,\chunk{a}{1}{p},\chunk{b}{1}{q},\chunk{\gamma}{1}{d})\in\Theta$,
  where $\Theta$ is a compact subset of $\rsetpnz\times\rsetp^{p+q+d}$,
  and satisfying
\begin{align}
&X_{t}=\omega+\sum_{i=1}^{p}a_i X_{t-i}+\sum_{i=1}^q b_i
                Y_{t-i}+\func_\gamma (\Xi_{t-1}) \nonumber \\
&(Y_t,\Xi_{t})|{\mcf_t}\sim {\mathcal P} (X_t)\otimes
                                                                L(\Xi_{t-1};\cdot) \;, \label{eq:def:parx}
\end{align}
where $\mathcal{F}_t=\sigma\left(\chunk{X}{(-p+1)}{t},\chunk{Y}{(-q+1)}{(t-1)},\chunk{\Xi}{(-q+1)}{(t-1)}\right)$. 
\end{definition}
 \begin{remark}
   Note that our $a_1,\dots,a_p$ and $b_1,\dots,b_q$ correspond to
   $\beta_1,\dots,\beta_q$ and $\alpha_1,\dots,\alpha_p$ of
   \cite{AGOSTO2016640}.  Here we allowed $r$ to be different from $d$
   whereas in in \cite{AGOSTO2016640} it is assumed that $d=r$ and
   they are denoted by $d_x$.
 \end{remark}
 Since the exogenous variables are observed
 we can recast the PARX model into a VLODM$(p,q)$ by including the
 exogenous variables into the observations $Y_t$ and the hidden
 variable $X_t$. Namely, we can formally see the above PARX model as a
 special case of VLODM as follows.
 \begin{example}[PARX model as a VLODM]
   \label{exmpl:parx}
   Consider a PARX model as in \autoref{def:parx-model}.
Set $\bar Y_t=(Y_t,\Xi_t)$, $\bar X_t=(X_t,\Xi_{t-1})$, which are
valued in
 $\bar \Yset=\zsetp\times\rset^r$ and  $\bar
 \Xset=\rsetp\times\rset^r$. Then $\nsequence{\bar Y_k}{k \ge -q+1}$ is a
 VLODM($p,q,1+r,1+d+r$) by setting
 \begin{enumerate}[label=\alph*)]
 \item the admissible mapping as
 $\Upsilon(\bar y):=(y,f_1(\xi),\dots,f_{d}(\xi),\xi)\in \bar \Uset \eqdef \zsetp \times \rsetp^{d}\times\rset^r$ for
 all $\bar y=(y,\xi)\in\zsetp\times\rset^r$;
\item\label{item:exmpl:parx:G} for all $\theta=(\omega,\chunk{a}{1}{p},\chunk{b}{1}{q},\chunk{\gamma}{1}{d})$, 
 $\boldsymbol{\omega}(\theta):=\begin{bmatrix}\omega\\0_{r,1}
 \end{bmatrix}$,
$A_k(\theta):=
 \begin{bmatrix}
   a_k&0_{1,r}\\
   0_{r,1}&0_{r,r}
 \end{bmatrix}
$ for
$k=1,\dots,p$,
$B_1(\theta):=\begin{bmatrix}
   b_1&\begin{bmatrix}\gamma_1&\dots&\gamma_d
   \end{bmatrix}&0_{1,r}\\
   0_{r,1}&0_{r,d}&\mathrm{I}_r
 \end{bmatrix}$
and $B_k(\theta):= \begin{bmatrix}
   b_k&0_{1,d+r}\\
   0_{r,1}&0_{r,d+r}
 \end{bmatrix}
 $ for $k=1,\dots,q$;
\item  for all $\bar x=(x,\xi)\in\rsetp\times\rset^r$, $G^\theta(\bar x;\cdot):={\mathcal P} (x)\otimes L(\xi;\cdot)$. 
 \end{enumerate}
\end{example}
We end up this section by stating our result for general PARX model.
For establishing the ergodicity of the PARX model, we need some
assumptions on the dynamic of the exogeneous variables through the
Markov kernel $L$. Namely we consider the following assumptions on the
kernel $L$.
\begin{hyp}{L}
\item\label{hyp:parx-model-L-ell} The Markov kernel $L$ admits a kernel
density $\ell$ with respect to some measure $\nu_\X$ on Borel sets of
$\rset^r$ such that
\begin{enumerate}[(i)]
\item \label{eq:assum:parx:i} for all $(\xi,\xi') \in\rset^r \times
  \rset^r$, $\ell (\xi;\xi')>0$;
\item \label{eq:assum:parx:ii} for all $\xi \in\rset^r$, there exists
  $\delta>0$ such that
  $$
  \int_{\rset^r} \sup_{\|\xi'-\xi\| \leq \delta} \ell (\xi';\xi'')\;\nu_\X(\rmd \xi'')<\infty \eqsp,
  $$
  where $\|\cdot\|$ denotes the Euclidean norm on $\rset^r$.  
\end{enumerate}
\item\label{hyp:parx-model-wf} The Markov kernel $L$ is weak-Feller
  ($Lf$ is bounded and continuous for any $f$ that is bounded and continuous).
\item \label{item:assum:parx:v}There exist a probability measure $\pi_\X$ on Borel sets of $\rset^r$, a measurable function $V_\X:\rset^r \to [1,\infty)$ and constants $(\varrho,C) \in (0,1) \times \rsetp$ such that
  \begin{enumerate}[(i)]
  \item $L$ is $\pi_\X$-invariant,
  \item \label{item:assum:parx:v:b} $\pi_\X(V_\X)<\infty$ and $\{V_\X\leq M\}$ is  a compact set for any $M>0$,
  \item\label{item:assum:parx:v:c} for all $i\in\{1,\ldots,d\}$, we
    have $f_i\lesssim V_\X$. 
  \item \label{item:assum:parx:v:d} for all Borel functions
    $h:\rset^r\to\rset$ such that $h\lesssim V_\X$, we have for all $(\xi,n) \in \zsetp$,
    $$
    |L^n h(\xi)-\pi_\X(h)|\leq C \, \varrho^n\, \vnorm[V_\X]{h}\;V_\X(\xi) \eqsp.
    $$
  \end{enumerate}
\item \label{item:assum:parx:vi} There exists a constant $M\geq 1$ such that for all $(\xi,\xi')\in \rset^{2r}$,
  $$
  \tvdist{L(\xi,\cdot)}{L(\xi',\cdot)} \leq M \|\xi-\xi'\|
  $$
  where $\tvdist{\nu}{\nu'}$ is the total variation distance
  between two probability measures $\nu,\nu'$ on $\rset^r$.
\item\label{assum:thm:parx:equiv:conv:mle:gen:od:L} There exist
  $C_\X\geq0$, $\mathrm{h}_\X:\rset_+\to\rset_+$ and a
  measurable function $\bar\phi_\X:\rset^r\to\rsetp$ such that the following
  assertions hold for $L$, its kernel density $\ell$ introduced
  in~\ref{hyp:parx-model-L-ell} and its drift function introduced in~\ref{item:assum:parx:v}.
\begin{enumerate}[label=(\roman*)]
\item\label{assum:thm:parx:equiv:conv:mle:gen:od-1} For all $\xi'\in\rset^r$, the mapping $\xi\mapsto\ell(\xi,\xi')$
  is continuous. 
\item\label{assum:thm:parx:equiv:conv:mle:gen:od-2} We have $\sup_{\xi,\xi'\in\rset^r} \ell(\xi,\xi')<\infty$.
\item\label{assum:thm:parx:barphi:bar:phi:gen:od} For all
  $\xi\in\rset^r$, we have $1+\|\xi\|\leq\bar\phi_{\X}(\xi)$.
\item\label{assum:thm:parx:ineq:loggg:gen:od} For all $(\xi,\xi',\xi'')\in\rset^{3r}$,
\begin{equation*}
\left|\ln\frac{\ell(\xi;\xi'')}{\ell(\xi';\xi'')}\right|\leq \mathrm{h}_\X(\|\xi-\xi'\|) \,
\rme^{C_\X \, \left(1+\|\xi\|\vee\|\xi'\right)}\;\bar\phi_\X(\xi'')\;.
\end{equation*}
\item \label{assum:thm:parx:H:gen:od} $\mathrm{h}_\X(u)=O(u)$ as $u\to0$.
\item \label{assum:thm:parx:22or23-barphi-V:gen:od} If $C_\X=0$, then
 $\lnp \bar{\phi}_\X \lesssim V_\X$. Otherwise,
 $\bar{\phi}_\X\lesssim V_\X$.
\end{enumerate}
\item\label{item:PARX--non-degenerateident-G-X} We have
    $L(v;\{\chunk f1d(\cdot)\in A\})<1$ for all  $v\in\Vset$ and affine hyperplanes
  $A\subset\rset^{d}$.
\end{hyp}
We can now state our main result for the PARX model.
 \begin{theorem}\label{theo:ergo-convergence:parx}
   Consider the PARX$(p,q)$ model of \autoref{def:parx-model}, seen as
   the VLODM of \autoref{exmpl:parx}. Suppose
   that~\ref{hyp:parx-model-L-ell}--\ref{item:assum:parx:vi} hold,
   and that, for all $\theta=(\omega,\chunk{a}{1}{p},\chunk{b}{1}{q},\chunk{\gamma}{1}{d})\in\Theta$, we have
   \begin{equation}
     \label{eq:cond:ergo-convergence:parx:one}
\sum_{i=1}^{p}a_i+\sum_{i=1}^q b_i < 1\;.     
   \end{equation}
   Then the following assertions hold.
   \begin{enumerate}[(i),wide=0pt]
   \item \label{item:thm:parx:ergo:gen:od}
   for all $\theta\in\Theta$, there exists a unique stationary solution
$\nsequence{(X_k,Y_k,\Xi_k)}{k\in\zsetp}$ to~(\ref{eq:def:gen-ob}), that
is,~\ref{assum:gen:identif:unique:pi:gen} holds for the VLODM of  \autoref{exmpl:parx}. Moreover,~\ref{ass:21-lyapunov:gen:od} holds with
 $\VX(\bar x)=x+V_\X(\xi)$ for all $\bar x=(x,\xi)\in\rsetp\times\rset^r$.
\item \label{item:thm:parx:equiv:conv:mle:gen:od} Suppose moreover
  that~\ref{assum:thm:parx:equiv:conv:mle:gen:od:L} holds.  Then, for
  any $\initmlex_1\in\rsetp$, $\initmlexi_1\in\rset^r$ and
  $\initmley_1\in\zsetp$, setting
  $\initmle=((\initmlex_1,\initmlexi_1),\dots,(\initmlex_1,\initmlexi_1),(\initmley_1,\initmlexi_1),\cdots,(\initmley_1,\initmlexi_1))\in(\rsetp\times\rset^r)^p\times(\zsetp\times\rset^r)^{q-1}$,
  the MLE $\mlY{\initmle,n}$ as defined by~(\ref{eq:defi:mle-od:gen})
  is equivalence-class consistent, that is,
  (\ref{eq:equi:class:consistency:od}) holds for any $\thv\in\Theta$.
\item \label{item:thm:parx:conv:gen:od}
Suppose in addition
that~\ref{item:PARX--non-degenerateident-G-X} holds. Then the
MLE $\mlY{\initmle,n}$ is consistent.
\end{enumerate}
\end{theorem}
The proof is postponed to \autoref{sec:proof-autor-conv:parx}. 
Let us briefly comment on our assumptions.
\begin{remark}
\label{rem:assumptions:parx}
Contrary to the previous example, the PARX model requires an
additional Markov kernel $L$ and therefore additional assumptions on
this kernel. We briefly comment on them hereafter and compare our
assumptions to those in \cite{AGOSTO2016640}.
\begin{enumerate}[(i)]
\item Assumptions
\ref{hyp:parx-model-L-ell}--\ref{item:assum:parx:v} are classical
assumptions on Markov kernels to ensure its stability and
ergodicity. In \cite[Assumption~2]{AGOSTO2016640}, a different approach is used and
instead a first order contraction of the random iterative  functions
defining the Markov transition is used. Their assumption would
typically imply our
\ref{item:assum:parx:v} with $V_\X(\xi)=1+\|\xi\|^s$ for some $s\geq1$,
with their Assumption(3)(ii) implying our
\ref{item:assum:parx:v}\ref{item:assum:parx:v:c}.
\item \label{item:assumption:tv} Our
  Assumption~\ref{item:assum:parx:vi} is inherited from our approach
  for proving ergodicity using the embedding of the PARX model into a
  VLODM model detailed in \autoref{exmpl:parx}. It is not clear to us
  whether it can be omitted for proving ergodicity. This question is
  left for future work. Note that we provide another formulation of
  this assumption in \autoref{rem:assumption:tv}, see
  \autoref{sec:proof-autor-conv:parx}. There is no equivalent of our
  Assumption~\ref{item:assum:parx:vi} in \cite{AGOSTO2016640} as their
  technique for proving ergodicity is different. However, although we
  have the same condition~(\ref{eq:cond:ergo-convergence:parx:one}) as
  them for ergodicity, see their Assumption~3(i), they require an
  additional condition, Assumption~3(iii), which we do not
  needed. This condition involves both the coefficients $a_i$ and
  $b_i$ and the contraction constant $\rho$ used in their stability
  assumption for the Markov transition of the covariates. In contrast,
  our condition on the parameters $a_i$ and $b_i$ of the model, merely 
  imposed through~(\ref{eq:cond:ergo-convergence:parx:one}), are
  completely separated from our
  assumptions~\ref{hyp:parx-model-L-ell}--\ref{item:assum:parx:vi} on
  the dynamics of the covariates.
\item Assumptions~\ref{assum:thm:parx:equiv:conv:mle:gen:od:L}
  and~\ref{item:PARX--non-degenerateident-G-X} are not required to get
  ergodicity. they are needed to establish the equivalence class
  consistency of the MLE and the identifiability of the model. Like
  our Assumption~\ref{item:assum:parx:vi} for proving ergodcity,
  Assumption~\ref{assum:thm:parx:equiv:conv:mle:gen:od:L} is inherited 
  from the embedding of the PARX model into a VLODM model here used
  for proving the equivalent class consistency.
  Assumption~\ref{item:PARX--non-degenerateident-G-X} is a mild and
  natural identifiability assumption. It basically says that the
  covariates $f_1(\xi_{k}),\dots,f_d(\xi_{k})$ are not linearly
  related conditionally to $\xi_{k-1 }$. If they were, it would
  suggest using a smaller set of covariates. The identifiability
  condition of \cite{AGOSTO2016640} is different as it involves a
  condition both on the covariate distribution and on the parameters
  $a_1,\dots,a_p$ and $b_1,\dots,b_q$, see their Assumption~5.
\end{enumerate}
\end{remark}
To conclude this section, let us carefully examine the simple case
where the covariates are assumed to follow a Gaussian linear dynamic
and see, in this specific case, how our assumptions compare to that of
\cite{AGOSTO2016640}. More precisely, assume that $L$ is defined by
the following equation on the exogenous variables
\begin{equation}
  \label{eq:linear-gaussian-covariate-dynamics}
\Xi_t=\aleph\; \Xi_{t-1} + \sigma \eta_t\;,
  \end{equation}
where $\aleph$ is an $r\times r$ matrix with spectral radius 
$\rho(\aleph)\in(0,1)$,  $\sigma>0$, and
$\eta_t \sim {\mathcal N}(0,I_r)$. Then
Assumption~\ref{hyp:parx-model-L-ell} holds
with
$$
\ell(\xi,\xi')=(2\pi \sigma^2)^{-r/2} \rme^{-\|\xi'-\aleph
  \xi\|^2/(2\sigma^2)}\;,
$$
and $\nu_\X$ being the Lebesgue measure on $\rset^r$. It is
straightforward to check~\ref{hyp:parx-model-wf}
and~\ref{item:assum:parx:v} with $V_\X(\xi)=\rme^{\lambda\,\|\xi\|}$
for any $\lambda>0$.

Let us now
check that Assumption \ref{item:assum:parx:vi} holds. First note that,
setting $f(x)= \rme^{-x^2/2}$, we have
$\ell(\xi,\xi'')=\frac{1}{(2\pi\sigma^2)^{r/2}}f(\|\xi''-\aleph\xi\|)$.
Then, for all $\xi,\xi',\xi''\in\rset^r$ such that
$\|\xi-\xi'\|\leq \epsilon$, which implies
$\left\|\aleph(\xi-\xi')\right\|\leq\|\aleph\| \epsilon$, where
$\|\aleph\|$ denotes the operator norm of $\aleph$, and thus
\begin{align*}
|\ell(\xi,\xi'')-\ell(\xi',\xi'')|&\leq \frac{\|\aleph\|}{(2\pi\sigma^2)^{r/2}}\,\|\xi-\xi'\|\,
\sup_{\|\xi''\|-\|\aleph\|\epsilon\leq x\leq
                                    \|\xi''\|+\|\aleph\|\epsilon}\left|f'\left(x\right)\right|\\
                                  &\leq \frac{\|\aleph\|}{(2\pi\sigma^2)^{r/2}}\,\|\xi-\xi'\|\,
                                    \left(\|\xi''\|+\|\aleph\|\epsilon\right)\;\rme^{-\left(\|\xi''\|-\|\aleph\|\epsilon\right)_+^2/2}
\;.
\end{align*}
Then we get that for all $\xi,\xi'\in\rset^r$ such that
$\|\xi-\xi'\|\leq \epsilon$,
\begin{align*}
\tvdist{L(\xi,\cdot)}{L(\xi',\cdot)}=\frac12 \int_{\rset^r} |\ell(\xi,\xi'')-\ell(\xi',\xi'')|\; \rmd \xi'' \leq M_0 \|\xi-\xi' \|  \eqsp, 
\end{align*}
for some positive constant $M_0$ only depending on $\|\aleph\|$, $\sigma$ and $\epsilon$.
Finally, for all $\xi,\xi'\in\rset^r$,
$$
\tvdist{L(\xi,\cdot)}{L(\xi',\cdot)}\leq \lr{M_0  \1_{\|\xi-\xi' \| \leq \epsilon}+\epsilon^{-1} \1_{\|\xi-\xi' \| > \epsilon}}\|\xi-\xi' \|\eqsp, 
$$
and Assumption \ref{item:assum:parx:vi} holds.

Let us now check
Assumption~\ref{assum:thm:parx:equiv:conv:mle:gen:od:L}. Assumptions~\ref{assum:thm:parx:equiv:conv:mle:gen:od:L}\ref{assum:thm:parx:equiv:conv:mle:gen:od-1}
and~\ref{assum:thm:parx:equiv:conv:mle:gen:od:L}\ref{assum:thm:parx:equiv:conv:mle:gen:od-2}
are immediate. For all $(\xi,\xi',\xi'')\in\rset^{3r}$, we have
\begin{align*}
  \left|\ln\frac{\ell(\xi;\xi'')}{\ell(\xi';\xi'')}\right|
  &=\frac1{2\sigma^2}\left|\|\xi''-\aleph\xi\|^2-\|\xi''-\aleph\xi'\|^2\right|\\
  &\leq\frac{\|\aleph\|}{2\sigma^2}\,\|\xi-\xi' \|\, (1+\|\aleph\|\,(\|\xi\|+\|\xi'\|))\, (1+\|\xi''\|)\\
  &\leq \mathrm{h}_\X(\|\xi-\xi'\|) \,
\rme^{C_\X \, \left(1+\|\xi\|\vee\|\xi'\right)}\;\bar\phi_\X(\xi'')\;.
\end{align*}
Thus
\ref{assum:thm:parx:equiv:conv:mle:gen:od:L}\ref{assum:thm:parx:ineq:loggg:gen:od}
holds with $\mathrm{h}_\X(u)=({\|\aleph\|}/{2\sigma^2})\,u$,
$C_\X=\|\aleph\|$ and
$\bar\phi_\X(\xi)=1+\|\xi\|$. Then, \ref{assum:thm:parx:equiv:conv:mle:gen:od:L}\ref{assum:thm:parx:barphi:bar:phi:gen:od}
and
\ref{assum:thm:parx:equiv:conv:mle:gen:od:L}\ref{assum:thm:parx:H:gen:od}
follow from these choices of $\mathrm{h}_\X$ and $\bar\phi_\X$. We
also get
\ref{assum:thm:parx:equiv:conv:mle:gen:od:L}\ref{assum:thm:parx:22or23-barphi-V:gen:od}
since $V_\X(\xi)=\rme^{\lambda\|\xi\|}$ for some $\lambda>0$.

Finally~\ref{item:PARX--non-degenerateident-G-X} immediately holds
since $\nu_\X$ is the Lebesgue measure on $\rset^r$.

Having
shown~\ref{hyp:parx-model-L-ell}--\ref{item:PARX--non-degenerateident-G-X}
for covariates satisfying the
dynamics~(\ref{eq:linear-gaussian-covariate-dynamics}),
\autoref{theo:ergo-convergence:parx} applies in this case under the sole
condition~(\ref{eq:cond:ergo-convergence:parx:one}). In comparison,
checking the assumptions for ergodicity in \cite{AGOSTO2016640} require
the additional assumption
$$
\sum_{i=1}^{p}a_i+\sum_{i=1}^q b_i < 1-
\left(\left(1-\sum_{i=1}^pa_i\right)(\rho(\aleph)-b_1)\right)_+ \;,
$$
see their Assumption~3(iii).

\section{General Results}\label{sec:main-result:gen:od}
\subsection{Preliminaries}
\label{sec:preliminaries}
In the well-specified setting, a general result on the consistency of the MLE
for a class of first-order ODMs has been obtained in
\cite{dou:kou:mou:2013}. Let us briefly describe the approach used to establish
the convergence of the MLE $\mlY{\initmle,n}$ in this reference and in the
present contribution for higher order ODMs. Let $\thv\in\Theta$
denote the true parameter. The consistency of the MLE is obtained through
the following steps.
\begin{enumerate}[label=\textbf{Step}~\arabic*]
\item\label{item:ergo:step} Find sufficient conditions for the ergodic
  property~\ref{assum:gen:identif:unique:pi:gen} of the model. Then the convergence of the MLE to
$\thv$ is studied under $\tilde{\PP}^\thv$ as defined in \autoref{def:equi:theta:gen:od}.
\item\label{item:likelihood:step} Establish that,
as the number of observations $n\to\infty$, the normalized log-likelihood
$\mathsf{L}_{\initmle,n}^\theta$ as defined in \eqref{eq:defi:lkdM:gen:od}, for some
well-chosen $\initmle\in\Xset^p$, can be approximated by
$$
n^{-1} \sum_{k=1}^n \ln p^{\theta}(Y_k|Y_{-\infty:k-1}),
$$
where $p^{\theta}(\cdot|\cdot)$ is a $\tilde\PP^\thv\as$ finite real-valued
measurable function defined on $(\Yset^\zset,\Ysigma^{\otimes\zset})$.  To
define $p^{\theta}(\cdot|\cdot)$, we set, for all
$\chunk{y}{-\infty}{0}\in\Yset^{\zsetn}$ and $y\in \Yset$, whenever the following limit is well
defined,
  \begin{equation}
    \label{eq:def-p-theta-neq-thv:gen:od}
    p^{\theta}\left(y\,|\,\chunk{y}{-\infty}{0}\right) =
    \lim_{m\to\infty}
    g^\theta\left(\f[\theta]{\chunk{y}{-m}{0}}(\initmle);y\right)\;.
  \end{equation}
\item\label{item:likelihood:conv:step} By~\ref{assum:gen:identif:unique:pi:gen}, the observed process $\{Y_k:k\in\zset\}$ is ergodic under
$\tilde\PP^{\thv}$ and provided that
$$
\tilde\PE^\thv\left[\lnp p^{\theta}(Y_1|Y_{-\infty:0})\right]<\infty\eqsp,
$$
 it then follows that
$$
\lim_{n\to\infty}\mathsf{L}_{\initmle,n}^\theta=\tilde\PE^\thv\left[\ln p^{\theta}(Y_1|Y_{-\infty:0})\right]\eqsp,\quad\tilde\PP^\thv\as
$$
\item\label{item:mle:conv:step} Using an additional argument (similar to that in \cite{pfanzagl:1969}),
  deduce that the MLE $\mlY{\initmle,n}$ defined
  by~(\ref{eq:defi:mle-od:gen}) eventually lies in any given neighborhood of
  the set
\begin{equation}\label{eq:def-Theta-star-set:gen:od}
\Theta_\star=\argmax_{\theta\in\Theta}\tilde{\PE}^{\thv}\left[\ln
  p^{\theta}(Y_1|\chunk{Y}{-\infty}{0})\right],
\end{equation}
which only depends on $\thv$, establishing that
\begin{equation}\label{eq:strong-consistency-prelim:gen:od}
\lim_{n\to\infty}\met(\mlY{\initmle,n},\Theta_\star)=0,\quad \tilde\PP^{\theta_\star}\as,
\end{equation}
where $\met$ is the metric endowing the parameter space $\Theta$.
\item\label{item:maximizing:set:equivalence:step} Establish that $\Theta_\star$ defined
  in~(\ref{eq:def-Theta-star-set:gen:od}) reduces to the equivalent class
  $[\thv]$ of \autoref{def:equi:theta:gen:od}. The
  convergence~(\ref{eq:strong-consistency-prelim:gen:od}) is then called the
  \emph{equivalence-class consistency} of the MLE.
\item\label{item:identifiability:step}  Establish that   $[\thv]$  reduces to the singleton $\{\thv\}$. The
  convergence~(\ref{eq:strong-consistency-prelim:gen:od}) is then called the
  \emph{strong consistency} of the MLE.
\end{enumerate}
In \cite{douc2015handy}, we provided easy-to-check conditions on first order
ODMs for obtaining \ref{item:ergo:step} to \ref{item:mle:conv:step}.
See \cite[Theorem~2]{douc2015handy} for \ref{item:ergo:step}, and
\cite[Theorem~1]{douc2015handy} for the following steps.
In \cite{doumonrou2016maximizing}, we proved a general result for partially observed
Markov chains, which include first order ODMs, in order to
get \ref{item:maximizing:set:equivalence:step}, see Theorem~1 in this
reference. Finally \ref{item:identifiability:step} is often carried out using
particular means adapted to the precise considered model.

We present in \autoref{sec:ergodicity:gen:od} the conditions that we
use to prove ergodicity (\ref{item:ergo:step}) and, in
\autoref{sec:convergence:mle}, we adapt the conditions already used in
\cite{douc2015handy,doumonrou2016maximizing} to carry
out~\ref{item:likelihood:step}
to~\ref{item:maximizing:set:equivalence:step} for first order model to
higher order ODMs.

Using the embedding described in \autoref{sec:embedding-into-an11}, all the
steps from \ref{item:ergo:step} to \ref{item:maximizing:set:equivalence:step}
can in principle be obtained by applying the existing results to the first
order ODMs in which the original higher order model is embedded.
This approach is indeed successful, up to some straightforward adaptation, for
\ref{item:likelihood:step} to
\ref{item:maximizing:set:equivalence:step}. Ergodicity in \ref{item:ergo:step}
requires a deeper analysis that constitutes the main part of this
contribution. As for~\ref{item:identifiability:step}, it is treated in
\cite{douc-roueff-sim_ident-genod2020}.

\subsection{Ergodicity}\label{sec:ergodicity:gen:od}

In this section, we provide conditions that yield stationarity and ergodicity
of the Markov chain $\nsequence{(Z_k,Y_k)}{k\in\zsetp}$, that is, we
check~\ref{assum:gen:identif:unique:pi:gen}
and~\ref{ass:21-lyapunov:gen:od}. We will set $\theta$ to be an arbitrary value
in $\Theta$ and since this is a ``for all $\theta$ (...)'' condition, to save
space and alleviate the notational burden, we will drop the superscript
$\theta$ from, for example, $G^\theta$ and $\psi^\theta$ and
respectively write $G$ and $\psi$, instead.

Ergodicity of Markov chains is  usually studied using
$\varphi$-irreducibility. This approach is well known to be quite efficient
when dealing with fully dominated models; see \cite{meyn:tweedie:2009}. It is
not at all the same picture for integer-valued observation-driven models, where
other tools need to be invoked;
see~\cite{fokianos:tjostheim:2011,dou:kou:mou:2013,douc2015handy} for
ODMs(1,1). Here we extend these results for general order
ODMs$(p,q)$. Let us now introduce our list of assumptions. They
will be further commented after they are all listed.

We first need some metric on the space $\Zset$
and assume the following.
\begin{hyp}{A}
\item\label{item:CCMShyp} The  $\sigma$-fields $\Xsigma$ and $\Usigma$
are Borel ones, respectively associated to $(\Xset,\Xmet)$ and
$(\Uset,\Umet)$, both assumed to be complete and separable metric spaces.
\end{hyp}
For an LODM($p,q$), the following condition,  often referred
to as the \emph{invertibility} condition, see \cite{straumann06}, is
classically assumed. 
\begin{hyp}{I}
\item\label{item:lodm-ident-cond-cond-dens-def} For all
  $\theta\in\Theta$, we have
  $\chunk{A}{1}{p}(\theta) \in \mathcal{S}_p$,
\end{hyp}
where $\mathcal{S}_p$ is defined in~(\ref{eq:char-poly-roots-ok}). For
an ODM$(p,q)$ with a possibly non-linear link
function,~\ref{item:lodm-ident-cond-cond-dens-def} is replaced by a
uniform contracting condition on the iterates of the link function,
see~\ref{assum:bound:rho:gen} below. In order to write this condition
in this more general case, recall that any finite $\Yset$-valued
sequence $y$, $\f{y}$ is defined by~(\ref{eq:lkd:psi:n:def:rec}) with
the recursion~(\ref{eq:pq-order-rec-equation-with-u}).  Next, we 
may rewrite these iterates directly in terms of
$\chunk u0{(n-1)}$ instead of $\chunk y0{(n-1)}$. Namely, we can
define
\begin{align}
\label{eq:lkd:psi:n:def:rec:withu}
&  \tf{\chunk{u}0{(n-1)}}(z)\eqdef x_{n} \;,\text{ with
  $x_n$ defined by \eqref{eq:pq-order-rec-equation-with-u}}
\end{align}
so that $\f{\chunk{y}0{(n-1)}}(z)=\tf{\Upsilon^{\otimes
    n}(\chunk{y}0{(n-1)})}(z)$ for all $z\in\Zset$ and
$\chunk{y}0{(n-1)}\in\Yset^n$.
Now define, for all $n\in\zsetpnz$, the Lipschitz constant for $\tf{u}$, uniform over
$u\in\Uset^{n}$,
  \begin{equation}
    \label{eq:Lip:constant:n:def}
\mathrm{Lip}_n^\theta=\sup\set{\frac{\Xmet(\tf{u}(z),\tf{u}(z'))}{\Zmet(z,z')}}{(z,z',u)\in\Zset^{2}\times\Uset^{n}}\;,
\end{equation}
where we set, for all $v\in\Zset^2$,
\begin{equation}
  \label{eq:def:Zmet}
  \Zmet(v)=\left(\max_{1\leq k\leq p}\Xmet\circ\Proj
    k^{\otimes2}(v)\right)\,\bigvee\,
  \left(\max_{p< k<p+q}\Umet\circ\Proj k^{\otimes2}(v)\right)\;.
\end{equation}
We use the following assumption on a general link function.
\begin{hyp}{A}
\item\label{assum:bound:rho:gen}  For all $\theta\in\Theta$, we have $\mathrm{Lip}_1^\theta<\infty$ and
  $\mathrm{Lip}_n^\theta\to0$ as $n\to\infty$.
\end{hyp}
The following assumption is mainly related to the observation kernel
$G$ and relies on the metrics introduced in~\ref{item:CCMShyp} and on the iterates of the link functions defined
in~(\ref{eq:lkd:psi:n:def:rec:withu}).
\begin{hyp}{A}
\item  \label{assum:weak:feller:X}
The space $(\Xset,\Xmet)$ is locally compact and if $q>1$, so is $(\Uset,\Umet)$.
 For all $x\in\Xset$, there exists $\delta>0$ such that
  \begin{equation}
    \label{eq:uniform-dom-conv}
    \int\sup\set{g(x';y)}{x'\in\Xset\,,\,\Xmet(x',x)<\delta}\;\nu(\rmd y)<\infty\;.
  \end{equation}
  Moreover, one of the two following assertions
  hold.
  \begin{enumerate}[label=(\alph*)]
  \item \label{assum:feller:atom} The kernel $G$ is strong Feller.
  \item \label{assum:feller:cont} The kernel $G$ is weak Feller and the
    function $u\mapsto\tf[]u(z)$ defined in~(\ref{eq:lkd:psi:n:def:rec:withu}) is
  continuous on $\Uset$ for all $z\in\Zset$.
  \end{enumerate}
\end{hyp}
The definitions of weak and strong Feller
in~\ref{assum:feller:atom} and~\ref{assum:feller:cont} correspond to
Feller and strong Feller of \cite[Defintion~12.1.1]{douc2018markov}.
Next, we consider a classical drift condition used for showing the existence of
an invariant probability distribution.
\begin{hyp}{A}
\item \label{assum:V:gen:X} There exist measurable functions
  $\VX:\Xset\to\rsetp$ and $\VU:\Uset\to\rsetp$ such that, setting $\VY=\VU\circ\Upsilon$, $G\VY\lesssim\VX$,
  $\{\VX\leq M\}$ is compact for any $M>0$, and so is $\{\VY\leq M\}$ if $q>1$,
  and
  \begin{equation}
    \label{eq:ergoVX}
    \lim_{n\to\infty}\lim_{M\to\infty}\sup_{z\in\Zset}\frac{\PE_z\left[\VX(X_n)\right]}{M+V(z)} =0\;,
  \end{equation}
  where we defined
  \begin{equation}
    \label{eq:VVX:def}
    V(z)=\max_{\stackrel{1\leq k\leq p}{{p<\ell<p+q}}}\left\{\VX(\Projarg kz),\frac{\VU(\Projarg {\ell}z)}{\vnorm[\VX]{G\VY}}\right\}\;.
  \end{equation}
\end{hyp}
The following condition is used to show the existence of a reachable
point.
\begin{hyp}{A}
\item \label{assum:reaching:mild} The conditional density of $G$ with
  respect to $\nu$ satisfies,  for all
$(x,y) \in \Xset \times \Yset$,
\begin{equation}\label{eq:ass:g}
g(x;y)>0\;,
\end{equation}
and one of the two following assertions
  hold.
  \begin{enumerate}[label=(\alph*)]
  \item \label{assum:reaching:atom}   There exists $y_0\in\Yset$ such that
    $\nu(\{y_0\})>0$.
  \item \label{assum:reaching:cont}  The function $(x,u)\mapsto\tilde\psi_u(x)$ is
  continuous on $\Xset^p\times\Uset^{q}$.
  \end{enumerate}
\end{hyp}
The last assumption is used to show the uniqueness of the invariant probability
measure, through a coupling argument.  It requires the following definition,
used in a coupling argument.  Under  \ref{assum:alpha-phi:gen}-\ref{item:alpha-phi:3:gen} (defined below), for any initial distribution $\Xinit$ on
$(\Zset^2,\Zsigma^{\otimes2})$, let $\hat\PE_\Xinit$ denote the expectation
(operator) associated to the distribution of
$\nsequence{X_k,X'_k,U_{k'},U'_{k'}}{k>-p,\ k'>-q}$ satisfying 
$(\chunk X{(-p+1)}{0},\chunk U{(-q+1)}{-1},\chunk {X'}{(-p+1)}{0},\chunk
{U'}{(-q+1)}{-1})\sim\Xinit$ and, for all $k\in\zsetp$,
  \begin{equation}
    \label{eq:Rhatdef:XXY}
    \begin{split}
&      Y_{k}|\mathcal{F}'_k\sim \underG(X_{k},X'_{k};\cdot)
\quad\text{and}\quad Y'_k=Y_{k}\;,\\
&      X_{k+1}=\tilde \psi_{\chunk U{(k-q+1)}k}(\chunk X{(k-p+1)}k)\quad\text{and}\quad U_k=\Upsilon(Y_{k})\;,\\
&      X'_{k+1}=\tilde \psi_{\chunk{U'}{(k-q+1)}k}(\chunk {X'}{(k-p+1)}k)\quad\text{and}\quad U'_k=\Upsilon(Y'_{k})\;.
    \end{split}
  \end{equation}
where $\mathcal{F}'_k=\sigma\left(\chunk
  X{(-p+1)}{k},\chunk{U}{(-q+1)}{(k-1)},\chunk {X'}{(-p+1)}{k},\chunk
  {U'}{(-q+1)}{(k-1)}\right)$.
\begin{hyp}{A}
\item \label{assum:alpha-phi:gen} There exist measurable functions
  $\alpha:\Xset^2\to[0,1]$,
  $\WX: \Xset^2 \to [1,\infty)$, $\WU:\Uset\to\rsetp$ and a Markov kernel $\underG$ on $\Xset^2 \times \Ysigma$, dominated by $\nu$ with kernel density $\underg(x,x';y)$ such that, setting
  $\WY=\WU\circ\Upsilon$, we have $\underG\WY\lesssim\WX$ and the three following
  assertions hold.
  \begin{enumerate}[label=(\roman*),series=ergosubassump]
  \item\label{item:alpha-phi:3:gen} For all $(x,x')\in\Xset^2$ and $y\in\Yset$,
\begin{equation}\label{eq:con:min-g-phi-gen:od}
\min\left\{g(x;y),g(x';y)\right\}\ge\alpha(x,x')\underg\left(x,x';y\right)\eqsp.
\end{equation}
\item  \label{assum:definition-gamma-x:gen} The function $\WX$ is symmetric on
  $\Xset^2$,
  $\WX(x,\cdot)$ is locally bounded for all $x \in \Xset$, and $\WU$ is locally bounded on $\Uset$.
\item  \label{assum:hyp:1-alpha:W:gen}
We have $1-\alpha\leq \Xmet\times \WX$
on $\Xset^2$.
\end{enumerate}
And, defining,  for all $v=(z,z')\in\Zset^2$,
  \begin{equation}
    \label{eq:Wbar:def}
    W(v)=\max\set{\WX\circ\Proj k^{\otimes2}(v),\,\frac{\WU(\Projarg {\ell}{\tilde
          z})}{\vnorm[\WX]{\underG\WY}}}
    {\substack{1\leq k\leq p\\\tilde z\in\{z,z'\}\\p<\ell<p+q}}\;,
  \end{equation}
 one of the two following assertions holds.
\begin{enumerate}[resume*=ergosubassump]
\item  \label{assum:driftCond:W:new:gen:od}
 $\displaystyle \lim_{\zeta\to\infty}\limsup_{n\to\infty}\frac1n\ \ln\,\sup_{v\in\Zset^2}\frac{\hat\PE_v\left[\WX(X_n,X_n')\right]}{W^\zeta(v)}\leq0.$
\item \label{assum:driftCond:W:new:gen:od:new}
  $\displaystyle\lim_{n\to\infty}\lim_{M\to\infty}\sup_{v\in\Zset^2}\frac{\hat\PE_v\left[\WX(X_n,X'_n)\right]}{M+W(z)}
  =0$ and, for all $r=1,2,\dots$, there exists $\tau\geq1$ such that
  $\displaystyle\sup_{v\in
    \Zset^2}\hat\PE_{v}\left[\WX(X_r,X'_r)\right]/W^{\tau}(v)<\infty$.
\end{enumerate}
\end{hyp}
\begin{remark}
  \label{rem:ergo:ass}
  Let us comment briefly on these assumptions.
  \begin{enumerate}[label=(\arabic*)]
  \item \label{item:ergo:ass:0} In many examples, \eqref{eq:con:min-g-phi-gen:od} is satisfied with
  \begin{equation} \label{eq:def:phi}
  \underg(x,x';y)=g(\phi(x,x');y)
  \end{equation}
  where $\phi$ is a measurable function from $\Xset^2$ to $\Xset$, in which case, $\underG \WY$ should be replaced by $G \WY \circ \phi$ in \ref{assum:alpha-phi:gen}.
  \item\label{item:ergo:ass:1}   If $\tilde\psi^\theta$ is of the
  form~(\ref{eq:tilde-psi:affine:vector:case})  with $p'=q'=1$, then
  \ref{assum:bound:rho:gen} is equivalent
  to \ref{item:lodm-ident-cond-cond-dens-def}.
\item If $q=1$, the terms depending on $\ell$  both in~(\ref{eq:VVX:def})
  and~(\ref{eq:Wbar:def}) vanish. We can take $\VU=\WU=0$ without loss of generality in
  this case.
\item \label{item:rem:weak-feller-argument}
  Recall that a kernel is strong (resp. weak) Feller if it maps any bounded measurable
  (resp. bounded continuous) function to a bounded
  continuous function. By Scheffé's lemma, a sufficient condition for $G$ to be
  weak Feller is to have that $x\mapsto g(x;y)$ is continuous on $\Xset$ for
  all $y\in\Yset$. But then~(\ref{eq:uniform-dom-conv}) gives that $G$ is
  also strong Feller by dominated convergence.
\item \label{item:rem:uniform-dom-conv} Note that
  Condition~(\ref{eq:uniform-dom-conv}) holds when $G(x;\cdot)$ is taken among
  an exponential family with natural parameter continuously depending on $x$
  and valued within the open set of natural parameters, in which case $G$ is
  also strong Feller. We are in this
  situation for both Examples~\ref{expl:log-lin-poi}
  and~\ref{exmpl:nbin-garch:gen}.
\end{enumerate}
\end{remark}
We can now state the main ergodicity result.
\begin{theorem}\label{thm:ergodicity:gen}
  Let $\PP_z$ be defined as $\PP_z^\theta$ in \autoref{def:obs-driv-gen}.
  Conditions~\ref{item:CCMShyp},
  \ref{assum:bound:rho:gen},~\ref{assum:weak:feller:X},~\ref{assum:V:gen:X},~\ref{assum:reaching:mild}
  and~\ref{assum:alpha-phi:gen} imply that there exists a unique initial
  distribution $\pi$ which makes $\PP_\pi$ shift invariant. Moreover it
  satisfies $\PE_\pi[\VX(X_0)]<\infty$. Hence, provided that these assumptions
  hold at each $\theta\in\Theta$, they
  imply~\ref{assum:gen:identif:unique:pi:gen} and~\ref{ass:21-lyapunov:gen:od}.
\end{theorem}
For convenience, we postpone this proof to
\autoref{sec:ergodicity-proofs-gen:od}.

The following lemma provides a general way for constructing the instrumental
functions $\alpha$ and $\phi$ that appear in~\ref{assum:alpha-phi:gen} and in \autoref{rem:ergo:ass}-\ref{item:ergo:ass:0}. The
proof can be easily adapted from
\cite[Lemma~1]{douc2015handy} and is thus omitted.
\begin{lemma}\label{lem:det:alpha:gen:od}
  Suppose that $\Xset=\Cset^\Sset$ for some measurable space $(\Sset, \Ssigma)$
  and $\Cset\subseteq\rset$. Thus for all $x\in\Xset$, we write $x=(x_s)_{s\in
    \Sset}$, where $x_s\in\Cset$ for all $s\in\Sset$. Suppose moreover that for
  all $x=(x_s)_{s\in\Sset}\in\Xset$, we can express the conditional density
  $g(x;\cdot)$ as a mixture of densities of the form $j(x_s)h(x_s;\cdot)$ over
  $s\in\Sset$.  This means that for all $t\in\Cset$, $y\mapsto j(t)h(t;y)$ is a
  density with respect to $\nu$ and there exists a probability measure
  $\mu$ on $(\Sset, \Ssigma)$ such that
  \begin{equation}
    \label{eq:def-g-j-h-mu:gen:od}
g(x;y)=\int_\Sset {j(x_s)h(x_s;y)\mu(\rmd s)},\quad y\in\Yset\;.
\end{equation}
We moreover assume that $h$ takes nonnegative values and
that one of the two following assumptions holds.
\begin{hyp}{H'}
\item\label{item:F1:gen} For all $y\in\Yset$, the function $h(\cdot;y): t\mapsto h(t;y)$ is nondecreasing.
\item\label{item:F2:gen} For all $y\in\Yset$, the function $h(\cdot;y): t\mapsto h(t;y)$ is nonincreasing.
\end{hyp}
For all $x,x'\in\Xset^\Sset$, we denote $x\wedge x':=(x_s\wedge x_s')_{s\in \Sset}$ and $x\vee
x':=(x_s\vee x_s')_{s\in \Sset}$ and we define
$$
\begin{cases}
\displaystyle\alpha(x,x')=\inf_{s\in \Sset}{\left\{\frac{j(x_s\vee x'_s)}{j(x_s\wedge
      x'_s)}\right\}} \quad\text{and}\quad\phi(x,x')=x\wedge x'
&\text{ under\emph{~\ref{item:F1:gen}}};\\
\displaystyle\alpha(x,x')=\inf_{s\in \Sset}{\left\{\frac{j(x_s\wedge x'_s)}{j(x_s\vee x'_s)}\right\}}
\quad\text{and}\quad\phi(x,x')=x\vee x'
&\text{ under\emph{~\ref{item:F2:gen}}}.
\end{cases}
$$
Then $\alpha$ and $\phi$ defined above
satisfy~\ref{assum:alpha-phi:gen}\ref{item:alpha-phi:3:gen} and \eqref{eq:def:phi}.
\end{lemma}
\subsection{Convergence of the MLE}\label{sec:convergence:mle}
Once the ergodicity of the model is established, one can derive the asymptotic
behavior of the MLE, provided some regularity and moment condition holds
for going through \ref{item:likelihood:step} to
\ref{item:maximizing:set:equivalence:step}, as described in
\autoref{sec:preliminaries}. These steps are carried out using
 \cite[Theorem~1]{douc2015handy} and \cite[Theorem~3]{doumonrou2016maximizing},
 written for general ODMs(1,1). The adaptation to higher order
ODMs$(p,q)$ will follow easily from the embedding of
\autoref{sec:embedding-into-an11}. We consider the following assumptions, the
last of which uses $\VX$ as introduced in \autoref{def:equi:theta:gen:od}
under Assumptions~\ref{assum:gen:identif:unique:pi:gen} and~\ref{ass:21-lyapunov:gen:od}.
\begin{hyp}{B}
\item\label{assum:continuity:Y:g-theta:gen:od}
For all $y \in\Yset$, the function $(\theta,x) \mapsto g^\theta(x;y)$ is
continuous on $\Theta\times\Xset$.
\item\label{assum:continuity:Y:phi-theta:gen:od}
For all $y\in\Yset$, the function $(\theta,z)
\mapsto \f{y}(z)$ is continuous  on $\Theta\times\Zset$.
\item \label{assum:practical-cond:O:gen:od} There exist
  $\initmlex_1\in\Xset$,
  $\initmleu_1\in\Uset$, a closed set $\Xset_1\subseteq\Xset$,
  $C\geq0$, $\mathrm{h}:\rset_+\to\rset_+$ and a
  measurable function $\bar\phi:\Yset\to\rsetp$ such that the following
  assertions hold with
  $\initmle=(\initmlex_1,\dots,\initmlex_1,\initmleu_1,\dots,\initmleu_1)\in\Zset$.
\begin{enumerate}[label=(\roman*)]
\item \label{assum:exist:practical-cond:subX:gen:od}
For all
  $\theta\in\Theta$ and $(z,y)\in\Zset \times \Yset$, $\f{y}(z)\in\Xset_1$.
\item\label{assum:momentCond:gen:od} $\displaystyle\sup_{(\theta,x,y) \in\Theta\times\Xset_1\times\Yset}g^\theta(x;y)<\infty$. 
\item\label{item:barphi:bar:phi:gen:od} For all
  $y\in\Yset$ and  $\theta\in\Theta$,
  $\displaystyle\Xmet\left(\initmlex_1,\f{y}(\initmle)\right)\vee\Umet(\Upsilon(y),\initmleu_1)\leq\bar\phi(y)$.
\item\label{item:ineq:loggg:gen:od} For all $\theta\in\Theta$ and $(x,x',y)\in\Xset_1\times\Xset_1\times\Yset$,
\begin{equation}\label{eq:ineq:loggg:gen:od}
\left|\ln\frac{g^\theta(x;y)}{g^\theta (x';y)}\right|\leq \mathrm{h}(\Xmet(x,x')) \,
\rme^{C \, \left(\Xmet(\initmlex_1,x)\vee \Xmet(\initmlex_1,x')\right)}\;\bar\phi(y),
\end{equation}
\item \label{item:H:gen:od} $\mathrm{h}(u)=O(u)$ as $u\to0$.
\item \label{item:22or23-barphi-V:gen:od} If $C=0$, then, for all $\theta\in\Theta$,
 $G^\theta\lnp \bar{\phi} \lesssim \VX$. Otherwise, for all $\theta\in\Theta$,
 $G^\theta\bar{\phi}\lesssim \VX$.
\end{enumerate}
\end{hyp}
\begin{remark}
  \label{rem:linear-case:conv:assump}
  If we consider a VLODM as in \autoref{def:obs-driv-gen},
  Condition~\ref{assum:continuity:Y:phi-theta:gen:od} is obvious and
  \ref{assum:practical-cond:O:gen:od}~\ref{item:barphi:bar:phi:gen:od} reduces
  to impose that $\bar\phi(y)\geq A+B\,|\Upsilon(y)|$ for some non-negative
  constants $A$ and $B$ only depending on $\initmlex_1$, on (the compact
  set) $\Theta$ and on the choice of the norm $|\cdot|$ on $\Uset=\rset^{q'}$.
\end{remark}
\begin{remark}
  \label{rem:discret-case:conv:assump}
In the case where the observations are discrete, one usually take $\nu$ to be
the counting measure on the at most countable space $\Yset$. In this case,
$g^\theta(x;y)\in[0,1]$ for all $\theta,x$ and $y$ and
Condition~\ref{assum:practical-cond:O:gen:od}\ref{assum:momentCond:gen:od}
trivially holds whatever $\Xset_1$ is.
\end{remark}
We have the following result, whose proof is postponed
to  \autoref{sec:proof-thm:conv}.
\begin{theorem}\label{thm:convergence-main:gen:od}
    Consider an ODM$(p,q)$ for some $p,q\geq1$
  satisfying~\ref{assum:bound:rho:gen}.
  Assume that~\ref{assum:gen:identif:unique:pi:gen},
  \ref{ass:21-lyapunov:gen:od}, \ref{assum:continuity:Y:g-theta:gen:od},
  \ref{assum:continuity:Y:phi-theta:gen:od} and
  \ref{assum:practical-cond:O:gen:od} hold.  Then the MLE
  $\mlY{\initmle,n}$ defined by ~(\ref{eq:defi:mle-od:gen}) is
  equivalence-class consistent, that is, the
  convergence~\eqref{eq:equi:class:consistency:od} holds for any $\thv\in\Theta$.
\end{theorem}

\section{Postponed Proofs}\label{sec:append:proofs:gen:od}

\subsection{Embedding into an observation-driven model of order $(1,1)$}
\label{sec:embedding-into-an11}

A simplifying and unifying step is to embed the general order case into the
order $(1,1)$ by augmenting the state space. Consider an ODM as
in \autoref{def:obs-driv-gen}.  For all $u\in\Uset$ and $y\in\Yset$, we
denote by $\tilde\Psi^\theta_u$ and  $\tilde\Psi^\theta_y$ the two
$\Zset\to\Zset$ mappings defined by
\begin{align}\label{eq:def:tpsi:gen:od}
&  \tilde{\Psi}^\theta_{u}:z=\chunk z1{(p+q-1)}\mapsto
  \begin{cases}
  \left(\chunk z2p,
  \tf{u}(z), \chunk z{(p+2)}{(p+q-1)},
  u\right)&\text{ if $q>1$}\\
  \left(\chunk z2p,
  \tf{u}(z)\right)&\text{ if $q=1$}\;,
  \end{cases}\\
\label{eq:def:psi:gen:od}
&  \Psi^\theta_{y}=  \tilde{\Psi}^\theta_{\Upsilon(y)}\;.
\end{align}
We further denote the successive composition of  $\Psi^\theta_{y_0}$,
$\Psi^\theta_{y_1}$, ..., and $\Psi^\theta_{y_k}$ by
\begin{equation}
\label{eq:notationItere:f:cl:gen}
\F{\chunk{y}{0}{k}}=\Psi^\theta_{y_k} \circ \Psi^\theta_{y_{k-1}} \circ \dots \circ \Psi^\theta_{y_0}\,.
\end{equation}
Note in particular that
$\f{\chunk{y}{0}{k}}$ defined by~(\ref{eq:lkd:psi:n:def:rec}) with the
recursion~(\ref{eq:pq-order-rec-equation-with-u}) can
be written as
\begin{equation}
\label{eq:psi:Psi:relation}
\f{\chunk{y}{0}{k}}=\Proj p\circ\F{\chunk{y}{0}{k}}\;,
\end{equation}
Conversely, we have, for
all $k\geq0$ and $\chunk{y}{0}{k}\in\Yset^{k+1}$,
\begin{equation}
  \label{eq:psi:Psi:relation:recip}
  \F{\chunk{y}{0}{k}}(z)=\left(\setvect{\f{\chunk{y}{0}{j}}(z)}{k-p<j\leq k},\chunk
  u{(k-q+2)}{k}\right)\;,
\end{equation}
where we set $u_j=\Projarg {p+q+j}z$ for $-q< j\leq -1$ and $u_j=\Upsilon(y_j)$
for $0\leq j \leq k$ and use the convention
$\f{\chunk{y}{0}{j}}(z)=\Projarg {p-j}z$ for $-p< j\leq 0$.  By letting
$Z_k=(\chunk{X}{(k-p+1)}{k},\chunk{U}{(k-q+1)}{(k-1)})$ (see
\eqref{eq:Zk:def}), Model~\eqref{eq:def:gen-ob} can be rewritten as, for all
$k\in\zsetp$,
\begin{equation}\label{eq:def:ob:standard}
\begin{split}
&Y_{k}\mid \mathcal{F}_k \sim G^\theta(\Projarg{p}{Z_{k}};\cdot)\eqsp,\\
&Z_{k+1}=\Psi^\theta_{Y_k}(Z_k)\eqsp.
\end{split}
\end{equation}
By this representation, the ODM$(p,q)$ is thus embedded in an
ODM$(1,1)$. This in principle allows us to apply the same results
obtained for the class of ODMs$(1,1)$ in \cite{douc2015handy} to the
broader class of ODMs$(p,q)$. Not all but some conditions written
above for an ODM$(p,q)$ indeed easily translate to the embedded
ODM$(1,1)$. Take for instance \ref{assum:bound:rho:gen}.
By~(\ref{eq:Lip:constant:n:def}),~(\ref{eq:def:Zmet})
and~(\ref{eq:psi:Psi:relation:recip}), we have, for all
$n\in\zsetpnz$, using the convention $\mathrm{Lip}^\theta_{m}=1$ for
$m\leq0$,
   \begin{equation}
    \label{eq:lipschitz:Psi}
  \sup_{y\in\Yset^{n},z,z'\in\Zset^2}\frac{\Zmet\left(\F{y}(z),\F{y}(z')\right)}{\Zmet(v)}\leq
\1_{\{n< q\}}\vee\left(  \max_{0\leq j< p}\mathrm{Lip}_{n-j}^\theta\right)\;.
\end{equation}
Hence the same assumption \ref{assum:bound:rho:gen} will hold for the embedded ODM$(1,1)$.
As an ODM$(1,1)$, the bivariate process
$\dsequence{Z}{Y}[k][\zsetp]$ is a Markov chain on the space
$(\Zset\times\Yset,\Zsigma\otimes\Ysigma)$ with transition kernel $K^\theta$
satisfying, for all $(z,y)\in\Zset\times\Yset$, $A\in\Zsigma$ and
$B\in\Ysigma$,
\begin{equation}\label{eq:def:kernel:K:gen}
K^\theta((z,y); A \times B)=\int \1_{A \times B}(\Psi_y^\theta(z),y') \;
G^\theta(\Projarg{p}z; \rmd y')\eqsp.
\end{equation}
\begin{remark}
Note that \ref{assum:gen:identif:unique:pi:gen} is equivalent to saying that the
transition kernel $K^\theta$ of the complete chain admits a unique invariant
probability measure $\pi^\theta$ on $\Zset\times\Yset$. Moreover the resulting
$\piX^\theta$ and $\piY^\theta$ can be obtained by projecting  $\pi^\theta$ on
any of its $\Xset$  component and any of its  $\Yset$ component, respectively.
\end{remark}
Note also that, by itself, the process $\{Z_k\,:\, k\in\zsetp\}$ is a Markov
chain on $(\Zset,\Zsigma)$ with transition kernel $R^\theta$ defined by setting, for
all $z\in\Zset$ and $A\in\Zsigma$,
\begin{equation}\label{eq:def:kernel:R:gen}
  R^\theta(z;A)
  =\int \1_A(\Psi^\theta_{y}(z))G^\theta(\Projarg{p}{z};\rmd y)\eqsp.
\end{equation}

\subsection{Proof of \autoref{thm:ergodicity:gen}}
\label{sec:ergodicity-proofs-gen:od}
The scheme of proof of this theorem is somewhat similar to that of
\cite[Theorem~2]{douc2015handy} which is dedicated to the ergodicity of
ODM(1,1) processes. The main difference is that we need to rely on
assumptions involving \emph{iterates} of kernels such
as~\ref{assum:weakFeller-V:gen},~\ref{assum:bound:rho:gen}
and~\ref{assum:alpha-phi:gen}\ref{assum:driftCond:W:new:gen:od} below, to be compared
with their counterparts (A-4), (A-7) and (A-8)(iv) in
\cite[Theorem~2]{douc2015handy}). Using the embedding of
\autoref{sec:embedding-into-an11}, we will use the following conditions
directly applying to the kernel $R$ of this embedding.
\begin{hyp}{E}
\item \label{assum:Ztopology} The space $(\Zset,\Zmet)$ is a locally
  compact and complete separable metric space.
\item \label{assum:weakFeller-V:gen} There exists a positive integer $m$ such
  that the Markov kernel $R^m$ is weak Feller. Moreover, there exist
  $(\lambda,\beta) \in (0,1) \times \rsetp$ and a measurable function
  $V: \Zset \to \rsetp$ such that $R^m V \leq \lambda V+\beta$ and $\{V\le M\}$
  is a compact set for any $M>0$.
\item \label{assum:reachable:gen}
The Markov kernel $R$ admits  a reachable point, that is, there exists
$z_\infty\in\Zset$ such that, for any $z\in\Zset$ and any neighborhood $\mathcal{N}$ of $z_\infty$,
$R^m(z,\mathcal{N})>0$ for at least one positive integer $m$.
\item \label{assum:asympStrgFeller:general:gen:od} There exists a Markov kernel
  $\bar R$ on $(\Zset^2\times\{0,1\},\Zsigma^{\otimes2}\otimes\mathcal{P}(\{0,1\}))$,
a Markov kernel $\hat R$ on
$(\Zset^2,\Zsigma^{\otimes2})$, measurable
  functions  $\bar\alpha:\Zset^2\to[0,1]$ and  $W: \Zset^2 \to [1,\infty)$ symmetric, and real
  numbers $(D,\zeta_1,\zeta_2,\rho) \in (\rset_+)^3 \times (0,1)$ such that for
  all $v=(z,z',u) \in \Xset^2\times\{0,1\}$ and $n\geq1$,
  \begin{align}
    \label{eq:assum:definition-gamma-x:gen:bar}
 &   1-\bar\alpha\leq\Zmet \times W\quad\text{on $\Zset^2$}\;,\\
        \label{eq:assum:hyp:1-alpha:W:gen:bar}
&\forall
    z\in\Zset,\;\exists\gamma>0,\;\sup\set{W(z,z')}{\Zmet(z,z')<\gamma}<\infty\;,\\
    \label{eq:majo:d:gen:od:coupling:theirEq3}
    &\begin{cases}
      \bar R(v;\cdot\times\Zset\times\{0,1\})=R(z,\cdot)\quad\text{and}\\
      \bar R(v;\Zset\times\cdot\times\{0,1\})=R(z',\cdot) \;,
    \end{cases}
    \\
    \label{eq:majo:d:gen:od:coupling:proba:theirEq5}
&\bar R(v;\cdot\times\{1\})=\bar\alpha(z,z')\,\hat R((z,z'),\cdot)\;,    \\
    \label{eq:majo:d:gen:od}
& {\hat R}^n((z,z'); \Zmet) \leq D \rho^n
  \Zmet(z,z') \;,\\
   \label{eq:majo:d:w:gen:od}
& {\hat R}^n((z,z'); \Zmet\times W) \leq D \rho^n
  \Zmet^{\zeta_1}(z,z')\ W^{\zeta_2}(z,z')\;.
\end{align}
\end{hyp}
Based on these conditions, we can rely on the two following results.
The existence of an invariant probability measure for $R$ is given by the following result.
\begin{lemma}\label{lem:existence:inv:R}
  Under~\ref{assum:Ztopology}  and \ref{assum:weakFeller-V:gen}, $R$
  admits an invariant distribution $\pi$; moreover, $\pi V<\infty$.
\end{lemma}
\begin{proof}
  By \cite[Theorem~2]{tweedie:1988}, Assumption
  \ref{assum:weakFeller-V:gen} implies that the transition kernel $R^m$ admits an
  invariant probability distribution denoted hereafter by $\pi_m$. Let
  $\tilde\pi$ be defined by, for all $A\in\Zsigma$,
\begin{equation*}
\tilde\pi(A)=\frac{1}{m}\sum_{k=1}^m\pi_mR^k(A)\eqsp.
\end{equation*}
Obviously, we have $\tilde\pi R=\tilde\pi$, which shows that $R$ admits an invariant probability distribution $\tilde\pi$. Now let $M>0$. Then by Jensen's inequality, we have for all $n\in\zsetp$,
\begin{align*}
\tilde\pi  (V\wedge M)&=\tilde\pi R^{nm} (V\wedge M) \leq \tilde\pi ((R^{nm} V)\wedge M)  \\
&\le\lambda^n\tilde\pi(V\wedge M)+\frac{\beta}{1-\lambda}\wedge M\eqsp.
\end{align*}
Letting $n\to\infty$, we then obtain
$\tilde\pi  (V\wedge M)\le\frac{\beta}{1-\lambda}\wedge M$.
Finally, by the monotone
convergence theorem, letting $M\to\infty$, we get $\tilde\pi V<\infty$.
\end{proof}
\begin{proposition} \label{thm:uniqueInvariantProbaMeasure:gen:od}
  Assume ~\ref{assum:Ztopology} \ref{assum:reachable:gen} and
  \ref{assum:asympStrgFeller:general:gen:od}. Then the Markov kernel $R$ admits
  at most one unique invariant  probability measure.
\end{proposition}
\begin{proof}
  This is extracted from the uniqueness part of the proof
  of~\cite[Theorem~6]{dou:kou:mou:2013}, see their Section~3. Note that our
  Condition~\ref{assum:reachable:gen} corresponds to their condition (A2) and
  our Condition~\ref{assum:asympStrgFeller:general:gen:od} to their condition
  (A3) (their $\alpha$, $Q$, $\bar Q$ and $Q^\sharp$ being our $\bar\alpha$,
  $R$, $\bar R$ and $\hat{R}$).
\end{proof}
Hence it is now clear that the conclusion of \autoref{thm:ergodicity:gen} holds
if we can apply both~\autoref{lem:existence:inv:R} and~\autoref{thm:uniqueInvariantProbaMeasure:gen:od}.
This is done according to the following successive steps.
\begin{enumerate}[label={\bf Step~\arabic*}]
\item  Under \ref{item:CCMShyp}, the
  metric~(\ref{eq:def:Zmet}) makes $(\Zset, \Zmet)$ locally compact, complete
  and separable, hence~\ref{assum:Ztopology} holds true.
\item\label{item:zero-step-ergo} Prove \ref{assum:weakFeller-V:gen}: this is
  done in \autoref{lem:lyapu:embedding} using~\ref{item:CCMShyp}, \ref{assum:weak:feller:X}, \ref{assum:V:gen:X}
  and the fact that, for all $y\in\Yset$, $\f[]{y}$ is continuous on $\Zset$,
  as a consequence of $\mathrm{Lip}_1<\infty$ in  \ref{assum:bound:rho:gen}.
\item\label{item:first-step-ergo} Prove \ref{assum:reachable:gen}: this is done in \autoref{lem:reaching},
  using \ref{item:CCMShyp}, \ref{assum:bound:rho:gen} and~\ref{assum:reaching:mild}.
\item\label{item:alphabar-step-ergo} Define $\bar\alpha$ and
  prove~(\ref{eq:assum:definition-gamma-x:gen:bar})
  and~(\ref{eq:assum:hyp:1-alpha:W:gen:bar}) in~\ref{assum:asympStrgFeller:general:gen:od} with $W$ as
  in~(\ref{eq:Wbar:def}): this directly follows from
  Conditions~\ref{assum:alpha-phi:gen}\ref{assum:definition-gamma-x:gen}
  and\ref{assum:alpha-phi:gen}\ref{assum:hyp:1-alpha:W:gen}.
\item\label{item:second-step-ergo} Provide an explicit construction for
  $\bar R$ and $\hat R$ satisfying~(\ref{eq:majo:d:gen:od:coupling:theirEq3})
  and~(\ref{eq:majo:d:gen:od:coupling:proba:theirEq5})
  in~\ref{assum:asympStrgFeller:general:gen:od}: this is done in
  \autoref{lem:coupling:con:gen:od} using
  \ref{assum:alpha-phi:gen}\ref{item:alpha-phi:3:gen};
\item\label{item:final-step-ergo} Finally, we need to establish the additional
  properties of this $\hat R$ required in~\ref{assum:asympStrgFeller:general:gen:od},
  namely,~(\ref{eq:majo:d:gen:od}) and~(\ref{eq:majo:d:w:gen:od}). This will be
  done in the final part of this section using the additional
  \autoref{lem:practical:conditions:gen:od}.
\end{enumerate}
Let us start with~\ref{item:zero-step-ergo}.
\begin{lemma}
  \label{lem:lyapu:embedding}
If for all $y\in\Yset$, $\f[]{y}$ is continuous on $\Zset$, then \ref{assum:weak:feller:X} and   \ref{assum:V:gen:X} imply~\ref{assum:weakFeller-V:gen}.
\end{lemma}
\begin{proof}
  We first show that $R$ is weak Feller, hence $R^m$ is too, for any
  $m\geq1$. Let $f:\Zset\to\rset$ be continuous and bounded. For all
  $z=(\chunk x{(-p+1)}0,\chunk u{(-q+1)}{(-1)})\in\Zset$, $Rf(z)$ is given by
  \begin{align*}
\PE_z\left[f(Z_1)\right]&=\PE_z\left[f(\chunk x{(-p+2)}0,X_1,\chunk
                          u{(-q+2)}{(-1)},\Upsilon(Y_0))\right]\\
    &=\int f(\chunk x{(-p+2)}0,\f[]{y}(z),\chunk
    u{(-q+2)}{(-1)},\Upsilon(y))\;G(x_0;\rmd y)\;.
  \end{align*}
Let us define $\tilde f:\Zset\times\Yset\to\rset$
by setting, for $y\in\Yset$ and $z$ as above,
$$
\tilde f(z,y)=f(\chunk x{(-p+2)}0,\f[]{y}(z),\chunk
u{(-q+2)}{(-1)},\Upsilon(y))\;.
$$
Further define $\tilde F:\Zset\times\Xset\to\rset$ by setting, for all
$z\in\Zset$ and $x\in\Xset$,
$$
\tilde F(z,x)=\int \tilde f(z,y)\; G(x;\rmd y)\;.
$$
Hence, with these definitions, we have, for all
$z\in\Zset$,
$Rf(z)=\tilde F(z,\Projarg pz)$, and it is now sufficient to show that $\tilde F$ is
continuous. We write, for all $z,z'\in\Zset$ and $x,x'\in\Xset$, $\tilde
F(z',x')-\tilde F(z,x)=A(z,z',x')+B(z,x,x')$ with
\begin{align*}
A(z,z',x')&=\int\left( \tilde f(z',y)-
   \tilde f(z,y)\right)\; G(x';\rmd y)\\
  B(z,x,x')&=\int \tilde
f(z,y)\; \left(G(x';\rmd y)-G(x;\rmd y)\right)
\;.
\end{align*}
Since $z\mapsto\tilde f(z,y)$ is continuous for all $y$, we have
$A(z,z',x')\to0$ as $(z',x')\to(z,x)$ by~(\ref{eq:uniform-dom-conv}) and
dominated convergence. We have $B(z,x,x')\to0$ as $x'\to x$, as a consequence
of~\ref{assum:weak:feller:X}\ref{assum:feller:atom}
or of~\ref{assum:weak:feller:X}\ref{assum:feller:cont}. Hence $\tilde F$ is
continuous and we have proved that $R^m$ is weak Feller for all
$m\in\zsetpnz$.

We now show that we can find $m\in\zsetpnz$, $\lambda\in(0,1)$ and $\beta>0$
such that $R^mV\leq\lambda V+\beta$ with $V:\Zset\to\rsetp$ defined
by~(\ref{eq:VVX:def}). We have, for all $n\geq q$,
\begin{align*}
R^nV(z)&=\PE_z\left[\max_{0\leq k< p}\VX(X_{n-k})\bigvee\max_{1\leq k<q}\frac{\VU(U_{n-k})}{\vnorm[\VX]{G\VY}}\right]\\
  &\leq\sum_{0\leq k< p}\PE_z\left[\VX(X_{n-k})\right]
    +\vnorm[\VX]{G\VY}^{-1} \ \sum_{1\leq k<q}\PE_z\left[\VY(Y_{n-k})\right]\\
    &\leq2\,\sum_{0\leq k< p\vee q}\PE_z\left[\VX(X_{n-k})\right]\;,
\end{align*}
where we used that $\VU(U_{n-k})=\VU(\Upsilon(Y_{n-k}))=\VY(Y_{n-k})$ and
$\PE_z\left[\VY(Y_{n-k})\right]=\PE_z\left[G\VY(X_{n-k})\right]\leq\vnorm[\VX]{G\VY}\PE_z\left[\VX(X_{n-k})\right]$,
which is valid for $n-k\geq0$.
Now, by~(\ref{eq:ergoVX}), for any $\lambda\in(0,1)$, we can find $m\in\zsetp$
and $M>0$ such that $\PE_z\left[\VX(X_{m-k})\right]\leq \lambda(V(z)+M)/(2(p\vee
q))$ for  $0\leq k< p\vee q$. Hence $R^mV\leq\lambda V+\beta$ for
$\beta=M\lambda$.
\end{proof}
We now proceed with~\ref{item:first-step-ergo}.
\begin{lemma}\label{lem:reaching}
\ref{item:CCMShyp}, \ref{assum:bound:rho:gen} and~\ref{assum:reaching:mild} imply~\ref{assum:reachable:gen}.
\end{lemma}
\begin{proof}
  We separate the proof in two cases: first we
  assume~\ref{assum:reaching:mild}\ref{assum:reaching:atom} and
  secondly, we
  assume~\ref{assum:reaching:mild}\ref{assum:reaching:cont}.

  \noindent\textbf{Case~1:} Assume~\ref{assum:reaching:mild}\ref{assum:reaching:atom}.
  In this case, we pick $y_0\in\Yset$ as
  in~\ref{assum:reaching:mild}\ref{assum:reaching:atom} Set $y_k=y_0$
  and $u_k=\Upsilon(y_0)$
  for all $k\in\zsetp$. By \autoref{lemma_reaching-limit-iterates}, we
  have that $\f{\chunk y0n}(z)=\tf{\chunk u0n}(z)$ converges to the same
  $x_\infty$ in $\Xset$, for all $z\in\zset$. Now set
  $z_\infty=(x_\infty,\dots,x_\infty,\Upsilon(y_0),\dots,\Upsilon(y_0))\in\Zset$. Then,
  for any $z\in\Zset$, by~(\ref{eq:psi:Psi:relation:recip}) and~(\ref{eq:def:Zmet}), for
  any $\delta>0$, there exists $m\in\zsetpnz$ such that
  $\Zmet(z_\infty,\F{\chunk y0m}(z))<\delta$ and thus
    \begin{align*}
    R^{m+1}\left(z;\set{z'\in\Zset}{\Zmet(z_\infty,z')<\delta}\right)&\geq\PP_z(Z_{m+1}=\F{\chunk y0m}(z))\\
                                                       &\geq\PP_z(Y_k=y_0,\ \forall k \in \{0,\ldots, m\})\\
      &=\prod_{k=0}^mG^\theta\left(\f{\chunk y0k}(z);\{y_0\}\right)\;.
    \end{align*}
    By~\ref{assum:reaching:mild}~\ref{assum:reaching:atom}
    and~(\ref{eq:ass:g}), we have $G^\theta(x;\{y_0\})>0$ for all
    $x\in\Xset$, and we conclude that $z_\infty$ is a reachable point.

    \noindent\textbf{Case~2:}
    Assume~\ref{assum:reaching:mild}\ref{assum:reaching:cont}. Since
  $(\Uset,\Umet)$ is assumed to be separable in~\ref{item:CCMShyp},
  there exists $u_0\in\Uset$ such that
  \begin{equation}
    \label{eq:yzero-well-chosen}
    \text{for all $\delta>0$},\quad\nu\left(\set{y\in\Yset}{\Umet(u_0,\Upsilon(y))<\delta}\right)>0\;.
  \end{equation}
  For any $u\in\Uset$ and any integers $k\leq l$, in the
  following, we denote by $\chunk{[u]}kl$ the constant sequence
  $\chunk{u}kl$ in $\Uset^{l-k+1}$ defined by
  $u_j=u$ for all $k\leq j\leq l$. By
  \autoref{lemma_reaching-limit-iterates}, we have that
  $\tf{\chunk{[u_0]}0n}(z)$ converges to the same $x_\infty$ in
  $\Xset$, for all $z\in\zset$.  And
  by~(\ref{eq:def:tpsi:gen:od}), setting
  $z_\infty=(x_\infty,\dots,x_\infty,u_0,\dots,u_0)\in\Zset$, we have
  that, for all $z\in\Zset$,
    \begin{equation}
    \label{eq:yzero-well-chosen-suite-ter}
    \lim_{m\to\infty}\tF{\chunk{[u_0]}0m}(z)=
      \lim_{m\to\infty}\underbrace{\tilde{\Psi}_{u_0}\circ\dots\circ\tilde{\Psi}{u_0}}_{m\;
        \text{times}}(z)=z_\infty\quad\text{ in
    $\Zset$.}
\end{equation}
Now pick $z\in\Zset$ and
$\delta>0$. By~(\ref{eq:yzero-well-chosen-suite-ter}), we can choose
an integer $m$ such that
$$
\Zmet(z_\infty,\tF{\chunk{[u_0]}0m}(z))<\delta/2\;.
$$
Moreover,
using~\ref{assum:reaching:mild}~\ref{assum:reaching:cont}
and~(\ref{eq:def:tpsi:gen:od}), we have that, for any $m\in\zsetp$,
$\chunk{u}0m\mapsto\tF{\chunk{u}0m}(z)$ is continuous on $\Uset^{m+1}$. Hence there
exists $\delta'>0$ such that, for all $\chunk y0m\in\Yset^{m+1}$, $\Umet(u_0,\Upsilon(y_k))<\delta'$ for
all $0\leq k\leq m$ implies
$\Zmet(\tF{\chunk{[u_0]}0m}(z),\tF{\Upsilon^{\otimes(m+1)}(\chunk{y}0m)}(z))<\delta/2$,
which with the previous display gives that
$$
\forall0\leq k\leq
m\,,\;\Umet(u_0,\Upsilon(y_k))<\delta'\Longrightarrow\Zmet(z_\infty,\tF{\Upsilon^{\otimes(m+1)}(\chunk{y}0m)}(z))<\delta\;.
$$
Thus we have
$$
R^{m+1}\left(z;\set{z}{\Zmet(z_\infty,z)<\delta}\right)
\geq\PP_z(\Umet(u_0,\Upsilon(Y_k))<\delta',\ 0\leq k\leq m)\;.
$$
Applying~(\ref{eq:def:ob:standard}) and
then~(\ref{eq:yzero-well-chosen}) with~(\ref{eq:ass:g}), we have, for all $\ell\geq1$,
$$
\PP_z(\Umet(u_0,\Upsilon(Y_{\ell}))<\delta'|\mathcal{F}_{\ell})=
G^\theta(\Projarg{p}{Z_{\ell}};\set{y}{\Umet(u_0,\Upsilon(y))<\delta'})>0\;.
$$
It follows that for all $\ell=m,m-1,\dots,1$, conditioning on
$\mathcal{F}_{\ell}$,
$\PP_z(\Umet(u_0,\Upsilon(Y_k))<\delta',\ 0\leq k\leq{\ell})=0$
implies
$\PP_z(\Umet(u_0,\Upsilon(Y_k))<\delta',\ 0\leq k\leq{\ell-1})=0$.
Since
$\PP_z(\Umet(u_0,\Upsilon(Y_{0}))<\delta')=
G^\theta(\Projarg{p}{z};\set{y}{\Umet(u_0,\Upsilon(y))<\delta'})>0$,
we conclude that
$R^{m+1}\left(z;\set{z}{\Zmet(z_\infty,z)<\delta}\right)>0$ and
$z_\infty$ is a reachable point.
\end{proof}
We now proceed with~\ref{item:alphabar-step-ergo}.
Let us define $\bar\alpha=\alpha\circ\Proj
p^{\otimes2}$. By~(\ref{eq:def:Zmet}) and~(\ref{eq:Wbar:def}), we have $W\geq\WX\circ\Proj
p^{\otimes2}$ and $\Zmet\geq\Xmet\circ\Proj
p^{\otimes2}$. Hence~(\ref{eq:assum:definition-gamma-x:gen:bar}) follows from~\ref{assum:alpha-phi:gen}\ref{assum:hyp:1-alpha:W:gen}.
Condition~(\ref{eq:assum:hyp:1-alpha:W:gen:bar}) directly follows
from~\ref{assum:alpha-phi:gen}\ref{assum:definition-gamma-x:gen} and the
definition of $W$ in~(\ref{eq:Wbar:def}).

We now proceed with~\ref{item:second-step-ergo}.
\begin{lemma}\label{lem:coupling:con:gen:od}
  Let $\alpha:\Xset^2\to[0,1]$ be a measurable function and $\underG$ be a Markov kernel on $\Xset^2\times \Ysigma$ satisfying~\ref{assum:alpha-phi:gen}\ref{item:alpha-phi:3:gen} and
  define the Markov kernel $\hat R$ on
  $(\Zset^2,\Zsigma^{\otimes2})$ by
\begin{equation}\label{eq:con:min-g-phi-gen:od:defJ}
  \hat Rf(v)=   \int_{\Yset} f\circ\Psi_y^{\otimes2}(v)\;\underG\left(\Proj p^{\otimes2}(v);\rmd y\right)\eqsp.
\end{equation}
Then one can define
  a Markov kernel $\bar R$ on
  $(\Zset^2\times\{0,1\},\Zsigma^{\otimes2}\otimes\mathcal{P}(\{0,1\}))$ which
  satisfies~(\ref{eq:majo:d:gen:od:coupling:theirEq3})
  and~(\ref{eq:majo:d:gen:od:coupling:proba:theirEq5}).
\end{lemma}
\begin{proof}
  We first define a probability kernel $\bar H$ from $\Zset^2$ to
  $\Ysigma^{\otimes 2}\otimes\mathcal{P}(\{0,1\})$ Let $(z, z')\in\Zset^2$ and
  set $x=\Projarg p{z}$ and $x'=\Projarg p{z'}$. We define
  $\bar H((z,z');\cdot)$ as the distribution of $(Y, Y',\epsilon)$ drawn as
  follows.  We first draw a random variable $\bar Y$ taking values in $\Yset$
  with distribution $\underG(x,x';\cdot)$. Then we define $(Y,Y',\epsilon)$ by
  separating the two cases, $\alpha(x,x')=1$ and $\alpha(x,x')<1$.
\begin{enumerate}[label=-]
\item
 Suppose first that $\alpha(x,x')=1$. Then
    by~\ref{assum:alpha-phi:gen}\ref{item:alpha-phi:3:gen}, we have
  $G(x;\cdot)=G(x';\cdot)=\underG(x,x';\cdot)$.
 In this case, we set $(Y,Y',\epsilon)=(\bar{Y},\bar Y,1)$.
\item
 Suppose now that $\alpha(x,x')<1$. Then,
 using~\eqref{eq:con:min-g-phi-gen:od}, the functions
$(1-\alpha({x,x'}))^{-1}\left[g(x;\cdot)-\alpha({x,x'})\underg(x,x';\cdot)\right]$
and
$(1-\alpha({x,x'}))^{-1}\left[g(x';\cdot)-\alpha({x,x'})\underg(x,x';\cdot)\right]$
are probability density functions with respect to $\nu$ and we draw
$\Lambda$ and $\Lambda'$ according to these two density functions,
respectively. We then draw $\epsilon$ in $\{0,1\}$ with
mean $\alpha(x, x')$ and, assuming $\bar Y$, $\Lambda$, $\Lambda'$ and $\epsilon$
to be independent, we set
\[
(Y,Y')=
\begin{cases} (\bar Y, \bar Y)  \quad &\text{if $\epsilon=1$} \eqsp,\\
(\Lambda, \Lambda') \quad &\text{if $\epsilon=0$} \eqsp.
\end{cases}
\]
\end{enumerate}
One can easily check that the kernel $\bar H$ satisfies
the following marginal conditions,
  for all $(z,z') \in \Zset^2$ and $B \in \Ysigma$,
\begin{equation}\label{eq:marginalConditionsH:gen}
  \begin{cases}
    \bar H((z,z');B \times \Yset\times\{0,1\})
    =G(\Projarg{p}{z};B)\,, \\
    \bar H((z,z');\Yset \times B\times\{0,1\} )
    =G(\Projarg{p}{z'};B) \eqsp,
  \end{cases}
\end{equation}
  Define the Markov kernel $\bar R$ on
  $(\Zset^2\times\{0,1\},\Zsigma^{\otimes2}\otimes\mathcal{P}(\{0,1\}))$ by
  setting for all $(z,z',u)\in\Zset^2\times\{0,1\}$ and $A\in \Zsigma^{\otimes2}\otimes\mathcal{P}(\{0,1\})$,
\begin{equation*}
\bar
R((z,z',u);A)=\int\1_A\left(\Psi_{y}(z),\Psi_{y'}(z'),u_1\right)\;\bar
H((z,z');\rmd y\ \rmd y'\ \rmd u_1).
\end{equation*}
Then~(\ref{eq:marginalConditionsH:gen}) and~(\ref{eq:def:kernel:R:gen}) immediately
gives~(\ref{eq:majo:d:gen:od:coupling:theirEq3}). To conclude the proof we
check~(\ref{eq:majo:d:gen:od:coupling:proba:theirEq5}). We have, for all
$v=(z,z',u)\in\Zset^2\times\{0,1\}$ and $A\in \Zsigma^{\otimes2}$,
$$
\bar R(v;A\times\{1\})=\PE\left[\1_A\left(\Psi_{y}(\bar Y),\Psi_{y'}(\bar Y)\right)\1_{\{\epsilon=1\}}\right]\;,
$$
where $\bar Y$ and $\epsilon$ are independent and distributed according to
$\underG(\Proj p^{\otimes2}(z,z');\cdot)$ and a Bernoulli distribution with
mean $\bar\alpha(z,z')$. This, and the definition of $\hat R$
in~(\ref{eq:con:min-g-phi-gen:od:defJ}) lead
to~(\ref{eq:majo:d:gen:od:coupling:proba:theirEq5}).
\end{proof}
In order to achieve~\ref{item:final-step-ergo}, we rely on the following result
which is an adaption of \cite[Lemma~9]{dou:kou:mou:2013}.
\begin{lemma} \label{lem:practical:conditions:gen:od}
Assume that there exists $(\varrho, D_1, \ell)\in (0,1) \times\rsetp\times\zsetpnz$ such that for all $(z,z')\in \Zset^2$,
\begin{align}
& \hat R\left((z,z'); \{\Zmet \leq D_1\Zmet(z,z')\}\right)=1, \label{eq:contrac:gen:k} \\
& \hat R^\ell\left((z,z'); \{ \Zmet \leq \varrho\Zmet(z,z')\}\right)=1, \label{eq:contrac:gen:ell}
\end{align}
and that $W:\Zset^2\to\rsetp$ satisfies
  \begin{align} \label{eq:driftCond:W:gen:expo-not}
\lim_{\zeta\to\infty}  \limsup_{n\to\infty}\frac1n\ \ln\,\sup_{v\in\Zset^2}\frac{\hat R^nW(v)}{W^\zeta(v)}\leq0\;,
\end{align}
Then, \eqref{eq:majo:d:gen:od} and \eqref{eq:majo:d:w:gen:od} hold.
\end{lemma}
\begin{proof}
Note that~(\ref{eq:contrac:gen:k}) implies, for any non-negative measurable
function $f$ on $\Zset^2$, $\hat R(\Zmet\times f)\leq
D_1\,\left(\Zmet\times\hat Rf\right)$ and~(\ref{eq:contrac:gen:ell}) implies $\hat R^\ell(\Zmet\times f)\leq
\varrho\,(\Zmet\times\hat R^\ell f)$. Hence for all $n\geq1$, writing
$n=k\ell+r$, where $r\in\{0,\ldots, \ell-1\}$ and $k\in\zsetp$, we get, setting $\rho'=\varrho^{1/\ell}$,
$$
\hat R^n(\Zmet\times f)\leq D_1^r\,\varrho^k\,\left(\Zmet\times\hat R^nf\right)
\leq\left(1\vee D_1^{\ell-1}\right)\varrho^{-1}\,\rho^{\prime\ n}\,\left(\Zmet\times\hat R^nf\right)\;.
$$
Taking $f\equiv1$, we get \eqref{eq:majo:d:gen:od} for any
$\rho\in[\rho',1)$. To obtain \eqref{eq:majo:d:w:gen:od}, we take $f=W$ and
observe that~(\ref{eq:driftCond:W:gen:expo-not}) implies
that,  for any $\delta>1$, there exists $\zeta>0$ such that $\sup(\hat R^nW/W^\zeta)=O(\delta^n)$. Choosing $\delta$ small
enough to make $\rho'\delta<1$, we get  \eqref{eq:majo:d:w:gen:od} with
$\rho=\rho'\delta$, $\zeta_1=1$ and $\zeta_2=\zeta$.
\end{proof}
We finally conclude~\ref{item:final-step-ergo}.  By
\autoref{lem:practical:conditions:gen:od}, it remains to
check Conditions~\eqref{eq:contrac:gen:k} and \eqref{eq:contrac:gen:ell},
and~(\ref{eq:driftCond:W:gen:expo-not}).
For all $\ell\geq 1$, we have
$\hat R^\ell\left(v;\set{\F[]{\chunk y1\ell}^{\otimes2}(v)}{\chunk y1\ell\in\Yset^\ell}\right)=1$.
Using~(\ref{eq:lipschitz:Psi}), we thus get
$$
\hat R^\ell\left(v;\Zmet\leq\Zmet(v)\left(\1_{\{\ell<q\}}\vee\max_{(\ell-p)_+<m\leq\ell}\mathrm{Lip}_m\right)\right)=1\;,
  $$
and ~\eqref{eq:contrac:gen:k} and \eqref{eq:contrac:gen:ell} both follow from~\ref{assum:bound:rho:gen}.

We now
check~(\ref{eq:driftCond:W:gen:expo-not}). By~(\ref{eq:con:min-g-phi-gen:od:defJ})~(\ref{eq:Rhatdef:XXY})
and~(\ref{eq:Wbar:def}), for all $n\geq q$ and $v\in\Zset^2$, $\hat R^nW(v)$ can be written as
\begin{align*}
\hat R^nW(v)&=\hat \PE_v\left[\max_{0\leq k<
               p}\WX(X_{n-k},X'_{n-k})\bigvee\max_{1\leq
  k<q}\frac{\WY(Y_{n-k})\vee\WY(Y'_{n-k})}{\vnorm[\WX]{\underG\WY}}\right]
  \\
  &\leq\sum_{0\leq k< p\vee q}\hat \PE_v\left[\WX(X_{n-k},X'_{n-k})
    +\frac{\WY(Y_{n-k})+\WY(Y'_{n-k})}{\vnorm[\WX]{\underG\WY}}\right]\\
    &\leq3\,\sum_{0\leq k< p\vee q}\hat \PE_v\left[\WX(X_{n-k},X'_{n-k})\right]\;,
\end{align*}
where we used, for $n-k\geq0$,
$\hat\PE_v\left[\WY(Y_{n-k})\right]=\hat
\PE_v\left[\WY(Y'_{n-k})\right]=\hat\PE_v\left[\underG\WY(X_{n-k},X'_{n-k})\right]$.
We directly get that~\ref{assum:driftCond:W:new:gen:od}
implies~(\ref{eq:driftCond:W:gen:expo-not}).  As for the two conditions
in~\ref{assum:driftCond:W:new:gen:od:new}, the first one implies that, for any
$\rho\in(0,1)$, there exists $m>0$ and $\beta>0$ such that
$\hat R^mW\leq \rho W+\beta\;,$ and the second one that there exists
$\tau\geq1$ and $C>0$ such that $\hat R^rW\leq C\, W^\tau$ for all
$r=0,\dots,m-1$. Combining the two previous bounds, we obtain, for all $n=mk+r$
with $k\in\zsetp$ and $0\leq r<m$,
$$
\hat R^nW\leq \rho^k \hat R^rW+\beta/(1-\rho)\leq
\rho^k\,C\,W^{\tau}+\beta/(1-\rho)\leq
C_m\,W^{\tau}\;,
$$
where $C_m$ is a positive constant, not depending on $n$.
Hence we obtain ~(\ref{eq:driftCond:W:gen:expo-not})
and the proof is concluded.
\subsection{Proof of \autoref{thm:convergence-main:gen:od}}
\label{sec:proof-thm:conv}
We apply \cite[Theorem~1]{douc2015handy} to the embedded ODM(1,1)
with hidden variable space $\Zset$ derived in
\autoref{sec:embedding-into-an11}.  Note that our
conditions~\ref{assum:gen:identif:unique:pi:gen}
and~\ref{ass:21-lyapunov:gen:od} yield their condition (A-1) and (A-2) on the
embedded model with
$$
\bar V(x)=\max\set{\VX(x_k)}{1\leq k\leq p}\;,\quad x=\chunk x1p\in\Xset^p\;.
$$
Let us briefly check that (B-2), (B-3) and (B-4) in \cite{douc2015handy}
hold. Conditions (B-2) and (B-3) correspond to
our~\ref{assum:continuity:Y:g-theta:gen:od}
and~\ref{assum:continuity:Y:phi-theta:gen:od}, noting, for the latter one, that
$\psi_y^\theta$ here corresponds to the $\Psi_y^\theta$ defined
in~(\ref{eq:def:psi:gen:od}), inherited from the embedding.  As for (B-4) in
\cite{douc2015handy}, we have that (B-4)(i) and (B-4)(ii) corresponds to
our~\ref{assum:practical-cond:O:gen:od}\ref{assum:exist:practical-cond:subX:gen:od}
and~\ref{assum:practical-cond:O:gen:od}\ref{assum:momentCond:gen:od}, but with
$\Xset_1$ in (B-4) replaced by
$\Zset_1=\Xset^{p-1}\times\Xset_1\times\Yset^{q-1}$ with the latter $\Xset_1$
as in~\ref{assum:practical-cond:O:gen:od}.  Also our
condition~\ref{assum:bound:rho:gen}, by \autoref{lem:Lip:XY-cond:geom}
and~(\ref{eq:lipschitz:Psi}) imply (B-4)(iii) for some $\varrho\in(0,1)$ by
setting $\bar\psi(z)=C\,\Zmet(\initmle,z)$ for some $C>0$.  This $\bar\psi$ is
locally bounded, hence (B-4)(iv) holds. Condition (B-4)(v) follows (up to a
multiplicative constant)
from~\ref{assum:practical-cond:O:gen:od}\ref{item:barphi:bar:phi:gen:od} by
observing that, $\initmle$ has constant first $p$ entries and constant $q-1$
last entries,
$$
\bar\psi(\Psi^\theta_y(\initmle))=C\,\left(\Xmet(\initmlex_1,\f{y}(\initmle))\vee\Umet(\initmleu_1,\Upsilon(y))\right)\;.
$$
The remaining conditions (vi), (vii) and (viii) of (B-2) follow directly
from~(\ref{eq:ineq:loggg:gen:od}), \ref{item:H:gen:od}
and~\ref{item:22or23-barphi-V:gen:od} in~\ref{assum:practical-cond:O:gen:od}.
All the conditions of \cite[Theorem~1]{douc2015handy} are checked and this
result gives that, for all $\theta,\thv\in\Theta$, $\tilde\PP^{\thv}$\as,
$p^\theta(y|\chunk Y{-\infty}0)$ defined as
in~(\ref{eq:def-p-theta-neq-thv:gen:od}) is well defined for all $y\in\Yset$,
that, if $\theta=\thv$, it is the density of $Y_1$ given $\chunk Y{-\infty}0$
with respect to $\nu$, and also that the MLE $\mlY{\initmle,n}$
satisfies~(\ref{eq:strong-consistency-prelim:gen:od}) with $\Theta^\star$
defined by~(\ref{eq:def-Theta-star-set:gen:od}). Finally, by
\cite[Theorem~3]{doumonrou2016maximizing}, we also obtain that
$\Theta^\star=[\thv]$ and the proof is concluded.

\subsection{Proof of \autoref{theo:ergo-convergence:logpois:gen}}
\label{sec:proof-autor-conv:logpoi}
We prove
\ref{item:thm:logpois:ergodic:gen:od},~\ref{item:thm:logpois:strong:consistency:gen:od}
and~\ref{item:thm:logpois:ident} successively.

\begin{proof}[Proof of~\ref{item:thm:logpois:ergodic:gen:od}]
We apply \autoref{thm:ergodicity:gen} with
$\VX(x)=\rme^{\tau\,|x|}$ for some arbitrary $\tau>0$.  Note that
\autoref{rem:log:poi:ergo:cond}\ref{item:rem:log:poi:ergo:1}
give~\ref{item:lodm-ident-cond-cond-dens-def}, which, by
\autoref{rem:ergo:ass}\ref{item:ergo:ass:1},
gives~\ref{assum:bound:rho:gen}.
\autoref{rem:ergo:ass}\ref{item:rem:uniform-dom-conv}
gives~\ref{assum:weak:feller:X}, and
~\ref{assum:reaching:mild}\ref{assum:reaching:atom} is trivial in this example.
Hence, it only remains to show that \ref{assum:V:gen:X},
and~\ref{assum:alpha-phi:gen} hold.

We start with~\ref{assum:V:gen:X}, with $\VX(x)=\rme^{\tau\,|x|}$. We can
further set  $\VU(u)=\rme^{\tau \,|u|}$, hence $\VY(y)=\rme^{\tau|\ln(1+y)|}$ and, by
\autoref{lem:support:bound:moment:pois}, we then have
$G\VY(x)\leq2\rme^{(1+x_+)\tau}$ so that
$\vnorm[\VX]{G\VY}\leq2\rme^{\tau}$. With these definitions,~(\ref{eq:VVX:def})
leads to
\begin{equation}
  \label{eq:V:log:poi}
V(z)\geq\left(2\rme^{\tau}\right)^{-1}\,\rme^{\tau\ |z|_\infty},\qquad z\in\Zset \;,
\end{equation}
where $|z|_\infty$ denotes the max norm of $z\in\rset^{p+q-1}$.  Now,
to bound $\PE_z[\VX(X_n)]$ as $n$ grows, we see $\VX(X_n)$ as
$\rme^{\tau|\lambda(Z_n)|}$ with the specific $\lambda=\Proj p$ and,
for any linear form $\lambda$ on $\Zset$, we look for a recursion
relation applying to
$$
\PE_z\left[\rme^{\tau|\lambda(Z_1)|}\right]=\PE\left[\rme^{\tau\,\left|\tilde\lambda_z(\ln(1+V))\right|}\right]\;,
$$
where
$V\sim\mathcal{P}(\rme^{\Projarg pz})$ and, for all  $z=(\chunk x{(-p+1)}0,\chunk y{(-q+1)}{(-1)})$, $\tilde\lambda_z:\rset\to\rset$ is defined by
\begin{equation}
  \label{eq:def:lambdatilde}
\tilde\lambda_z(y_0)=\lambda(\chunk x{(-p+2)}0,\f[]{y_0}(z),\chunk
y{(-q+1)}{(-1)},y_0)\;.
\end{equation}
Observing that $\tilde\lambda_z$ is an affine function,
of the form $\tilde\lambda_z(y)=\vartheta_0+\vartheta y$, we can apply
\autoref{lem:support:bound:moment:pois} with $\zeta=x_0$ (and the trivial bound
$|\vartheta_0|\vee|\vartheta_0+\vartheta\zeta_+|\leq|\vartheta_0|\vee|\vartheta_0+\vartheta\zeta|$)
and obtain that
$$
\PE_z\left[\rme^{\tau|\lambda(Z_1)|}\right]\leq c\,\rme^{\tau\left(\left|\lambda\circ{\FLP[]{0}}(z)\right|\vee\left|\lambda\circ{\FLP[]{1}}(z)\right|\right)}
\leq c\,\left[\rme^{\tau\left|\lambda\circ{\FLP[]{0}}(z)\right|}+\rme^{\tau\left|\lambda\circ{\FLP[]{1}}(z)\right|}\right]\;,
$$
where we set $c=2\rme^{\tau\omega+b_1}$ and, for $w=0,1$ and all $z\in\Zset$, ${\FLP[]{w}}(z)$ is
defined by ${\FLP[]{w}}(z)=(\chunk x{(-p+2)}{0},x_1,\chunk y{(-q+2)}{-1},y_0)$
with
\begin{align*}
&x_{-p+k}=\Projarg kz\quad\text{for}\quad 1\leq k\leq p\;,\\
&y_{-q+k}=\Projarg {p+k}z\quad\text{for}\quad 1\leq k< q\;,\\
&u_0=w x_0\quad\text{and}\quad  u_{1-k}= \Upsilon(y_{1-k})\quad\text{for}\quad 1< k< q\;,\\
  &x_1=\sum_{k=1}^pa_k\,x_{1-k}+\sum_{k=1}^qb_k\,u_k\;.
\end{align*}
Defining for all $w=\chunk w0{n-1}\in\{0,1\}^n$,
${\FLP[]{w}}={\FLP[]{w_{n-1}}}\circ\dots{\FLP[]{w_0}}$, we get
\begin{equation}
  \label{eq:liapu:log:poi:intermediaire}
\PE_z[\VX(X_n)]\leq c^n\,\sum_{w\in\{0,1\}^n}\rme^{\tau\left|\Proj
    p\circ{\FLP[]{w}}(z)\right|}
\;.
\end{equation}
Now observe that,
for all $w=\chunk w{0}{n+q-1}\in\{0,1\}^{n+q}$, we can define
$\Proj p\circ{\FLP[]{\chunk w{0}{(n+q-1)}}}(z)$ as $x_{n+q}$ obtained by adding
to the previous recursive equations, for $1\leq k<n+q$,
\begin{align*}
  u_k=w_k x_k\quad\text{and}\quad   x_{k+1}=\sum_{j=1}^pa_j\,x_{k+1-j}+\sum_{j=1}^qb_j\,u_{k+1-j}\;.
\end{align*}
Note that, in this recursion, we can replace $u_{k+1-j}$ by
$w_{k+1-j}x_{k+1-j}$ for $k+1-j\geq0$, hence $\chunk xq{(n+q)}$ satisfies
the
recursion~(\ref{eq:log:pois:ergo:cond:rec:def})  and it follows that $\Proj
p\circ{\FLP[]{\chunk w{0}{(n+q-1)}}}(z)$ can be expressed as
$$
\flp[]{\chunk
  w{q}{(n+q-1)}}\left(\setvect{\Projarg{p-(\ell-q)_+}{\FLP[]{\chunk
        w{0}{(q-\ell)_+}}(z)}}{1\leq\ell\leq p\vee q}\right)\;.
$$
Hence Condition~(\ref{eq:log:poi:ergo:cond}) implies that, for all $z\in\Zset$,
$$
\lim_{n\to\infty}\sup\set{\left|\Proj p\circ{\FLP[]{w}(z)}\right|}{w\in\{0,1\}^n} = 0\;.
$$
By linearity of $z\mapsto\FLP[]{w}(z)$, it follows that
\begin{equation}
  \label{eq:ergo:lo:poi:2}
\lim_{n\to\infty}\sup\set{|z|_\infty^{-1}
  \left|\Proj p\circ{\FLP[]{w}(z)}\right|}{w\in\{0,1\}^n,\
  z\in\Zset\setminus\{0\}} = 0\;.
\end{equation}
Hence using~(\ref{eq:liapu:log:poi:intermediaire}), we finally obtain for a
positive sequence $\sequence{\rho}[n][\zsetp]$,
$$
\PE_z[\VX(X_n)]\leq (2c)^n\ \rme^{\tau\,\rho_n\, |z|_\infty}\;,\quad\text{with}\quad\lim_{n\to\infty}\rho_n=0\;,
$$
which, with~(\ref{eq:V:log:poi}), leads to, for all $z\in\Zset$ and $M>0$,
$$
\frac{\PE_z\left[\VX(X_n)\right]}{M+V(z)} \leq (2c)^n\
\min\left(M^{-1}\,\rme^{\tau\,\rho_n\, |z|_\infty},
  \left(2\rme^{\tau}\right)^{-1}\, \rme^{\tau\,(\rho_n-1)\, |z|_\infty}\right)\;.
$$
Let $C>0$ be arbitrarily chosen. Using the first term and the second term in
this min for $|z|_\infty\leq C$ and $|z|_\infty>C$ respectively, we get that, for any $n$ such that $\rho_n<1$,
\begin{align*}
\limsup_{M\to\infty}\sup_{z\in\Zset}\frac{\PE_z\left[\VX(X_n)\right]}{M+V(z)}
  \leq (2c)^n\   \left(2\rme^{\tau}\right)^{-1}\, \rme^{\tau\,(\rho_n-1)\,
  C}\to0\text{ as }C\to\infty\;.
\end{align*}
Using that $\sequence{\rho}[n][\zsetp]$ converges to 0, this holds for $n$
large enough and we
get~(\ref{eq:ergoVX}), and~\ref{assum:V:gen:X} holds.

We now turn to the proof of~\ref{assum:alpha-phi:gen}.
We can apply \autoref{lem:det:alpha:gen:od} with $\Cset=\rset=\Xset$ and
$\Sset=\{1\}$, $\mu$ being the Dirac mass at point 1. For all $(x,y)\in\Xset\times\Yset$, let $j(x)=\rme^{-\rme^x}$ and
$h(x;u)=\frac{\rme^{x(\rme^u-1)}}{(\rme^u-1)!}$, which
$h$ satisfies ~\ref{item:F1:gen}. Hence \autoref{lem:det:alpha:gen:od} gives
that~\ref{assum:alpha-phi:gen}\ref{item:alpha-phi:3:gen} and  \eqref{eq:def:phi} hold with
$$
\alpha(x,x')=\frac{\rme^{-\rme^{x\vee x'}}}{\rme^{-\rme^{x\wedge x'}}}
=\rme^{-\left|\rme^{x}-\rme^{x'}\right|}
\quad\text{and}\quad \phi(x,x')=x\wedge x'\;,\qquad x,x'\in\Xset\;.
$$
Now for all $x,x'\in\Xset$, we have
\begin{align*}
1-\alpha(x,x')=1-\rme^{-\left|\rme^{x}-\rme^{x'}\right|}\le \left|\rme^{x}-\rme^{x'}\right|
\leq  \rme^{|x|\vee|x'|}\left|x-x'\right|\;.
\end{align*}
We thus obtain~\ref{assum:alpha-phi:gen}\ref{assum:hyp:1-alpha:W:gen} by
setting $\WX(x,x')=\rme^{|x|\vee|x'|}$,
and~\ref{assum:alpha-phi:gen}\ref{assum:definition-gamma-x:gen} also follows by
setting $\WU(y)=\rme^{|y|}$. Since $\WU$ is $\VU$ with $\tau=1$, we already saw
that $G\WY(x)\leq2\rme^{1+x_+}$, hence $\vnorm[\WX]{\WY\circ\phi}\leq2\rme$,
and~(\ref{eq:Wbar:def}) leads to, for all $z$, $z'$ in $\Zset$,
\begin{equation}
  \label{eq:W:log:poi}
W(z,z')\geq
(2\rme)^{-1}\,\rme^{|z|_\infty\vee|z'|_\infty}\;.
\end{equation}
It now remains to prove either~\ref{assum:alpha-phi:gen}\ref{assum:driftCond:W:new:gen:od}
or~\ref{assum:alpha-phi:gen}\ref{assum:driftCond:W:new:gen:od:new}, which both involve $\hat\PE_z\left[\WX(X_n,X'_n)\right]$.
We proceed as
previously when we bounded $\PE_z\left[\rme^{\tau|\lambda(Z_1)|}\right]$. Let $\tau\geq1$.
For any linear function $\lambda:\Zset^2\to\Zset^2$, we have, for all $v=(z,z')\in\Zset^2$,
\begin{align*}
\hat\PE_v\left[\rme^{\tau|\lambda(Z_1)|\vee|\lambda(Z'_1)|}\right]
&\leq\hat\PE_v\left[\rme^{\tau|\lambda(Z_1)|}\right]+\hat\PE_z\left[\rme^{\tau|\lambda(Z'_1)|}\right]\\
&=\PE\left[\rme^{\tau\left|\tilde\lambda_z(\ln(1+V))\right|}\right]+\PE\left[\rme^{\tau\left|\tilde\lambda_{z'}(\ln(1+V))\right|}\right]\;,
\end{align*}
where $V\sim\mathcal{P}(\rme^{\phi\circ\Proj p^{\otimes2}(v)})$ and $\tilde\lambda_z$ is defined
by~(\ref{eq:def:lambdatilde}). By definition of $\phi$ above, we have
$\phi\circ\Proj p^{\otimes2}(v)\leq\Projarg{p}{z},\Projarg{p}{z'}$.
Hence \autoref{lem:support:bound:moment:pois} with $\zeta=\phi\circ\Proj
p^{\otimes2}(v)$ and $\zeta'=\Projarg{p}{z}$ and $\Projarg{p}{z'}$
successively, we obtain, similarly as before for bounding
$\PE_z\left[\rme^{\tau|\lambda(Z_1)|}\right]$, that for all $v=(z,z')\in\Zset^2$,
\begin{align}
  \nonumber
\hat\PE_v\left[\rme^{\tau(|\lambda(Z_1)|\vee|\lambda(Z'_1)|)}\right]
&\leq
  c\,\sum_{z''=z,z'}\sum_{w=0,1}\rme^{\tau\left|\lambda\circ{\FLP[]{w}}(z'')\right|}\\
  \label{eq:Wbound-avec-tau}
&  \leq
  2c\,\sum_{w=0,1}\rme^{\tau(\left|\lambda\circ{\FLP[]{w}}(z)\right|\vee\left|\lambda\circ{\FLP[]{w}}(z')\right|)}\;.
\end{align}
Taking  $\lambda=\Proj p$, and observing that $\Proj p\circ{\FLP[]{w}}$ is a
linear form for $w=0,1$, we get, setting $c_0=\max_{w=0,1}\sup_{|z|_\infty\leq1}|\Proj p\circ{\FLP[]{w}}(z)|$,  for all $v=(z,z')\in\Zset^2$,
\begin{equation}
  \label{eq:avec-tau-W-cond1}
\hat\PE_v\left[\WX^\tau(X_1,X'_1)\right]
\leq 4c\,\rme^{\tau\,c_0\,|v|_\infty}\text{ with }|v|_\infty:=|z|_\infty\vee|z'|_\infty\;.
\end{equation}
Thus, with~(\ref{eq:W:log:poi}), the second condition
of~\ref{assum:alpha-phi:gen}\ref{assum:driftCond:W:new:gen:od:new} holds with
$\tau'=\tau(c_0\vee1)$. To conclude, it is now sufficient to show that the first condition
of~\ref{assum:alpha-phi:gen}\ref{assum:driftCond:W:new:gen:od:new} also holds.
Iterating~(\ref{eq:Wbound-avec-tau}) and taking $\tau=1$ and $\lambda=\Proj p$, we thus get, for all $n\in\zsetp$ and $v=(z,z')\in\Zset^2$,
\begin{align*}
\hat\PE_v\left[\WX(X_n,X'_n)\right]&=\hat\PE_v\left[\rme^{|\lambda(Z_n)|\vee|\lambda(Z'_n)|}\right]\\
&\leq
  (4c)^n\,\max\set{\rme^{\left|\Proj p\circ{\FLP[]{w}}(z)\right|\vee\left|\Proj
        p\circ{\FLP[]{w}}(z')\right|}}{w\in\{0,1\}^n}\;.
\end{align*}
Applying~(\ref{eq:ergo:lo:poi:2}) and~(\ref{eq:W:log:poi}), we get that, for all  $v\in\Zset^2$,
$$
\frac{\hat\PE_v\left[\WX(X_n,X'_n)\right]}{M+W(v)}\leq (4c)^n\,\min
\left\{\frac{\rme^{\rho_n\,|v|_\infty}}{M},(2\rme)\,\rme^{(\rho_n-1)\,|v|_\infty}\right\}\;,
$$
where $\sequence{\rho}[n][\zsetp]$ is a positive sequence converging to 0.  We
now proceed as for proving~\ref{assum:V:gen:X} previously: the first term in the min
tends to 0 as $M\to\infty$ uniformly over $|v|_\infty\leq C$ for any $C>0$,
while, if $\rho_n<1$, the second one tends to zero as
$|v|_\infty\to\infty$. Hence, for $n$ large enough, we have
$$
\lim_{M\to\infty}\sup_{v\in\Zset^2}\frac{\hat\PE_v\left[\WX(X_n,X'_n)\right]}{M+W(v)}=0\;,
$$
and ~\ref{assum:alpha-phi:gen}\ref{assum:driftCond:W:new:gen:od:new} follows,
which concludes the proof.
\end{proof}
\begin{proof}[Proof of~\ref{item:thm:logpois:strong:consistency:gen:od}]
We apply \autoref{thm:convergence-main:gen:od}.
We have already
shown that~\ref{assum:bound:rho:gen},~\ref{assum:gen:identif:unique:pi:gen} and~\ref{ass:21-lyapunov:gen:od}
hold in the proof of Assertion~\ref{item:thm:logpois:ergodic:gen:od}, with $\VX(x)=\rme^{\tau|x|}$ for
any $\tau>0$.   Assumptions \ref{assum:continuity:Y:g-theta:gen:od} and
\ref{assum:continuity:Y:phi-theta:gen:od} obviously hold for the log-linear
Poisson GARCH model (see \autoref{rem:linear-case:conv:assump} for the second one). It now only remains
to show that, using $\VX$ as above, \ref{assum:practical-cond:O:gen:od} is also
satisfied. We set $\Xset_1=\Xset$ which trivially satisfies
\ref{assum:practical-cond:O:gen:od}\ref{assum:exist:practical-cond:subX:gen:od}.
Condition~\ref{assum:practical-cond:O:gen:od}\ref{assum:momentCond:gen:od} is
also immediate (see \autoref{rem:discret-case:conv:assump}).
Next, we look for an adequate $\bar\phi$, $\mathrm{h}$ and $C\geq0$ so
that~\ref{item:barphi:bar:phi:gen:od}~\ref{item:ineq:loggg:gen:od}~\ref{item:H:gen:od}
and~\ref{item:22or23-barphi-V:gen:od} hold in~\ref{assum:practical-cond:O:gen:od}.
We have, for all $\theta\in\Theta$ and $(x,x',y)\in\rset^2\times\zsetp$,
\begin{align*}
  \left|\ln g^\theta\left(x;y\right)-\ln g^\theta\left(x';y\right)\right|
  &    \leq |x-x'|\,\rme^{x\vee x'}\,y \\
  &\leq |x-x'|\,\rme^{\left|x-\initmlex_1\right|\vee\left|x'-\initmlex_1\right|}\,y\,\rme^{\initmlex_1}\;.
\end{align*}
We thus set $\mathrm{h}(u)=\rme^{\initmlex_1}\,u$, $C=1$ and $\bar\phi(y)=A+B\,y$ for
some adequate non-negative $A$ and $B$
so that~\ref{item:barphi:bar:phi:gen:od} (using
\autoref{rem:linear-case:conv:assump} and $\Upsilon(y)=\ln(1+y)\leq y$),  \ref{item:ineq:loggg:gen:od}
and~\ref{item:H:gen:od} follow. Then we have $G^\theta\bar\phi(x)\leq
A+B\,\rme^x\lesssim\VX$, provided that we chose $\tau\geq1$. This
gives~\ref{item:22or23-barphi-V:gen:od} and
thus~\ref{assum:practical-cond:O:gen:od}  holds true, which concludes the proof.
\end{proof}
\begin{proof}[Proof of~\ref{item:thm:logpois:ident}]
  We apply \cite[Theorem~17]{douc-roueff-sim_ident-genod2020} in which:
  \begin{enumerate}[label=-]
  \item Condition (A-1)
  corresponds to our assumption~\ref{assum:gen:identif:unique:pi:gen} shown in
  Point~\ref{item:thm:logpois:ergodic:gen:od} above;
\item Condition (L-3) corresponds to checking that
    \begin{equation}
    \label{eq:identif:abitofmoment:logpoisson}
\int \ln^+ (\ln(1+y)) \
  \piY^\theta(\rmd y)<\infty\;.
  \end{equation}
  We already checked that
  $G^\theta\VY\lesssim\VX$ with $\VY(y)=(1+y)^{\tau}$, implying
  $\int(1+y)^{\tau}\piY^\theta(\rmd
  y)<\infty$. Hence~(\ref{eq:identif:abitofmoment:logpoisson}) holds.
\item Condition (SL-1)  is the same as our
  Condition~\ref{item:lodm-ident-cond-cond-dens-def}, which were already proved
  for showing  Point~\ref{item:thm:logpois:ergodic:gen:od} above;
\item Conditions (SL-2) and (SL-3) are immediately checked for the Log Poisson Garch model;
\item Condition (SL-4) is directly assumed in~\ref{item:thm:logpois:ident};
\end{enumerate}
Thus  \cite[Theorem~17]{douc-roueff-sim_ident-genod2020}  gives that
  $[\thv]$ reduced to $\{\thv\}$ and Assertion~\ref{item:thm:logpois:ident}
  follows from~\ref{item:thm:logpois:strong:consistency:gen:od}.
\end{proof}

\subsection{Proof of \autoref{theo:ergo-convergence:garch:gen:nbin}}
\label{sec:proof-autor-conv:nbin}

We prove \ref{item:thm:nbin:ergo:gen:od},~\ref{item:thm:nbin:strong:consistency:gen:od}
and~\ref{item:thm:nbin:ident} successively. We take $\nu$ is the counting
measure on $\zsetp$ so that we have
\begin{equation}\label{eq:density:g:pos:gen:nbin}
g^\theta(x;y)=\frac{\Gamma(r+y)}{y\,!\,\Gamma(r)}\left(\frac{1}{1+x}\right)^r\left(\frac{x}{1+x}\right)^y\eqsp.
\end{equation}

\begin{proof}[Proof of~\ref{item:thm:nbin:ergo:gen:od}]
  As for the Log-linear Poisson Garch$(p,q)$, we apply
  \autoref{thm:ergodicity:gen} this time with $\VX(x)=x$. Note that, since
  $a_k,b_k\geq0$ for all $k$, Condition~(\ref{eq:con:ergo:gen:nbin:cns})
  implies $\sum_{k=1}^pa_k<1$, which
  implies~\ref{item:lodm-ident-cond-cond-dens-def} and
  thus~\ref{assum:bound:rho:gen} by \autoref{rem:ergo:ass}\ref{item:ergo:ass:1}.
  Also as for the log-linear Poisson Garch$(p,q)$ case,
  \autoref{rem:ergo:ass}\ref{item:rem:uniform-dom-conv}
  gives~\ref{assum:weak:feller:X},
  and~\ref{assum:reaching:mild}\ref{assum:reaching:atom} is trivially
  satisfied.  Hence, again, we only have to show that \ref{assum:V:gen:X},
  and~\ref{assum:alpha-phi:gen} hold.  Here $\Upsilon$ is the identity mapping.

We start with~\ref{assum:V:gen:X}, with $\VX(x)=x$. We can
further set $\VY(y)=y$ since, we then have
$G\VY(x)=r\VX(x)$ so that
$\vnorm[\VX]{G\VY}=r$. With these definitions,~(\ref{eq:VVX:def})
leads to
\begin{equation}
  \label{eq:V:nbin}
V(z)\geq\left(1+r\right)^{-1}\,|z|_\infty,\qquad z\in\Zset \;.
\end{equation}
On the other hand, we have that, for all $z=(\chunk x{(-p+1)}0,\chunk
u{(-q+1)}{(-1)}\in\Zset=\rsetp^{p+q-1}$ and all $n=1,2,\dots$,
$$
\PE_z[\VX(X_{n})]=\PE_z[X_{n}]=\omega+\sum_{k=1}^pa_k\PE_z[X_{n-k}]+\sum_{k=1}^qb_k\PE_z[Y_{n-k}]\;,
$$
where, in the right-hand side of this equation, if $n-k\leq0$, $\PE_z[X_{n-k}]$
should be replaced by $x_{n-k}$ and if $n-k\leq -1$, $\PE_z[Y_{n-k}]$
should be replaced by $u_{n-k}$. By definition of $G$,
$\PE_z[Y_{n-k}]=r\PE_z[X_{n-k}]$. Hence, denoting
$x_n=\PE_z[X_{n}]$ we get that the sequence $\nsequence{x_k}{k \in\zsetp}$ satisfies
the recursion
$$
x_n=\omega+\sum_{k=1}^{p\vee q}c_kx_{n-k}\;,\quad n> p\vee q\;,
$$
where here $c_k=a_k+r\ b_k\geq0$ for $k=1,\dots,p\wedge q$, $c_k=a_k$ if
$q<k\leq p$ and $c_k=rb_k$ if $p<k\leq q$. Moreover we can clearly
find a constant $C$ only depending on $\theta$ such that 
$\left|\chunk x1{(p\vee q)}\right|_\infty\leq C(1+|z|_\infty)$. Then
applying \autoref{lem:affine-rec}  and~(\ref{eq:V:nbin}), we get that
there exists $C'>0$ and $\rho\in(0,1)$ such that, for all $n\geq1$,
$$
\PE_z[X_{n+1}]\leq C'\,\rho^n\,(1+V(z))\;.
$$
Condition~\ref{assum:V:gen:X} follows.

We conclude with the proof of~\ref{assum:alpha-phi:gen}.
Let us apply \autoref{lem:det:alpha:gen:od} with $\Cset=(0,\infty)=\Xset$ and
$\Sset=\{1\}$, $\mu$ being the Dirac mass at point 1, $j(x)=(1+x)^{-r}$ and
$h(x;y)=\frac{\Gamma(r+y)}{y\,!\,\Gamma(r)}\left(\frac{x}{1+x}\right)^y$, which
satisfies~\ref{item:F1:gen}. This leads to
$\alpha$ and $\phi$
satisfying~\ref{assum:alpha-phi:gen}\ref{item:alpha-phi:3:gen} and \eqref{eq:def:phi} by setting, for all $x,  x'\in\rsetp$
$$
\alpha(x,x')=\left(\frac{1+x\wedge x'}{1+x\vee x'}\right)^r
\quad\text{and}\quad
\phi(x,x')=x\wedge x'\eqsp.
$$
Next, for all $(x, x')\in\Zset^2$, we have
\begin{align*}
1-\alpha(x,x') =1-\left(\frac{1+x\wedge x'}{1+x\vee x'}\right)^r \leq (1\vee r)\left|x-x'\right|
\end{align*}
We thus obtain~\ref{assum:alpha-phi:gen}\ref{assum:hyp:1-alpha:W:gen} by
setting $\WX(x,x')=(1\vee r)$. All the other
conditions of~\ref{assum:alpha-phi:gen} are trivially satisfied in this case,
taking $\WU\equiv1$.
\end{proof}
\begin{proof}[Proof of~\ref{item:thm:nbin:strong:consistency:gen:od}]
  We apply \autoref{thm:convergence-main:gen:od}.  We have already shown
  that~\ref{assum:bound:rho:gen},~\ref{assum:gen:identif:unique:pi:gen}
  and~\ref{ass:21-lyapunov:gen:od} hold in the proof of
  Assertion~\ref{item:thm:logpois:ergodic:gen:od}, with $\VX(x)=x$.  Assumption
  \ref{assum:continuity:Y:g-theta:gen:od} follows
  from~(\ref{eq:density:g:pos:gen:nbin}), and
  \ref{assum:continuity:Y:phi-theta:gen:od}
  from \autoref{rem:linear-case:conv:assump}. It now only remains to show that,
  using $\VX$ as above, \ref{assum:practical-cond:O:gen:od} is also satisfied.

  Since $\Theta$ is compact, we can find $\underline{\omega}>0$ and
  $\overline{r}\in\zsetp$ such that $\omega\geq\underline{\omega}$ and
  $r\leq\overline{r}$ for all
  $\theta=(\omega,\chunk{a}{1}{p},\chunk{b}{1}{q},r)\in\Theta$.
  We set $\Xset_1=[\underline{\omega},\infty)$ which then satisfies
\ref{assum:practical-cond:O:gen:od}\ref{assum:exist:practical-cond:subX:gen:od}.
Condition~\ref{assum:practical-cond:O:gen:od}\ref{assum:momentCond:gen:od}
follows from \autoref{rem:discret-case:conv:assump}.

Next, we look for an adequate $\bar\phi$, $\mathrm{h}$ and $C\geq0$ so
that~\ref{item:barphi:bar:phi:gen:od}~\ref{item:ineq:loggg:gen:od}~\ref{item:H:gen:od}
and~\ref{item:22or23-barphi-V:gen:od} hold in~\ref{assum:practical-cond:O:gen:od}.
For all $\theta\in\Theta$, $(x,x')\in\Xset_1$ and $y\in\zsetp$, we have
\begin{align*}
\left|\ln g^\theta (x;y)-\ln g^\theta (x';y)\right|
&=\left|(r+y)[\ln(1+x')-\ln(1+x)]+y[\ln x-\ln x']\right|\\
&\leq \left[(r+y)(1+\underline \omega)^{-1}+y\,\underline
  \omega^{-1}\right]\;|x-x'|\\
&\leq \left[\overline r +2y\,  \underline{\omega}^{-1}\right]\;|x-x'| \;.
\end{align*}
We set $C=0$, $\mathrm{h}(s)=s$ and $\bar\phi(y)=A+B\,y$ for
some adequate non-negative $A$ and $B$
so that~\ref{item:barphi:bar:phi:gen:od} (using
\autoref{rem:linear-case:conv:assump} and $\Upsilon(y)= y$), \ref{item:ineq:loggg:gen:od}
and~\ref{item:H:gen:od} follow. Then we have $G^\theta\ln^+\bar\phi(x)\leq
A+B\,r\,x\lesssim\VX$. This
gives~\ref{item:22or23-barphi-V:gen:od} and
thus~\ref{assum:practical-cond:O:gen:od}  holds true, which concludes the proof.
\end{proof}
\begin{proof}[Proof of~\ref{item:thm:nbin:ident}]
  The proof of this point is similar to the proof
  of~\autoref{theo:ergo-convergence:logpois:gen}\ref{item:thm:nbin:ident} in
  \autoref{sec:proof-autor-conv:logpoi}, except
  that Condition~(\ref{eq:identif:abitofmoment:logpoisson}) is replaced by
  \begin{equation}
    \label{eq:identif:abitofmoment:nbin}
\int \ln^+ (|y|) \
  \piY^\theta(\rmd y)<\infty\;.
  \end{equation}
We already checked that
  $G^\theta\VY\lesssim\VX$ with $\VY(y)=y$, implying
  $\int y\,\piY^\theta(\rmd y)<\infty$. Hence~(\ref{eq:identif:abitofmoment:nbin}) holds
  and the proof is concluded.
\end{proof}
\subsection{Proof of \autoref{theo:ergo-convergence:parx}}
\label{sec:proof-autor-conv:parx}
The following remark about Assumption~\ref{item:assum:parx:vi} will be
useful.
\begin{remark}
\label{rem:assumption:tv} If the Markov kernel
$L$ admits a kernel density $\ell$ with respect to some measure
$\nu_\X$ on Borel sets of $\rset^r$ as in~\ref{hyp:parx-model-L-ell},
then
\begin{align*}
\tvdist{L(\xi,\cdot)}{L(\xi',\cdot)}& = \frac12
\int\left|\ell(\xi;\xi'')-\ell(\xi';\xi'')\right|\;\nu_\X(\rmd\xi'')\\
&=1-\int\left(\ell(\xi;\xi'')\wedge \ell(\xi';\xi'')\right) \;\nu_\X(\rmd\xi'')\;.  
\end{align*}
It follows that Assumption~\ref{item:assum:parx:vi} is
  equivalent to assuming that there exist a measurable function
$\alpha_\X: \rset^{2r} \to [0,1]$ and a Markov kernel $\underline{L}$
on $(\rset^r)^2 \times \mathcal{B}(\rset^r)$ dominated by $\nu_\X$
with kernel density $\underh\left(\xi,\xi';\xi''\right)$ such that
             \begin{enumerate}[(a)]
             \item for all $(\xi,\xi',\xi'')\in \rset^{3r}$,
               $$
               \ell(\xi;\xi'')\wedge \ell(\xi';\xi'') \geq \alpha_\X(\xi,\xi')\underh\left(\xi,\xi';\xi''\right)\eqsp.
             $$
             \item \label{item:assum:parx:vi:b} there exists a constant $M\geq 1$ such that for all $(\xi,\xi')\in \rset^{2r}$,
             $$
             1-\alpha_\X(\xi,\xi')\leq M \|\xi-\xi'\| \eqsp.
             $$
             \end{enumerate}
             Obviously these two conditions
             imply~\ref{item:assum:parx:vi}. To see the converse
             implication, take
             $\alpha_\X(\xi,\xi')=\int_{\rset^r}\left( \ell(\xi;\xi'')
               \wedge \ell(\xi';\xi'') \right)\;\nu_\X(\rmd \xi'')$
             and
             $\underh\left(\xi,\xi';\xi''\right)=\left(\ell(\xi;\xi'')
               \wedge \ell(\xi';\xi'')\right)/\alpha_\X(\xi,\xi')$.
\end{remark}
\begin{proof}[Proof of~\ref{item:thm:parx:ergo:gen:od}]
  As explained in \autoref{exmpl:parx}, the PARX model can be cast
  into a VLODM$(p,q)$ and, to prove ergodicity, we can thus apply
  \autoref{thm:ergodicity:gen}. Hence we now need to check
  Conditions~\ref{item:CCMShyp}, \ref{assum:bound:rho:gen},
  \ref{assum:weak:feller:X}, \ref{assum:V:gen:X},
  \ref{assum:reaching:mild} and~\ref{assum:alpha-phi:gen}.
  \ref{item:CCMShyp} always hold for a VLODM. \ref{assum:bound:rho:gen} follows
from~(\ref{eq:cond:ergo-convergence:parx:one}) by
Assertion~\ref{rem:assum:parx:i-1} in~\autoref{lem:assum:parx:i}.

Let us check~\ref{assum:weak:feller:X}, by first verifying
\eqref{eq:uniform-dom-conv}.  Using the definition of  the 
kernel $G$ (we drop
$\theta$ from the notation $G^\theta$ since it does not depend on
$\theta$) in  \autoref{exmpl:parx}~\ref{item:exmpl:parx:G}, and
using~\ref{hyp:parx-model-L-ell}, $G$ admits a kernel density $g$ wrt
$\nu\eqdef \nu_{\zsetp}\otimes \nu_\X$, with $\nu_{\zsetp}$ denoting
the counting measure on $\zsetp$, given by
 \begin{equation} \label{eq:g:parx}
 g( \bar x;\bar y)=\rme^{-x} \frac{x^y}{y!} \ell(\xi;\xi')\,, \quad (\bar x,\bar y)=((x,\xi),(y,\xi')) \in \bar \Xset\times \bar \Yset\eqsp.
\end{equation}
We set, for  $(\bar x,\bar x')=((x,\xi),(x',\xi')) \in \bar \Xset^2$
\begin{equation}
  \label{eq:parx:def:delta}
  \Xmet(\bar x , \bar x')=|x-x'|+\|\xi-\xi'\| \eqsp.
\end{equation}
According
to assumption \ref{hyp:parx-model-L-ell}\ref{eq:assum:parx:ii}, there
exists $\delta>0$ such that
$\int_{\rset^r} \sup_{\|\xi'-\xi\| \leq \delta} h
(\xi';\xi'')\nu_\X(\rmd \xi'')<\infty$. Thus, using the expression of
$g$ given in \eqref{eq:g:parx},
\begin{align*}
  &      \sum_{y=0}^\infty \int_{\rset^r} \sup \left\{g( (x',\xi');(y,\xi''))\,,\; \Xmet((x',\xi'),(x,\xi)) \leq \delta\right\} \nu_\X(\rmd \xi'')\\
  &\quad \leq \sum_{y=0}^\infty \rme^{-x+\delta} \frac{(x+\delta)^y}{y!} \int_{\rset^r} \sup_{\|\xi'-\xi\|\leq\delta}\{h(\xi';\xi'')\} \nu_\X(\rmd \xi'') \eqsp,\\
&\quad \leq \rme^{2\delta}  \int_{\rset^r} \sup_{\|\xi'-\xi\|\leq\delta}\{h(\xi';\xi'')\} \nu_\X(\rmd \xi'') <\infty\eqsp,
\end{align*}
which shows \eqref{eq:uniform-dom-conv}. Since $L$ is weak-Feller by~\ref{hyp:parx-model-wf},
\ref{assum:weak:feller:X}-\ref{assum:feller:cont} holds, and finally
\ref{assum:weak:feller:X} is satisfied.

We now turn to checking~\ref{assum:V:gen:X}. Define
\begin{align*}
  &\VY(y,\xi)=\VU \circ \Upsilon(y,\xi)=\VU(y,f_1(\xi),\ldots,f_d(\xi),\xi)=y+V_\X(\xi)\eqsp,\\
  &\VX(x,\xi)=x+V_\X(\xi)\eqsp.
\end{align*}
Then, using \ref{item:assum:parx:v}-\ref{item:assum:parx:v:d} with
$n=1$ and $h=V_\X$,
$$
\frac{G\VY}{\VX}(x,\xi)=\frac{x+LV_\X(\xi)}{x+V_\X(\xi)}\leq \frac{x+C\varrho V_\X(\xi) +\pi(V_\X)}{x+V_\X(\xi)} \leq 1\vee (C\varrho) + \pi(V_\X)\eqsp,
$$
so that $G\VY\lesssim\VX$. Moreover, $\{\VX\leq M\}$ is compact for
any $M>0$, and so is $\{\VY\leq M\}$ since $\{V_\X\leq M\}$ is compact
for any $M>0$ by \ref{item:assum:parx:v}-\ref{item:assum:parx:v:b}. To
complete the proof of~\ref{assum:V:gen:X}, it only remains to check
\eqref{eq:ergoVX}. Recall that in this model,
$\bar U_t=(Y_t,f_1(\Xi_t),\ldots,f_d(\Xi_{t}),\Xi_t)$. Noting that for
all $z\in \Zset=\bar \Xset^p\times \bar \Uset^{q-1}$,
$\PE_z[Y_t]=\PE_z[X_t]$ (see \eqref{eq:def:parx}), we have, using~\ref{item:assum:parx:v},
\begin{align*}\nonumber
  \PE_z[X_n]&= \PE_z\lrb{\omega+\sum_{i=1}^{p}a_i X_{n-i}+\sum_{i=1}^q b_i
    Y_{n-i}}+L^n (\func_\gamma) (\Xi_{-1}) \\
    \nonumber  &\leq \omega + \sum_{i=1}^{p}a_i
    \PE_z[X_{n-i}]+\sum_{i=1}^{q}b_i \PE_z[X_{n-i}]\\
&  \quad +C\varrho^n
    \vnorm[V_\X]{\func_\gamma} V_\X(\Xi_{-1})+\pi_\X(\func_\gamma)\;.
\end{align*}
Defining $c_i$ as in Assertion~\ref{rem:assum:parx:i-2} of~\autoref{lem:assum:parx:i}, we write this inequality as
\begin{align}
  \label{eq:A6:parx:one}
  \PE_z[X_n] \leq \omega + \sum_{i=1}^{p\vee q} c_i \PE_z[X_{n-i}] +C\varrho^n \vnorm[V_\X]{\func_\gamma} \VU(\Pi_{p+q}(z))+\pi_\X(\func_\gamma) \eqsp.
\end{align}
Similarly, we have
\begin{align}
  \PE_z[\VX(X_n,\Xi_{n-1})]=\PE_z[X_n]+L^n(V_\X)(\Xi_{-1}) \nonumber \\
  \quad\leq \PE_z[X_n] +C\varrho^n  \VU(\Pi_{p+q}(z))+\pi_\X(V_\X) \label{eq:A6:parx:two} \eqsp.
\end{align}
Combining \eqref{eq:A6:parx:one} and \eqref{eq:A6:parx:two} yields
that, for any $\epsilon>0$, we can find $M$ large enough such that,
with $V$ defined by~(\ref{eq:VVX:def}), and setting
$$
u_n:=\sup_{z\in\Zset}\frac{\PE_z[\VX(X_n,\Xi_{n-1})]}{M+V(z)}\;,\quad n\in\zsetp\;,
$$
we have, for all $n\in\zsetpnz$, $u_n\in\rsetp$ and
\begin{align*}
u_n\leq \sum_{i=1}^{p\vee q} c_i u_{n-i} +C\varrho^n\vnorm[\VX]{G\VY}+\frac12\epsilon\;.
\end{align*}
Thus, since $\varrho<1$, for $n$ large enough,
\begin{align*}
u_n\leq \sum_{i=1}^{p\vee q} c_i u_{n-i} +\epsilon\;.
\end{align*}
It follows that, for any $\epsilon>0$,
$$
\limsup_{n\to\infty}\lim_{M\to\infty}\sup_{z\in\Zset}\frac{\PE_z[\VX(X_n,\Xi_{n-1})]}{M+V(z)}\leq \limsup_{n\to\infty}u_n\leq \frac{\epsilon}{1-\sum_{i=1}^{p\vee q}c_i}\;,
$$
where, in the second inequality, we used \autoref{lem:affine-rec} with
Assertion~\ref{rem:assum:parx:i-2} of~\autoref{lem:assum:parx:i}. We
thus get~\eqref{eq:ergoVX}.

Let us now check \ref{assum:reaching:mild}.
From~\ref{hyp:parx-model-L-ell}~\ref{eq:assum:parx:i} and
\eqref{eq:g:parx}, we have \eqref{eq:ass:g}. Since
\ref{assum:reaching:mild}-\ref{assum:reaching:cont} always holds for a
VLODM, we obtain \ref{assum:reaching:mild}.

To complete the proof, it remains to check~\ref{assum:alpha-phi:gen}.
In what follows, we set for
$(\bar x,\bar x')=((x,\xi),(x',\xi')) \in\bar\Xset^2$ and
$\bar y=(y,\xi'') \in\bar \Yset$,
\begin{equation} \label{eq:parx:def:alpha:phi}
  \alpha(\bar x,\bar x')=\rme^{-|x-x'|}  \alpha_\X(\xi,\xi'), \quad  \underg(\bar x,\bar x';\bar y)=\rme^{-x\wedge x'}\frac{(x \wedge x')^y}{y!} \underh(\xi,\xi';\xi'')\;,
\end{equation}
where $\alpha_\X$ is given by \autoref{rem:assumption:tv} and take for
$\WX: \bar \Xset^2 \to [1,\infty)$ and $\WU:\bar \Uset\to\rsetp$ the
constant functions equal to $M$, where $M$ is the constant that
appears in Assumption \ref{item:assum:parx:vi} and
\autoref{rem:assumption:tv}\ref{item:assum:parx:vi:b}.

Using \eqref{eq:g:parx},  we have, for all  $(\bar x,\bar x')=((x,\xi),(x',\xi')) \in \bar \Xset^2$ and $\bar y=(y,\xi'') \in \bar \Yset$,
\begin{align*}
  & \min\{g(\bar x;\bar y), g(\bar x';\bar y) \} \geq \rme^{-x\vee x'} \frac{(x \wedge x')^y}{y!} h(\xi;\xi'') \wedge h(\xi';\xi'') \\
  & \quad \geq \rme^{-|x-x'|} \rme^{-x\wedge x'}\frac{(x \wedge x')^y}{y!} \alpha_\X(\xi,\xi') \underh(\xi,\xi';\xi'')\\
  & \quad \geq \alpha(\bar x,\bar x') \underg(\bar x,\bar x';\bar y) \eqsp,
\end{align*}
where $\alpha$ and $\underg$ are defined in \eqref{eq:parx:def:alpha:phi}. Then \ref{assum:alpha-phi:gen}-\ref{item:alpha-phi:3:gen} holds. Moreover, since $\WX$ and $\WU$ are constants,   \ref{assum:alpha-phi:gen}-\ref{assum:definition-gamma-x:gen} trivially holds. We now turn to  \ref{assum:alpha-phi:gen}-\ref{assum:hyp:1-alpha:W:gen}. By \eqref{eq:parx:def:alpha:phi},  for all $(\bar x,\bar x')=((x,\xi),(x',\xi')) \in \bar \Xset^2$,
\begin{align}
  1-\alpha(\bar x,\bar x')&=1-\rme^{-|x-x'|}  \alpha_\X(\xi,\xi')\leq 1+\rme^{-|x-x'|}(M \|\xi-\xi'\|-1) \nonumber \\
                          & \leq 1-\rme^{-|x-x'|} + \rme^{-|x-x'|} M \|\xi-\xi'\| \leq M (|x-x'|+\|\xi-\xi'\|) \eqsp, \label{eq:parx:alpha}
\end{align}
where in the last inequality we have used
$ 1-\rme^{-|x-x'|}\leq |x-x'| \leq M|x-x'|$. Plugging the definition
of $\Xmet$ given in \eqref{eq:parx:def:delta} into
\eqref{eq:parx:alpha} shows
\ref{assum:alpha-phi:gen}-\ref{assum:hyp:1-alpha:W:gen}. Since $\WX$
and $\WU$ are constants, we obtain that $W$ defined in
\ref{assum:alpha-phi:gen} is also constant and therefore any of the
two conditions
\ref{assum:alpha-phi:gen}-\ref{assum:driftCond:W:new:gen:od} or
\ref{assum:alpha-phi:gen}-\ref{assum:driftCond:W:new:gen:od:new}
trivially holds.

We have thus checked all the assumptions of
\autoref{thm:ergodicity:gen} for the VLODM of \autoref{exmpl:parx} and
this concludes the proof of
Assertion~\ref{item:thm:parx:ergo:gen:od}. 
\end{proof}
\begin{proof}[Proof of~\ref{item:thm:parx:equiv:conv:mle:gen:od}]
  We apply \autoref{thm:convergence-main:gen:od}.  We have already shown
  that~\ref{assum:bound:rho:gen},~\ref{assum:gen:identif:unique:pi:gen}
  and~\ref{ass:21-lyapunov:gen:od} hold in the proof of
  Assertion~\ref{item:thm:logpois:ergodic:gen:od}, with
  $\VX(x,\xi)=x+V_\X(\xi)$.
  Using~(\ref{eq:g:parx}) and the assumption~\ref{assum:thm:parx:equiv:conv:mle:gen:od:L}\ref{assum:thm:parx:equiv:conv:mle:gen:od-1}, we
  get~\ref{assum:continuity:Y:g-theta:gen:od} (here $g$ does not
  depend on
  $\theta$). Assumption~\ref{assum:continuity:Y:phi-theta:gen:od}
  always holds for a VLODM.

  To conclude the proof of~\ref{item:thm:parx:equiv:conv:mle:gen:od},
  it thus remains to check~\ref{assum:practical-cond:O:gen:od}.
  Since $\Theta$ is compact, we can find $\underline{\omega}>0$ such that $\omega\geq\underline{\omega}$ for all
  $\theta=(\omega,\chunk{a}{1}{p},\chunk{b}{1}{q},\chunk{\gamma}{1}{d})\in\Theta$. Then
  we set $\bar\Xset_1=[\underline{\omega},\infty)\times\rset^r$ which
  is a closed subset of $\bar\Xset$ and for which
  \ref{assum:practical-cond:O:gen:od}\ref{assum:exist:practical-cond:subX:gen:od}
  holds (with $\Yset$ replaced by $\bar\Yset$ and $\Xset_1$ by
  $\bar\Xset_1$ since we deal with \autoref{exmpl:parx}).
  \ref{assum:practical-cond:O:gen:od}\ref{assum:momentCond:gen:od}
  holds as an immediate consequence of~(\ref{eq:g:parx})
  and~\ref{assum:thm:parx:equiv:conv:mle:gen:od:L}\ref{assum:thm:parx:equiv:conv:mle:gen:od-2}. In
  the VLODM case, to meet
  \ref{assum:practical-cond:O:gen:od}\ref{item:barphi:bar:phi:gen:od},
  it suffices to have
  $$
  \bar\phi(y,\xi)\geq c_0 (1+y+\|\xi\|)\;,\quad y\in\zsetp\,,\;\xi\in\rset^r\;,
  $$
  where $c_0$ is a positive constant only depending on $\initmle$. In
  the following, we choose
  \begin{equation}
    \label{eq:parx:barphi-chosen}
    \bar\phi(y,\xi):= \left(1\vee c_0\right) \,\left(1+\left(1\vee\underline{\omega}^{-1}\right)\,y+\bar\phi_{\X}(\xi)\right)\;,\quad y\in\zsetp\,,\;\xi\in\rset^r\;,
  \end{equation}
  which,
  with~\ref{assum:thm:parx:equiv:conv:mle:gen:od:L}\ref{assum:thm:parx:barphi:bar:phi:gen:od},
  guaranties that the previous inequality holds, and thus
  \ref{assum:practical-cond:O:gen:od}\ref{item:barphi:bar:phi:gen:od}
  is verified.

  Let $\bar x,\bar x'\in\bar\Xset_1$ and $\bar y\in \bar\Yset$, that is $\bar x=(x,\xi)$ and
  $\bar x'=(x',\xi')$ with $x,x'\in[\underline{\omega},\infty)$ and $\xi,\xi'\in\rset^r$,, and
  $\bar y=(y,\xi'')$ with $y\in\zsetp$ and $\xi''\in\rset^r$,
  we have, by~(\ref{eq:g:parx}),
  $$
  \left|\ln\frac{g(\bar x;\bar y)}{g(\bar x';\bar y)}\right|\leq|x-x'|+y\left|\ln\frac{x}{x'}\right|+  \left|\ln\frac{\ell(\xi;\xi'')}{\ell(\xi';\xi'')}\right|
  $$
  By~\ref{assum:thm:parx:equiv:conv:mle:gen:od:L}\ref{assum:thm:parx:ineq:loggg:gen:od}
  and using that $|\ln(x/x')|\leq\underline{\omega}^{-1}|x-x'|$, we
  obtain that
  $$
    \left|\ln\frac{g(\bar x;\bar y)}{g(\bar x';\bar y)}\right|\leq
  |x-x'|\,(1+\underline{\omega}^{-1}y) + \mathrm{h}_\X(\|\xi-\xi'\|) \,
\rme^{C_\X \, \left(1+\|\xi\|\vee\|\xi'\right)}\;\bar\phi_\X(\xi'')\;.
  $$
Thus  \ref{assum:practical-cond:O:gen:od}\ref{item:ineq:loggg:gen:od}
holds with $\bar\phi$ defined as in~(\ref{eq:parx:barphi-chosen}) and
for some well chosen $c_1>0$ only depending on $\initmle$
\begin{equation}
  \label{eq:parx:bar-h-def}
  \mathrm{h}(u)=c_1\,(u+\mathrm{h}_\X(u))\quad\text{and}\quad C=C_\X
\end{equation}
Obviously  \ref{assum:practical-cond:O:gen:od}\ref{item:H:gen:od} is
then satisfied. We now observe that, with $G$
given by \autoref{exmpl:parx}~\ref{item:exmpl:parx:G} and $\bar\phi$
given by~(\ref{eq:parx:barphi-chosen}), for all
$x\in\rsetp$ and $\xi\in\rset^r$, we have
\begin{align*}
G[\lnp\bar{\phi}](x,\xi)&\leq \ln\left(1\vee
                          c_0\right)+\left(1\vee\underline{\omega}^{-1}\right)\,x+L[\lnp\bar\phi_\X](\xi)\;,\\
  G\bar{\phi}(x,\xi)&\leq \left(1\vee c_0\right) \,\left(1+\left(1\vee\underline{\omega}^{-1}\right)\,x+L\bar\phi_{\X}(\xi)\right)\;,
\end{align*}
where, in the first inequality, we used that
$\lnp\left(1+\left(1\vee\underline{\omega}^{-1}\right)\,y+\bar\phi_{\X}(\xi)\right)\leq\left(1\vee\underline{\omega}^{-1}\right)\,y+\lnp(\bar\phi_{\X}(\xi))$.
By \ref{item:assum:parx:v}\ref{item:assum:parx:v:d} we have
$LV_\X\lesssim V_\X$ and thus
\ref{assum:thm:parx:equiv:conv:mle:gen:od:L}\ref{assum:thm:parx:22or23-barphi-V:gen:od}
implies $L[\lnp\bar\phi_\X]\lesssim LV_\X\lesssim V_\X$ if $C_\X=0$ and
$L\bar\phi_{\X}\lesssim LV_\X\lesssim V_\X$ otherwise.
The previous display
with~\ref{assum:thm:parx:equiv:conv:mle:gen:od:L}\ref{assum:thm:parx:22or23-barphi-V:gen:od} yields
\ref{assum:practical-cond:O:gen:od}\ref{item:22or23-barphi-V:gen:od}
with $C=C_\X$. (Recall
that  $\VX(x,\xi)=x+V_\X(\xi)$).
\end{proof}
\begin{proof}[Proof of~\ref{item:thm:parx:conv:gen:od}]
  To obtain~\ref{item:thm:parx:conv:gen:od} from
  Assertion~\ref{item:thm:parx:equiv:conv:mle:gen:od} we only need to
  have that $[\thv]$ reduced to the singleton $\{\thv\}$ for all
  $\thv\in\Theta$. To prove this, we use
  \cite[Section~5.5]{douc-roueff-sim_ident-genod2020} (or
  \cite[Thorem~18]{douc_roueff_sim_supplement20}), where it is proved
  under Assumption~\ref{item:PARX--non-degenerateident-G-X} and other
  assumptions named {\sf(A'-1), (L-1)} and {\sf(L-3)} which, in our
  setting here, respectively correspond
  to~\ref{assum:gen:identif:unique:pi:gen},~\ref{item:lodm-ident-cond-cond-dens-def}
  and to having, for all $\theta\in\Theta$,
  \begin{equation}
    \label{eq:ident-moment-cond:parx}
    \PE^\theta\left[\lnp\left(Y_0+\sum_{k=1}^df_k(\Xi_0)+\|\Xi_0\|\right)\right]<\infty\;.
  \end{equation}
  We have already
  shown in  Assertion~\ref{item:thm:parx:ergo:gen:od} that
 \ref{assum:gen:identif:unique:pi:gen}
  and \ref{item:lodm-ident-cond-cond-dens-def} hold, and
  also~\ref{ass:21-lyapunov:gen:od} with $\VX(x,\xi)=x+V_\X(\xi)$,
  which yield that,
  for all $\theta\in\Theta$,
  $\PE^\theta\left[X_0+V_\X(\Xi_{-1})\right]<\infty$, and thus
  \begin{equation}
    \label{eq:ident-moment-cond:parx:almost-there}
    \PE^\theta\left[Y_0+V_\X(\Xi_0)\right]<\infty\;.
  \end{equation}
  Using that, for all
  $y\in\zsetp$ and $\xi\in\rset^r$,
  \begin{align*}
    \lnp\left(y+\sum_{k=1}^df_k(\xi)+\|\xi\|\right)&\leq
                                                     \ln\left(1+\|\xi\|+\sum_{k=1}^df_k(\xi)+y\right)\\
                                                   &\leq
                                                     \sum_{k=1}^df_k(\xi)+
                                                     y+  \ln\left(1+\|\xi\|\right)\;,    
  \end{align*}
  we get that~(\ref{eq:ident-moment-cond:parx}) follows
  from~(\ref{eq:ident-moment-cond:parx:almost-there}),~\ref{item:assum:parx:v}\ref{item:assum:parx:v:c}
  and the fact that
  $\ln\left(1+\|\cdot\|\right)\lesssim\lnp\bar\phi_\X\lesssim V_\X$ as
  a consequence of
  \ref{assum:thm:parx:equiv:conv:mle:gen:od:L}\ref{assum:thm:parx:barphi:bar:phi:gen:od}
  and~\ref{assum:thm:parx:22or23-barphi-V:gen:od}, respectively. This concludes the
  proof of Assertion~\ref{item:thm:parx:conv:gen:od} of  \autoref{theo:ergo-convergence:parx}.
\end{proof}

\section{Useful Lemmas}
\label{sec:useful-lemmas}
The following result is used in the proof of
\autoref{thm:convergence-main:gen:od}. It is proven in
\cite[Lemma~19]{douc_roueff_sim_supplement20}.
\begin{lemma}\label{lem:Lip:XY-cond:geom}
  \ref{assum:bound:rho:gen} implies that for all $\theta\in\Theta$, there
  exists $C>0$ and $\rho\in(0,1)$ such that   $\mathrm{Lip}_n^\theta \leq C\ \rho^n$ for
  all $n\in\zsetpnz$.
\end{lemma}
A byproduct of \autoref{lem:Lip:XY-cond:geom} is the following result,
which is used in the proof of \autoref{lem:reaching}.
\begin{lemma}\label{lemma_reaching-limit-iterates}
  Suppose that \ref{item:CCMShyp} and \ref{assum:bound:rho:gen}
  hold. Let $\theta\in\Theta$ and $u_0\in\Uset$ and set $u_k=u_0$ for
  all $k\in\zsetp$. Then there exists $x_\infty\in\Xset$ such that for
  all $z\in\zset$, $\tf{\chunk{u}{0}{n}}(z)$ converges to $x_{\infty}$
  as $n\to\infty$ in $(\Xset,\Xmet)$.
\end{lemma}
\begin{proof}
  Let $z,z'$ in $\Zset$. Denote, for all $n\in\zsetpnz$,
  $x_n=\tf{\chunk{u}{0}{n}}(z)$ and
  $\tilde x_n=\f{\chunk{u}{0}{n}}(z')$. Then, for all $n\in\zsetpnz$,
    \begin{align}\label{eq:reaching1}
      \Xmet(x_n,\tilde x_n)&\leq \mathrm{Lip}^\theta_{n}\,   \Zmet\left(z,z'\right)\;,\\
      \label{eq:reaching2}
      \Xmet(x_{n+1},x_n)&\leq \mathrm{Lip}^\theta_{n}\,   \Zmet\left(z,\tilde{\Psi}^\theta_{u_0}(z)\right) \;,
    \end{align}
    where $\tilde{\Psi}$ is defined in~(\ref{eq:def:tpsi:gen:od}).
    By \autoref{lem:Lip:XY-cond:geom}, the right-hand side
    in~(\ref{eq:reaching2}) is decreasing geometrically fast and $x_n$
    converges to a point $\psi_\infty(y_0)$ which does not depend on $z$
    by~(\ref{eq:reaching1}).
\end{proof}
The following lemma is used in the proof of
 \autoref{theo:ergo-convergence:logpois:gen}.
 \begin{lemma}\label{lem:support:bound:moment:pois}
   Let $\vartheta\in\rset$. Then, for all $\vartheta_0\in\rset$ and $\zeta\in\rset$, if
   $U\sim\mathcal{P}(\rme^{\zeta})$, then
\begin{align}\label{eq:lem:bound:moment:pois:gen1}
  &  \PE[(1+U)^\vartheta]\leq \rme^{(1+\zeta_+)\,\vartheta_+}\\
  \label{eq:lem:bound:moment:pois:gen2}
 & \PE[\rme^{\left|\vartheta_0+\vartheta\ln(1+U)\right|}]\leq
   2\rme^{\vartheta_+}\,\rme^{|\vartheta_0|\vee|\vartheta_0+\vartheta\,\zeta_+|}\\
    \label{eq:lem:bound:moment:pois:gen3}
&\phantom{\PE[\rme^{\left|\vartheta_0+\vartheta\ln(1+U)\right|}]}\leq
   2\rme^{\vartheta_+}\,\rme^{|\vartheta_0|\vee|\vartheta_0+\vartheta\,\zeta'_+|}\quad\text{for all}\quad\zeta'\geq\zeta \;.
\end{align}
 \end{lemma}
 \begin{proof}
   We separate the proof of~(\ref{eq:lem:bound:moment:pois:gen1}) in three different cases by specifying the
   bound~(\ref{eq:lem:bound:moment:pois:gen1}) in each case.
   \begin{enumerate}[label=Case~\arabic*]
   \item\label{item:case1:lem:support:bound:moment:pois} For all $\vartheta\leq0$, we have
     $\PE[(1+U)^\vartheta]\leq1$.
   \item\label{item:case2:lem:support:bound:moment:pois} For all $\vartheta>0$ and $\zeta<0$, we have
     $\PE[(1+U)^\vartheta]\leq \rme^{\vartheta}$.
   \item\label{item:case3:lem:support:bound:moment:pois} For all $\vartheta>0$ and $\zeta\geq0$, we have
     $\PE[(1+U)^\vartheta]\leq \rme^{\vartheta}\,\rme^{\zeta\vartheta}$.
   \end{enumerate}
   The bound in~\ref{item:case1:lem:support:bound:moment:pois} is obvious.
   The bound in~\ref{item:case2:lem:support:bound:moment:pois} follows from
   $$
   \PE[(1+U)^\vartheta]\leq    \PE\left[\rme^{\vartheta U}\right] =
   \rme^{\rme^{\zeta}(\vartheta-1)}\;.
   $$
   Finally, the bound in~\ref{item:case3:lem:support:bound:moment:pois} follows
   from the following inequalities, valid for all
  $\vartheta>0$ and $\zeta\geq0$,
$$
   \PE\left[\rme^{-\zeta\vartheta}(1+U)^\vartheta\right]\leq
   \PE\left[\left(1+\rme^{-\zeta}U\right)^{\vartheta}\right]
\leq\PE\left[\rme^{\vartheta\rme^{-\zeta}U}\right]=\rme^{\rme^\zeta(\vartheta\rme^{-\zeta}-1)}\leq\rme^{\vartheta}\;.
$$
Hence we get~(\ref{eq:lem:bound:moment:pois:gen1}).

Let us now prove~(\ref{eq:lem:bound:moment:pois:gen2}). Observe that
\begin{align*}
  \PE[\rme^{\left|\vartheta_0+\vartheta\ln(1+U)\right|}]&\leq
                                                          \PE[\rme^{\vartheta_0+\vartheta\ln(1+U)}] +
                                                          \PE[\rme^{-\vartheta_0-\vartheta\ln(1+U)}]
  \\
                                                        &
                                                          =\rme^{\vartheta_0}\PE[(1+U)^\vartheta]+\rme^{-\vartheta_0}\PE[(1+U)^{-\vartheta}]\;.
  \end{align*}
  Then using~(\ref{eq:lem:bound:moment:pois:gen1}), we get
  $$
  \PE[\rme^{\left|\vartheta_0+\vartheta\ln(1+U)\right|}]\leq 2\rme^{\vartheta_+}\,
  \exp\left[(\vartheta_0+(\zeta_+\vartheta)_+)\vee(-\vartheta_0+(\zeta_+\vartheta)_-)\right]
  $$
  We conclude~(\ref{eq:lem:bound:moment:pois:gen2})
  and~(\ref{eq:lem:bound:moment:pois:gen3}) by observing that, for all
  $a,b\in\rset$,
  $(a+b_+)\vee(-a+b_-)=a\vee(a+b)\vee (-a)\vee (-a-b)=|a|\vee|a+b|$.
\end{proof}
The two following lemmas are straightforward and their proofs are thus
omitted. They are used in the proofs of
\autoref{theo:ergo-convergence:garch:gen:nbin} and
\autoref{theo:ergo-convergence:parx}, respectively. 
\begin{lemma}
  \label{lem:affine-rec}
  Let $r$ be a positive integer and $(\omega,\chunk c1r)\in\rset^{1+r}$. Let
  $\nsequence{x_k}{k \in\zsetp}$ be a sequence satisfying 
  $$
  x_n=\omega+\sum_{k=1}^{r}c_kx_{n-k}\;,\quad n\geq r\;,
  $$
  Suppose that
    the polynomial $P(z)=1-\sum_{k=1}^{r} c_k z^k$ has
    no roots in $\{z\in{\mathbb C}:\ |z|\leq 1\}$. 
  Then there
  exist $\rho<1$ and $C>0$ such that, for all $n\in\zsetp$,
  $$
  \left|x_n-\omega/P(1)\right|\leq  C \,\rho^n\,\left(1+\max(|x_0|,\dots,|x_{r-1}|)\right)\;.
  $$
\end{lemma}
\begin{lemma}
\label{lem:assum:parx:i}
For any $a_1,\dots,a_p,b_1,\dots,b_q\in\rsetp$,
Condition~(\ref{eq:cond:ergo-convergence:parx:one}) implies the first and is
equivalent to the second following assertions.
\begin{enumerate}[(i)]
\item\label{rem:assum:parx:i-1}
 The polynomial $z\mapsto 1-\sum_{i=1}^{p} a_i z^i$ has no
     roots in $\{z\in{\mathbb C}:\ |z|\leq 1\}$.
\item\label{rem:assum:parx:i-2} The polynomial $z\mapsto 1-\sum_{i=1}^{p\vee q}c_i z^i$ has no
  roots in $\{z\in{\mathbb C}:\ |z|\leq 1\}$ where $c_i=a_i\1(i\leq p)+b_i \1(i\leq q)\in\rsetp$.
\end{enumerate}
\end{lemma}

\end{document}